\documentclass[a4paper,11pt,reqno]{amsart}
\pdfoutput=1
\usepackage{xcolor}
\usepackage{multicol}
\usepackage{fullpage}
\usepackage{bbm,color}
\usepackage{graphics,array}
\usepackage{tikz}
\usetikzlibrary{matrix,arrows,snakes,patterns,cd,shapes,graphs,graphs.standard}
\usetikzlibrary{decorations.markings}

\usepackage{pgfplots}
\usetikzlibrary{intersections, pgfplots.fillbetween}
\usepackage{amsmath,amsthm,mathtools,hyperref}
\usepackage{amsfonts}
\usepackage{amssymb}
\usepackage{shuffle}

\definecolor{light-gray}{gray}{0.75}
\definecolor{ww}{gray}{0.9}
\definecolor{gg}{gray}{0.5}
\definecolor{bb}{gray}{0.1}
\definecolor{darkgreen}{rgb}{0.0,0.7,0.0}

\makeatletter
\tikzset{nomorepostaction/.code={\let\tikz@postactions\pgfutil@empty}} 
\makeatother

\newcommand{\iscomp}{\vDash}
\newcommand{\isweakcomp}{\vDash^{*}}

\newtheorem{theorem}{Theorem}[section]
 \newtheorem{corollary}[theorem]{Corollary}
 \newtheorem{prop}[theorem]{Proposition}
 \newtheorem{conjecture}[theorem]{Conjecture}
 \newtheorem{lemma}[theorem]{Lemma}
 \newtheorem{proposition}[theorem]{Proposition}

 \theoremstyle{definition}
 \newtheorem{example}[theorem]{Example}

 \newtheorem{definition}[theorem]{Definition}

\newcommand{\red}[1]{#1}

\definecolor{myblue}{RGB}{80,80,160}
\definecolor{mygreen}{RGB}{80,160,80}

\newcommand\ClassOneParams[2]{\DecRec(\Split_{#1,#2})}

\newcommand{\GETOUT}[1]{}

\newcommand{\thresh}{\mathrm{deg}}

\newcommand{\proofsuppress}[1]{{\red{Temporarily commented out for readability. proofsuppress switch.}}}

\newcommand{\fcp}{\mathsf{f}}
\newcommand{\Words}{\mathsf{Words}}
\newcommand{\Hex}{\mathsf{Hex}}

\newcommand{\itcbounce}{\mathsf{itcbounce}}
\newcommand{\schbounce}{\mathsf{bounce}^{\mathsf{Sch}}}
\newcommand{\itcBounce}{\mathsf{itcBounce}}
\newcommand{\ctibounce}{\mathsf{ctibounce}}
\newcommand{\ctiBounce}{\mathsf{ctiBounce}}
\newcommand{\bouncepath}{\mathsf{topple}}
\newcommand{\topplebounce}{\mathsf{wtopple}}

\newcommand{\area}{\mathsf{area}}
\newcommand{\smee}[1]{\tiny{#1}}

\newcommand{\Schroder}{\mathsf{Schr\ddot{o}der}}
\newcommand{\ToppleRecord}{\mathsf{Topple}}
\newcommand{\height}{\mathsf{height}}
\newcommand{\wonkypoly}{\mathsf{Sawtooth}}
\newcommand{\wonky}{sawtooth polyomino}
\newcommand{\wonkys}{sawtooth polyominoes}
\newcommand{\ee}{\mathsf{e}}
\newcommand{\ww}{\mathsf{w}}
\newcommand{\nn}{\mathsf{n}}
\renewcommand{\ss}{\mathsf{s}}
\newcommand{\se}{\mathsf{se}} 

\newcommand{\nw}{\mathsf{nw}}
\newcommand{\Upper}{\mathsf{Upper}}
\newcommand{\Lower}{\mathsf{Lower}}
\newcommand{\sts}{{\varphi}} 
\newcommand{\stc}{{\phi}} 
\newcommand*{\doi}[1]{\href{http://dx.doi.org/#1}{doi: #1}}
\newcommand{\level}{\mathsf{level}}
\newcommand{\Peakloop}{\mathsf{Peak}^{\mathsf{loop}}}
\newcommand{\PeakDyck}{\mathsf{Peak}^{\mathsf{Dyck}}}
\newcommand{\PeakSchroder}{\mathsf{Peak}^{\mathsf{Sch}}}
\newcommand{\itc}{\mathrm{itc}}
\newcommand{\ITC}{\mathsf{ITC}}
\newcommand{\ehkk}{\mathrm{ehkk}}

\newcommand{\unlabelledsplitgraph}{
\begin{tikzpicture}[
my node style/.style={circle,fill,draw,inner sep=1.5pt}
]
  \node[my node style]    (v0) at (0,1) {};
  \node[my node style]  (v1) at (.866,.5)  {};
  \node[my node style]  (v4) at (-.866,.5)  {};
  \node[my node style]  (v2) at (.866,-.5)  {};
  \node[my node style]  (v3) at (-.866,-.5)  {};
  \node[my node style]  (v5) at (0,-1)  {};
  \node[my node style]    (v6) at (4,1) {};
  \node[my node style]    (v7) at (4,0) {};
  \node[my node style]    (v8) at (4,-1) {};
  \foreach \i/\j in {0/1,0/2,0/3,0/4,0/5,1/2,1/3,1/4,1/5,2/3,2/4,2/5,3/4,3/5,4/5}
    \path (v\i) edge (v\j);
  \path (v0) edge [bend right=0] (v6);
  \path (v0) edge [bend right=0] (v7);
  \path (v0) edge [bend left=10] (v8);
  \path (v1) edge [bend right=0] (v6);
  \path (v1) edge [bend right=0] (v7);
  \path (v1) edge [bend left=0] (v8);
  \path (v2) edge [bend right=0] (v6);
  \path (v2) edge [bend right=0] (v7);
  \path (v2) edge [bend left=0] (v8);
  \path (v3) edge [bend right=0] (v6);
  \path (v3) edge [bend left=5] (v7);
  \path (v3) edge [bend right=5] (v8);
  \path (v4) edge [bend left=2] (v6);
  \path (v4) edge [bend right=7] (v7);
  \path (v4) edge [bend right=0] (v8);
  \path (v5) edge [bend right=10] (v6);
  \path (v5) edge [bend right=0] (v7);
  \path (v5) edge [bend right=0] (v8);
  \draw[rounded corners, dashed] (-0.25,0.75) rectangle (0.25,1.25);
\end{tikzpicture}
}

\def\generalfill{\tikzfillbetween[of=A and B]{light-gray,opacity=0.5}}

\newcommand{\newpolyone}{
\begin{tikzpicture}[scale=0.55, line width=1.0pt]
\draw[rounded corners=2pt,name path=A] (0.025,-0.025) -- ++(2,0) -- ++(0,1) -- ++(2,0) -- ++(0,2) -- ++(1,0) -- ++(0,2);
\draw[rounded corners=2pt,name path=B] (-0.025,0.025) -- ++(0,3) -- ++(1,-1) -- ++(0,1) -- ++(1,-1) -- ++(0,2) -- ++(1,-1) -- ++(0,3) -- ++(1,-1) -- ++(0,1) -- ++ (1,-1);
\generalfill;
\draw[step=1cm,very thin] (-0.2,-0.2) grid (5.2,7.2);
\end{tikzpicture}
}

\newcommand{\newpolyonebouncealt}{
\begin{tikzpicture}[scale=0.55,
line width=1.0pt
]
\draw[rounded corners=2pt,name path=A] (0.025,-0.025) -- ++(2,0) -- ++(0,1) -- ++(2,0) -- ++(0,2) -- ++(1,0) -- ++(0,2);
\draw[rounded corners=2pt,name path=B] (-0.025,0.025) -- ++(0,3) -- ++(1,-1) -- ++(0,1) -- ++(1,-1) -- ++(0,2) -- ++(1,-1) -- ++(0,3) -- ++(1,-1) -- ++(0,1) -- ++(1,-1);
\generalfill;
\draw[rounded corners=1pt,red,opacity=1,line width=1.5pt ] (4,5) -- ++(0,-2) -- ++(-1,1) -- ++(0,-3) -- ++(-1,1) -- ++(0,-1) -- ++(-1,1) -- ++(0,-2) -- ++(-1,1) -- ++(0,-1);
\draw[red,opacity=1,line width=1.5pt] (3.8,4)--(4,3.8)--(4.2,4);
\draw[step=1cm,very thin] (-0.2,-0.2) grid (5.2,7.2);
\end{tikzpicture}
}

\newcommand{\newpolyonebouncealttwo}{
\begin{tikzpicture}[scale=0.55,
line width=1.0pt
]
\draw[rounded corners=2pt,name path=A] (0.025,-0.025) -- ++(2,0) -- ++(0,1) -- ++(2,0) -- ++(0,2) -- ++(1,0) -- ++(0,2);
\draw[rounded corners=2pt,name path=B] (-0.025,0.025) -- ++(0,3) -- ++(1,-1) -- ++(0,1) -- ++(1,-1) -- ++(0,2) -- ++(1,-1) -- ++(0,3) -- ++(1,-1) -- ++(0,1) -- ++(1,-1);
\generalfill;
\draw[rounded corners=1pt,darkgreen,opacity=1,line width=1.5pt ] (4,5) -- ++(0,-2) -- ++(-1,1) -- ++(0,-3) -- ++(-1,1) -- ++(0,-1) -- ++(-1,1) -- ++(0,-2) -- ++(-1,1) -- ++(0,-1);
\draw[darkgreen,opacity=1,line width=1.5pt] (3.8,4)--(4,3.8)--(4.2,4);
\draw[step=1cm,very thin] (-0.2,-0.2) grid (5.2,7.2);
\end{tikzpicture}
}

\newcommand{\newpolytwo}{
\begin{tikzpicture}[scale=0.55,
line width=1.0pt
]
\draw[rounded corners=2pt,name path=A] (0.025,-0.025) -- ++(2,0) -- ++(0,1) -- ++(1,0) -- ++(0,1) -- ++(1,0) -- ++(0,1) -- ++(1,0) -- +(0,2);
\draw[rounded corners=2pt,name path=B] (-0.025,0.025) -- ++(0,5) -- ++(2,-2) -- ++(0,1) -- ++(1,-1) -- ++(0,4) -- ++(1,-1) -- ++(1,-1);
\generalfill;
\draw[step=1cm,black,very thin] (-0.2,-0.2) grid (5.2,7.2);
\end{tikzpicture}
}

\newcommand{\newpolytwobouncealt}{
\begin{tikzpicture}[scale=0.55,
line width=1.0pt
]
\draw[rounded corners=2pt,name path=A] (0.025,-0.025) -- ++(2,0) -- ++(0,1) -- ++(1,0) -- ++(0,1) -- ++(1,0) -- ++(0,1) -- ++(1,0) -- +(0,2);
\draw[rounded corners=2pt,name path=B] (-0.025,0.025) -- ++(0,5) -- ++(2,-2) -- ++(0,1) -- ++(1,-1) -- ++(0,4) -- ++(1,-1) -- ++(1,-1);
\generalfill;
\draw[rounded corners=1pt,red,opacity=1,line width=1.5pt ] (4,5-0.025) -- ++(-1,1) -- ++(0,-4) -- ++(-1,1) -- ++(0,-2) -- ++(-2,2) -- ++(0,-3);
\draw[red,opacity=1,line width=1.5pt] (3.5,5.2)--(3.5,5.5)--(3.8,5.5);
\draw[step=1cm,gray,very thin] (-0.2,-0.2) grid (5.2,7.2);
\end{tikzpicture}
}
\newcommand{\newpolytwobouncealttwo}{
\begin{tikzpicture}[scale=0.55,
line width=1.0pt
]
\draw[rounded corners=2pt,name path=A] (0.025,-0.025) -- ++(2,0) -- ++(0,1) -- ++(1,0) -- ++(0,1) -- ++(1,0) -- ++(0,1) -- ++(1,0) -- +(0,2);
\draw[rounded corners=2pt,name path=B] (-0.025,0.025) -- ++(0,5) -- ++(2,-2) -- ++(0,1) -- ++(1,-1) -- ++(0,4) -- ++(1,-1) -- ++(1,-1);
\generalfill;
\draw[rounded corners=1pt,darkgreen,opacity=1,line width=1.5pt ] (4,5-0.025) -- ++(0,-2) -- ++(-1,+1) -- ++(0,-2) -- ++(-1,1) -- ++(0,-2) -- ++(-2,2) -- ++(0,-3);
\draw[darkgreen,opacity=1,line width=1.5pt] (3.8,4)--(4,3.8)--(4.2,4);
\draw[step=1cm,very thin] (-0.2,-0.2) grid (5.2,7.2);
\end{tikzpicture}
}

\def\mysk{0.35}
\def\qwe{-0.5}
\def\abh{6.5}
\def\ach{5.7}
\def\commonframe{\draw[step=1cm,very thin] (-0.2,-0.2) grid (3.2,5.2)}
\def\redop{0.8}

\newcommand{\hreeone}{
\begin{tikzpicture}[scale=\mysk, line width=2.0pt ]
\draw[rounded corners=2pt,name path=A] (0.025,-0.025) -- ++(3,0) -- ++(0,2);
\draw[rounded corners=2pt,name path=B] (-0.025,0.025) -- ++(0,5) -- ++(3,-3);
\generalfill;
\draw[rounded corners=1pt,red,opacity=\redop ] (2,2-0.05) -- ++(-2,2) -- ++(0,-4);
\node[anchor=west] at (\qwe,\abh) {\smee{$(3{,}3{;}2{,}2)$}};
\node[anchor=west] at (\qwe,\ach) {\smee{$\langle 2{,}2 \rangle$}};
\commonframe;
\end{tikzpicture}
}

\newcommand{\hreetwo}{
\begin{tikzpicture}[scale=\mysk, line width=2.0pt ]
\draw[rounded corners=2pt,name path=A] (0.025,-0.025) -- ++(2,0) -- ++(0,1) -- ++(1,0) -- ++(0,1);
\draw[rounded corners=2pt,name path=B] (-0.025,0.025) -- ++(0,5) -- ++(3,-3);
\generalfill;
\draw[rounded corners=1pt,red,opacity=\redop ] (2,2-0.05) -- ++(-2,2) -- ++(0,-4);
\node[anchor=west] at (\qwe,\abh) {\smee{$(3{,}3{;}2{,}1)$}};
\node[anchor=west] at (\qwe,\ach) {\smee{$\langle 2{,}2 \rangle$}};
\commonframe;
\end{tikzpicture}
}

\newcommand{\hreethree}{
\begin{tikzpicture}[scale=\mysk, line width=2.0pt ]
\draw[rounded corners=2pt,name path=A] (0.025,-0.025) -- ++(1,0) -- ++(0,1) -- ++(2,0) -- ++(0,1);
\draw[rounded corners=2pt,name path=B] (-0.025,0.025) -- ++(0,5) -- ++(3,-3);
\generalfill;
\draw[rounded corners=1pt,red,opacity=\redop ] (2,2-0.05) -- ++(-2,2) -- ++(0,-4);
\node[anchor=west] at (\qwe,\abh) {\smee{$(3{,}3{;}2{,}0)$}};
\node[anchor=west] at (\qwe,\ach) {\smee{$\langle 2{,}2 \rangle$}};
\commonframe;
\end{tikzpicture}
}

\newcommand{\hreefour}{
\begin{tikzpicture}[scale=\mysk, line width=2.0pt ]
\draw[rounded corners=2pt,name path=A] (0.025,-0.025) -- ++(2,0) -- ++(0,2) -- ++(1,0);
\draw[rounded corners=2pt,name path=B] (-0.025,0.025) -- ++(0,5) -- ++(3,-3);
\generalfill;
\draw[rounded corners=1pt,red,opacity=\redop ] (2,2-0.05) -- ++(-2,2) -- ++(0,-4);
\node[anchor=west] at (\qwe,\abh) {\smee{$(3{,}3{;}1{,}1)$}};
\node[anchor=west] at (\qwe,\ach) {\smee{$\langle 2{,}2 \rangle$}};
\commonframe;
\end{tikzpicture}
}

\newcommand{\hreefive}{
\begin{tikzpicture}[scale=\mysk, line width=2.0pt ]
\draw[rounded corners=2pt,name path=A] (0.025,-0.025) --++(1,0)  -- ++(0,1) -- ++(1,0) -- ++(0,1) -- ++(1,0);
\draw[rounded corners=2pt,name path=B] (-0.025,0.025) -- ++(0,5) -- ++(3,-3);
\generalfill;
\draw[rounded corners=1pt,red,opacity=\redop ] (2,2-0.05) -- ++(-2,2) -- ++(0,-4);
\node[anchor=west] at (\qwe,\abh) {\smee{$(3{,}3{;}1{,}0)$}};
\node[anchor=west] at (\qwe,\ach) {\smee{$\langle 2{,}2 \rangle$}};
\commonframe;
\end{tikzpicture}
}

\newcommand{\hreesix}{
\begin{tikzpicture}[scale=\mysk, line width=2.0pt ]
\draw[rounded corners=2pt,name path=A] (0.025,-0.025) -- ++(1,0) -- ++(0,2) -- ++(2,0);
\draw[rounded corners=2pt,name path=B] (-0.025,0.025) -- ++(0,5) -- ++(3,-3);
\generalfill;
\draw[rounded corners=1pt,red,opacity=\redop ] (2,2-0.05) -- ++(-2,2) -- ++(0,-4);
\node[anchor=west] at (\qwe,\abh) {\smee{$(3{,}3{;}0{,}0)$}};
\node[anchor=west] at (\qwe,\ach) {\smee{$\langle 2{,}2 \rangle$}};
\commonframe;
\end{tikzpicture}
}

\newcommand{\hreeseven}{
\begin{tikzpicture}[scale=\mysk, line width=2.0pt ]
\draw[rounded corners=2pt,name path=A] (0.025,-0.025) -- ++(1,0) -- ++(2,0) -- ++(0,2);
\draw[rounded corners=2pt,name path=B] (-0.025,0.025) -- ++(0,1) -- ++(0,3) -- ++(1,-1) -- ++(0,1) -- ++(1,-1) -- ++(1,-1);
\generalfill;
\draw[rounded corners=1pt,red,opacity=\redop ] (2,2-0.05) -- ++(-1,1) -- ++(0,-3) -- ++(-1,1) -- ++(0,-1);
\node[anchor=west] at (\qwe,\abh) {\smee{$(3{,}2{;}2{,}2)$}};
\node[anchor=west] at (\qwe,\ach) {\smee{$\langle 1{,}2{,}1{,}0 \rangle$}};
\commonframe;
\end{tikzpicture}
}

\newcommand{\hreeeight}{
\begin{tikzpicture}[scale=\mysk, line width=2.0pt ]
\draw[rounded corners=2pt,name path=A] (0.025,-0.025) -- ++(1,0) -- ++(1,0) -- ++(0,1) -- ++(1,0) -- ++(0,1);
\draw[rounded corners=2pt,name path=B] (-0.025,0.025) -- ++(0,4) -- ++(1,-1) -- ++(0,1) -- ++(2,-2);
\generalfill;
\draw[rounded corners=1pt,red,opacity=\redop ] (2,2-0.05) -- ++(-1,1) -- ++(0,-3) -- ++(-1,1) -- ++(0,-1);
\node[anchor=west] at (\qwe,\abh) {\smee{$(3{,}2{;}2{,}1)$}};
\node[anchor=west] at (\qwe,\ach) {\smee{$\langle 1{,}2{,}1{,}0 \rangle$}};
\commonframe;
\end{tikzpicture}
}

\newcommand{\hreenine}{
\begin{tikzpicture}[scale=\mysk, line width=2.0pt ]
\draw[rounded corners=2pt,name path=A] (0.025,-0.025) -- ++(1,0) -- ++(0,1) -- ++(2,0) -- ++(0,1);
\draw[rounded corners=2pt,name path=B] (-0.025,0.025) -- ++(0,4) -- ++(1,-1) -- ++(0,1) -- ++(2,-2);
\generalfill;
\draw[rounded corners=1pt,red,opacity=\redop ] (2,2-0.05) -- ++(-1,1) -- ++(0,-2) -- ++(-1,1) -- ++(0,-2);
\node[anchor=west] at (\qwe,\abh) {\smee{$(3{,}2{;}2{,}0)$}};
\node[anchor=west] at (\qwe,\ach) {\smee{$\langle 1{,}1{,}1{,}1 \rangle$}};
\commonframe;
\end{tikzpicture}
}

\newcommand{\hreeten}{
\begin{tikzpicture}[scale=\mysk, line width=2.0pt ]
\draw[rounded corners=2pt,name path=A] (0.025,-0.025) --++ (1,0) -- ++(1,0) -- ++(0,2) -- ++(1,0);
\draw[rounded corners=2pt,name path=B] (-0.025,0.025) -- ++(0,4) -- ++(1,-1) -- ++(0,1) -- ++(2,-2);
\generalfill;
\draw[rounded corners=1pt,red,opacity=\redop ] (2,2-0.05) -- ++(-1,1) -- ++(0,-3) -- ++(-1,1) -- ++(0,-1);
\node[anchor=west] at (\qwe,\abh) {\smee{$(3{,}2{;}1{,}1)$}};
\node[anchor=west] at (\qwe,\ach) {\smee{$\langle 1{,}2{,}1{,}0 \rangle$}};
\commonframe;
\end{tikzpicture}
}

\newcommand{\hreeeleven}{
\begin{tikzpicture}[scale=\mysk, line width=2.0pt ]
\draw[rounded corners=2pt,name path=A] (0.025,-0.025) -- ++(1,0) -- ++(0,1) -- ++(1,0) -- ++(0,1) -- ++(1,0);
\draw[rounded corners=2pt,name path=B] (-0.025,0.025) -- ++(0,4) -- ++(1,-1) -- ++(0,1) -- ++(2,-2);
\generalfill;
\draw[rounded corners=1pt,red,opacity=\redop ] (2,2-0.05) -- ++(-1,1) -- ++(0,-2) -- ++(-1,1) -- ++(0,-2);
\node[anchor=west] at (\qwe,\abh) {\smee{$(3{,}2{;}1{,}0)$}};
\node[anchor=west] at (\qwe,\ach) {\smee{$\langle 1{,}1{,}1{,}1 \rangle$}};
\commonframe;
\end{tikzpicture}
}

\newcommand{\hreetwelve}{
\begin{tikzpicture}[scale=\mysk, line width=2.0pt ]
\draw[rounded corners=2pt,name path=A] (0.025,-0.025) -- ++(1,0)  -- ++(0,2) -- ++(2,0);
\draw[rounded corners=2pt,name path=B] (-0.025,0.025) -- ++(0,4) -- ++(1,-1) -- ++(0,1) -- ++(2,-2);
\generalfill;
\draw[rounded corners=1pt,red,opacity=\redop ] (2,2-0.05) -- ++(-1,1) -- ++(0,-1) -- ++(-1,1) -- ++(0,-3);
\node[anchor=west] at (\qwe,\abh) {\smee{$(3{,}2{;}0{,}0)$}};
\node[anchor=west] at (\qwe,\ach) {\smee{$\langle 1{,}0{,}1{,}2 \rangle$}};
\commonframe;
\end{tikzpicture}
}

\newcommand{\hreethirteen}{
\begin{tikzpicture}[scale=\mysk, line width=2.0pt ]
\draw[rounded corners=2pt,name path=A] (0.025,-0.025) -- ++(1,0) -- ++(2,0) -- ++(0,2);
\draw[rounded corners=2pt,name path=B] (-0.025,0.025) -- ++(0,3) -- ++(1,-1) -- ++(0,2) -- ++(2,-2);
\generalfill;
\draw[rounded corners=1pt,red,opacity=\redop ] (2,2-0.05) -- ++(-1,1) -- ++(0,-1) -- ++(-1,1) -- ++(0,-3);
\draw[rounded corners=1pt,red,opacity=\redop ] (2,2-0.05) -- ++(-1,1) -- ++(0,-3) -- ++(-1,1) -- ++(0,-1);
\node[anchor=west] at (\qwe,\abh) {\smee{$(3{,}1{;}2{,}2)$}};
\node[anchor=west] at (\qwe,\ach) {\smee{$\langle 1{,}2{,}1{,}0 \rangle$}};
\commonframe;
\end{tikzpicture}
}

\newcommand{\hreefourteen}{
\begin{tikzpicture}[scale=\mysk, line width=2.0pt ]
\draw[rounded corners=2pt,name path=A] (0.025,-0.025) -- ++(2,0) -- ++(0,1) -- ++(1,0) -- ++(0,1);
\draw[rounded corners=2pt,name path=B] (-0.025,0.025) -- ++(0,3) -- ++(1,-1) -- ++(0,2) -- ++(2,-2);
\generalfill;
\draw[rounded corners=1pt,red,opacity=\redop ] (2,2-0.05) -- ++(-1,1) -- ++(0,-3) -- ++(-1,1) -- ++(0,-1);
\node[anchor=west] at (\qwe,\abh) {\smee{$(3{,}1{;}2{,}1)$}};
\node[anchor=west] at (\qwe,\ach) {\smee{$\langle 1{,}2{,}1{,}0 \rangle$}};
\commonframe;
\end{tikzpicture}
}

\newcommand{\hreefifteen}{
\begin{tikzpicture}[scale=\mysk, line width=2.0pt ]
\draw[rounded corners=2pt,name path=A] (0.025,-0.025) -- ++(1,0) -- ++(0,1) -- ++(2,0) -- ++(0,1);
\draw[rounded corners=2pt,name path=B] (-0.025,0.025) -- ++(0,3) -- ++(1,-1) -- ++(0,2) -- ++(2,-2);
\generalfill;
\draw[rounded corners=1pt,red,opacity=\redop ] (2,2-0.05) -- ++(-1,1) -- ++(0,-2) -- ++(-1,1) -- ++(0,-2);
\node[anchor=west] at (\qwe,\abh) {\smee{$(3{,}1{;}2{,}0)$}};
\node[anchor=west] at (\qwe,\ach) {\smee{$\langle 1{,}1{,}1{,}1 \rangle$}};
\commonframe;
\end{tikzpicture}
}

\newcommand{\hreesixteen}{
\begin{tikzpicture}[scale=\mysk, line width=2.0pt ]
\draw[rounded corners=2pt,name path=A] (0.025,-0.025) -- ++(2,0) -- ++(0,2) -- ++(1,0);
\draw[rounded corners=2pt,name path=B] (-0.025,0.025) -- ++(0,3) -- ++(1,-1) -- ++(0,2) -- ++(2,-2);
\generalfill;
\draw[rounded corners=1pt,red,opacity=\redop ] (2,2-0.05) -- ++(-1,1) -- ++(0,-3) -- ++(-1,1) -- ++(0,-1);
\node[anchor=west] at (\qwe,\abh) {\smee{$(3{,}1{;}1{,}1)$}};
\node[anchor=west] at (\qwe,\ach) {\smee{$\langle 1{,}2{,}1{,}0 \rangle$}};
\commonframe;
\end{tikzpicture}
}

\newcommand{\hreeseventeen}{
\begin{tikzpicture}[scale=\mysk, line width=2.0pt ]
\draw[rounded corners=2pt,name path=A] (0.025,-0.025) -- ++(1,0) -- ++(0,1) -- ++(1,0) -- ++(0,1) -- ++(1,0);
\draw[rounded corners=2pt,name path=B] (-0.025,0.025) -- ++(0,3) -- ++(1,-1) -- ++(0,2) -- ++(2,-2);
\generalfill;
\draw[rounded corners=1pt,red,opacity=\redop ] (2,2-0.05) -- ++(-1,1) -- ++(0,-2) -- ++(-1,1) -- ++(0,-2);
\node[anchor=west] at (\qwe,\abh) {\smee{$(3{,}1{;}1{,}0)$}};
\node[anchor=west] at (\qwe,\ach) {\smee{$\langle 1{,}1{,}1{,}1 \rangle$}};
\commonframe;
\end{tikzpicture}
}

\newcommand{\hreeeighteen}{
\begin{tikzpicture}[scale=\mysk, line width=2.0pt ]
\draw[rounded corners=2pt,name path=A] (0.025,-0.025) -- ++(1,0) -- ++(2,0) -- ++(0,2);
\draw[rounded corners=2pt,name path=B] (-0.025,0.025) -- ++(0,2) -- ++(1,-1) -- ++(0,3) -- ++(2,-2);
\generalfill;
\draw[rounded corners=1pt,red,opacity=\redop ] (2,2-0.05) -- ++(-1,1) -- ++(0,-3) -- ++(-1,1) -- ++(0,-1);
\node[anchor=west] at (\qwe,\abh) {\smee{$(3{,}0{;}2{,}2)$}};
\node[anchor=west] at (\qwe,\ach) {\smee{$\langle 1{,}2{,}1{,}0 \rangle$}};
\commonframe;
\end{tikzpicture}
}

\newcommand{\hreenineteen}{
\begin{tikzpicture}[scale=\mysk, line width=2.0pt ]
\draw[rounded corners=2pt,name path=A] (0.025,-0.025) -- ++(2,0) -- ++(0,1) -- ++(1,0) -- ++(0,1);
\draw[rounded corners=2pt,name path=B] (-0.025,0.025) -- ++(0,2) -- ++(1,-1) -- ++(0,3) -- ++(2,-2);
\generalfill;
\draw[rounded corners=1pt,red,opacity=\redop ] (2,2-0.05) -- ++(-1,1) -- ++(0,-3) -- ++(-1,1) -- ++(0,-1);
\node[anchor=west] at (\qwe,\abh) {\smee{$(3{,}0{;}2{,}1)$}};
\node[anchor=west] at (\qwe,\ach) {\smee{$\langle 1{,}2{,}1{,}0 \rangle$}};
\commonframe;
\end{tikzpicture}
}

\newcommand{\hreetwenty}{
\begin{tikzpicture}[scale=\mysk, line width=2.0pt ]
\draw[rounded corners=2pt,name path=A] (0.025,-0.025) -- ++(2,0) -- ++(0,2) -- ++(1,0);
\draw[rounded corners=2pt,name path=B] (-0.025,0.025) -- ++(0,2) -- ++(1,-1) -- ++(0,3) -- ++(2,-2);
\generalfill;
\draw[rounded corners=1pt,red,opacity=\redop ] (2,2-0.05) -- ++(-1,1) -- ++(0,-3) -- ++(-1,1) -- ++(0,-1);
\node[anchor=west] at (\qwe,\abh) {\smee{$(3{,}0{;}1{,}1)$}};
\node[anchor=west] at (\qwe,\ach) {\smee{$\langle 1{,}2{,}1{,}0 \rangle$}};
\commonframe;
\end{tikzpicture}
}

\newcommand{\hreetwentyone}{
\begin{tikzpicture}[scale=\mysk, line width=2.0pt ]
\draw[rounded corners=2pt,name path=A] (0.025,-0.025) -- ++(1,0) -- ++(2,0) -- ++(0,2);
\draw[rounded corners=2pt,name path=B] (-0.025,0.025) -- ++(0,4) -- ++(2,-2) -- ++(0,1) -- ++(1,-1);
\generalfill;
\draw[rounded corners=1pt,red,opacity=\redop ] (2,2-0.05) -- ++(0,-2) -- ++(-2,2) -- ++(0,-2);
\node[anchor=west] at (\qwe,\abh) {\smee{$(2{,}2{;}2{,}2)$}};
\node[anchor=west] at (\qwe,\ach) {\smee{$\langle 0{,}2{,}2{,}0 \rangle$}};
\commonframe;
\end{tikzpicture}
}

\newcommand{\hreetwentytwo}{
\begin{tikzpicture}[scale=\mysk, line width=2.0pt ]
\draw[rounded corners=2pt,name path=A] (0.025,-0.025) -- ++(2,0) -- ++(0,1) -- ++(1,0) -- ++(0,1);
\draw[rounded corners=2pt,name path=B] (-0.025,0.025) -- ++(0,4) -- ++(2,-2) -- ++(0,1) -- ++(1,-1);
\generalfill;
\draw[rounded corners=1pt,red,opacity=\redop ] (2,2-0.05) -- ++(0,-1) -- ++(-2,2) -- ++(0,-3);
\node[anchor=west] at (\qwe,\abh) {\smee{$(2{,}2{;}2{,}1)$}};
\node[anchor=west] at (\qwe,\ach) {\smee{$\langle 0{,}1{,}2{,}1 \rangle$}};
\commonframe;
\end{tikzpicture}
}

\newcommand{\hreetwentythree}{
\begin{tikzpicture}[scale=\mysk, line width=2.0pt ]
\draw[rounded corners=2pt,name path=A] (0.025,-0.025) -- ++(1,0) -- ++(0,1) -- ++(2,0) -- ++(0,1);
\draw[rounded corners=2pt,name path=B] (-0.025,0.025) -- ++(0,4) -- ++(2,-2) -- ++(0,1) -- ++(1,-1);
\generalfill;
\draw[rounded corners=1pt,red,opacity=\redop ] (2,2-0.05) -- ++(0,-1) -- ++(-2,2) -- ++(0,-3);
\node[anchor=west] at (\qwe,\abh) {\smee{$(2{,}2{;}2{,}0)$}};
\node[anchor=west] at (\qwe,\ach) {\smee{$\langle 0{,}1{,}2{,}1 \rangle$}};
\commonframe;
\end{tikzpicture}
}

\newcommand{\hreetwentyfour}{
\begin{tikzpicture}[scale=\mysk, line width=2.0pt ]
\draw[rounded corners=2pt,name path=A] (0.025,-0.025) -- ++(3,0) -- ++(0,2);
\draw[rounded corners=2pt,name path=B] (-0.025,0.025) -- ++(0,3) -- ++(1,-1) -- ++(0,1) -- ++(1,-1) -- ++(0,1) -- ++(1,-1);
\generalfill;
\draw[rounded corners=1pt,red,opacity=\redop ] (2,2-0.05) -- ++(0,-2) -- ++(-2,2) -- ++(0,-2);
\node[anchor=west] at (\qwe,\abh) {\smee{$(2{,}1{;}2{,}2)$}};
\node[anchor=west] at (\qwe,\ach) {\smee{$\langle 0{,}1{,}2{,}1 \rangle$}};
\commonframe;
\end{tikzpicture}
}

\newcommand{\hreetwentyfive}{
\begin{tikzpicture}[scale=\mysk, line width=2.0pt ]
\draw[rounded corners=2pt,name path=A] (0.025,-0.025) -- ++(2,0) -- ++(0,1) -- ++(1,0) -- ++(0,1);
\draw[rounded corners=2pt,name path=B] (-0.025,0.025) -- ++(0,3) -- ++(1,-1) -- ++(0,1) -- ++(1,-1) -- ++(0,1) -- ++(1,-1);
\generalfill;
\draw[rounded corners=1pt,red,opacity=\redop ] (2,2-0.05) -- ++(0,-1) -- ++(-1,1) -- ++(0,-2) -- ++(-1,1) -- ++(0,-1);
\node[anchor=west] at (\qwe,\abh) {\smee{$(2{,}1{;}2{,}1)$}};
\node[anchor=west] at (\qwe,\ach) {\smee{$\langle 0{,}1{,}1{,}1{,}1{,}0 \rangle$}};
\commonframe;
\end{tikzpicture}
}

\newcommand{\hreetwentysix}{
\begin{tikzpicture}[scale=\mysk, line width=2.0pt ]
\draw[rounded corners=2pt,name path=A] (0.025,-0.025) -- ++(1,0) -- ++(0,1) -- ++(2,0) -- ++(0,1);
\draw[rounded corners=2pt,name path=B] (-0.025,0.025) -- ++(0,3) -- ++(1,-1) -- ++(0,1) -- ++(1,-1) -- ++(0,1) -- ++(1,-1);
\generalfill;
\draw[rounded corners=1pt,red,opacity=\redop ] (2,2-0.05) -- ++(0,-1) -- ++(-1,1) -- ++(0,-1) -- ++(-1,1) -- ++(0,-2);
\node[anchor=west] at (\qwe,\abh) {\smee{$(2{,}1{;}2{,}0)$}};
\node[anchor=west] at (\qwe,\ach) {\smee{$\langle 0{,}1{,}1{,}0{,}1{,}1 \rangle$}};
\commonframe;
\end{tikzpicture}
}

\newcommand{\hreetwentyseven}{
\begin{tikzpicture}[scale=\mysk, line width=2.0pt ]
\draw[rounded corners=2pt,name path=A] (0.025,-0.025) -- ++(3,0) -- ++(0,2);
\draw[rounded corners=2pt,name path=B] (-0.025,0.025) -- ++(0,2) -- ++(1,-1) -- ++(0,2) -- ++(1,-1) -- ++(0,1) -- ++(1,-1);
\generalfill;
\draw[rounded corners=1pt,red,opacity=\redop ] (2,2-0.05) -- ++(0,-2) -- ++(-1,1) -- ++(0,-1) -- ++(-1,1) -- ++(0,-1);
\node[anchor=west] at (\qwe,\abh) {\smee{$(2{,}0{;}2{,}2)$}};
\node[anchor=west] at (\qwe,\ach) {\smee{$\langle 0{,}2{,}1{,}0{,}1{,}0 \rangle$}};
\commonframe;
\end{tikzpicture}
}

\newcommand{\hreetwentyeight}{
\begin{tikzpicture}[scale=\mysk, line width=2.0pt ]
\draw[rounded corners=2pt,name path=A] (0.025,-0.025) -- ++(2,0) -- ++(0,1) -- ++(1,0) -- ++(0,1);
\draw[rounded corners=2pt,name path=B] (-0.025,0.025) -- ++(0,2) -- ++(1,-1) -- ++(0,2) -- ++(1,-1) -- ++(0,1) -- ++(1,-1);
\generalfill;
\draw[rounded corners=1pt,red,opacity=\redop ] (2,2-0.05) -- ++(0,-1) -- ++(-1,1) -- ++(0,-2) -- ++(-1,1) -- ++(0,-1);
\node[anchor=west] at (\qwe,\abh) {\smee{$(2{,}0{;}2{,}1)$}};
\node[anchor=west] at (\qwe,\ach) {\smee{$\langle 0{,}1{,}1{,}1{,}1{,}0 \rangle$}};
\commonframe;
\end{tikzpicture}
}

\newcommand{\hreetwentynine}{
\begin{tikzpicture}[scale=\mysk, line width=2.0pt ]
\draw[rounded corners=2pt,name path=A] (0.025,-0.025) -- ++(3,0) -- ++(0,2);
\draw[rounded corners=2pt,name path=B] (-0.025,0.025) -- ++(0,3) -- ++(2,-2) -- ++(0,2) -- ++(1,-1);
\generalfill;
\draw[rounded corners=1pt,red,opacity=\redop ] (2,2-0.05) -- ++(0,-2) -- ++(-2,2) -- ++(0,-2);
\node[anchor=west] at (\qwe,\abh) {\smee{$(1{,}1{;}2{,}2)$}};
\node[anchor=west] at (\qwe,\ach) {\smee{$\langle 0{,}2{,}2{,}0 \rangle$}};
\commonframe;
\end{tikzpicture}
}

\newcommand{\hreethirty}{
\begin{tikzpicture}[scale=\mysk, line width=2.0pt ]
\draw[rounded corners=2pt,name path=A] (0.025,-0.025) -- ++(3,0) -- ++(0,2);
\draw[rounded corners=2pt,name path=B] (-0.025,0.025) -- ++(0,2) -- ++(1,-1) -- ++(0,1) -- ++(1,-1) -- ++(0,2) -- ++(1,-1);
\generalfill;
\draw[rounded corners=1pt,red,opacity=\redop ] (2,2-0.05) -- ++(0,-2) -- ++(-1,1) -- ++(0,-1) -- ++(-1,1) -- ++(0,-1);
\node[anchor=west] at (\qwe,\abh) {\smee{$(1{,}0{;}2{,}2)$}};
\node[anchor=west] at (\qwe,\ach) {\smee{$\langle 0{,}2{,}1{,}0{,}1{,}0 \rangle$}};
\commonframe;
\end{tikzpicture}
}

\def\Split{S}

\def\Rec{\mathsf{Rec}}
\def\DecRec{\mathsf{SortedRec}}

\title{The sandpile model on the complete split graph: $q,t$-Schr\"oder polynomials, sawtooth polyominoes, and a cycle lemma}
\author{Henri Derycke, Mark Dukes, and Yvan Le Borgne}
\date{}
\address{IGM, Universit\'e Gustave Eiffel, 77454 Marne-la-Vallée cedex 2, France.}
\email{henri.derycke@univ-eiffel.fr}
\address{School of Mathematics and Statistics, University College Dublin, Dublin 4, Ireland.}
\email{mark.dukes@ucd.ie}
\address{LaBRI, Universit\'e Bordeaux 1, 351 cours de la Lib\'eration, 33405 Talence cedex, France.}
\email{borgne@labri.fr}
\thanks{YLB was partially supported by ANR Project COMBIN\'E  ANR-19-CE48-0011}

\begin{document}

\begingroup
\def\uppercasenonmath#1{} 
\let\MakeUppercase\relax 
\maketitle
\endgroup

\begin{abstract}
This paper studies sorted recurrent configurations of the Abelian sandpile model on the complete split graph. We introduce two natural toppling processes, CTI and ITC toppling, on the recurrent configurations and use these to define two toppling delay statistics, wtopple$_{CTI}$ and wtopple$_{ITC}$. These new toppling delay statistics are time-weighted sums for the number of vertices that topple during each iteration of the toppling processes. We then introduce the bivariate {\em $q,t$}-CTI and {\em $q,t$}-ITC polynomials that are the generating functions of the bistatistics (level,wtopple$_{ITC}$) and (level,wtopple$_{CTI}$), where level is the well-established sandpile level statistic.

We prove the bistatistic (level,wtopple$_{ITC}$) maps to a bistatistic (area,bounce) on Schr\"oder paths that was introduced by Egge, Haglund, Killpatrick and Kremer (2003). This establishes equality of the $q,t$-ITC polynomial and the $q,t$-Schr\"oder polynomial of those same authors. This connection allows us to relate the $q,t$-ITC polynomial to the theory of symmetric functions and also establishes symmetry of the $q,t$-ITC polynomials. We conjecture equality of the $q,t$-CTI and $q,t$-ITC polynomials.

We also present and prove a characterization of sorted recurrent configurations as a new class of polyominoes that we call {\em sawtooth polyominoes}. The CTI and ITC toppling processes on sorted recurrent configurations are proven to correspond to bounce paths within the polyominoes. The main difference between the two bounce paths is the initial direction in which they travel. In addition to this, and building on the results of Aval, D'Adderio, Dukes, and Le Borgne (2016), we present a cycle lemma for a slight extension of stable configurations that allows for an enumeration of sorted recurrent configurations within the framework of the sandpile model.
\end{abstract}

\tableofcontents

\section{Introduction}\label{intro}
Cori and Rossin's study~\cite{cori2000} of the Abelian sandpile model (ASM) on the complete graph showed how recurrent 
configurations of the ASM on that graph are in one-to-one correspondence with parking functions.
Subsequent work by Cori and Poulalhon~\cite{cori2002} showed a similarly rich combinatorial characterization 
of recurrent configurations of the ASM on the complete multipartite graph $K_{1,p_1,\ldots,p_t}$ in terms of generalized parking functions termed
{\em $(p_1,\ldots ,p_t)$ parking functions}. 
For a general graph $G$, there are several known bijections between recurrent configurations of the ASM on $G$ and spanning trees of the graph~\cite{bij1,bij2,bij3,Dhar}. 
More generally, the notion of a $G$-parking function captures the concept of recurrent configurations and 
features as the basis of algebras associated to the spanning trees of a general graph, see e.g. Postnikov and Shapiro~\cite{ps}. 

In a different direction, Dukes and Le Borgne~\cite{dlb} proved that sorted recurrent configurations of the ASM on the 
complete bipartite graph $K_{m,n}$ were in one-to-one correspondence with parallelogram polyominoes having an $m\times n$ bounding box.
That research was extended by Aval, D'Adderio, Dukes, Hicks, and Le Borgne in \cite{aadhl} and \cite{aadl} wherein a more thorough 
consideration of statistics on these polyominoes revealed a connection to the theory of symmetric functions. 
This connection was used to prove several symmetry conjectures first stated in \cite{dlb}.

The complete split graph is a graph consisting of two distinct parts:
a {\it clique part} in which all distinct pairs of vertices are connected by a single edge and
an {\it independent part} in which no two vertices are connected to an edge.
In addition, there is precisely one edge between every pair of vertices that lie in different parts.  
In this paper we denote by $\Split_{n,d}$ the complete split graph that consists of 
the solitary sink part $\{s\}$, vertices $V=\{v_1,\ldots,v_{n}\}$ in the clique part, and vertices $W=\{w_1,\ldots,w_{d}\}$ in the independent part.\footnote{Our decision to discount mention of the sink in the parameters of $\Split_{n,d}$ improves the presentation.}
The graph $S_{5,3}$ is illustrated in Figure~\ref{figsplit}. 

\begin{figure}[!h]
\label{figsplit}
	\begin{center}
	\unlabelledsplitgraph
	\end{center}
\caption{The complete split graph $S_{5,3}$. The sink is in the dashed rectangle, the 5 vertices below it form the clique part, and the $3$ vertically aligned vertices to its right are the independent part.}
\end{figure}

Dukes~\cite{ncf} recently characterized the recurrent configurations of the ASM on $S_{n,d}$.
He presented a bijection from sorted recurrent configurations on $\Split_{n,d}$ to 
$\Schroder_{n,d}$, the set of Schr\"oder words consisting of $n$ up steps, $n$ down steps, 
and $d$ horizontal steps.\footnote{In Dukes~\cite{ncf} the correspondence was presented with M\"otzkin paths. 
We have chosen to use the equivalent term Schr\"oder path here due to its relation to papers concerning the $q,t$-Schr\"oder polynomial. 
Moreover notice the change of notations for the split graph parameters: $n\leftarrow m-1$ and $d\leftarrow n$.}
We delve further into the research presented in Dukes~\cite{ncf} with a view to extending 
that work (in the same manner as was done in \cite{aadhl,aadl} for the complete bipartite graph) to the complete split graph.

In this paper we consider the ASM on the complete split graph in which a clique vertex is the sink.
We will only be interested in sorted, or {\it weakly-decreasing}, recurrent configurations on $\Split_{n,d}$.
Beyond the enumerative connection with Schr\"oder paths, this restriction to sorted configurations is also motivated by the natural action of the symmetric group on recurrent configurations, and this features in Section~\ref{sectionsix} and also in the recent work of D'Adderio et al.~\cite{ddillw}.
In Section~\ref{sectiontwo} we introduce the ASM and recall some terminology. 
We define two toppling processes, CTI and ITC, on sorted recurrent configurations that will be used throughout the paper.
We also introduce two bivariate polynomials that each encode two statistics on these sorted recurrent configurations and call these the $q,t$-CTI and $q,t$-ITC polynomials.

In Section~\ref{sectionthree} we modify the bijection from sorted recurrent configurations to Schr\"oder paths given in Dukes~\cite{ncf}, 
and prove it translates the level statistic on configurations to the area statistic on Schr\"oder paths while also translating a delay statistic to a bounce statistic on the paths.
This allows us to prove that the $q,t$-ITC polynomial introduced in Section~\ref{sectiontwo} is equal to the $q,t$-Schr\"oder polynomial of Egge, Haglund, Killpatrick and Kremer~\cite{ehkk}, and in so doing relate it to the theory of symmetric functions. 
Two symmetry properties of the $q,t$-ITC polynomials follow as a result of this connection.
We also present a formula for the $q,t$-ITC polynomials that is different to the one given by Egge et al.

In Section~\ref{sectionfive} we define a new class of polyominoes that we term {\em sawtooth polyominoes}.
We prove they are in bijection with Schr\"oder words. 
We also show how two bounce paths within these new polyominoes illustrate the CTI and ITC toppling processes of the corresponding recurrent configuration. 
Moreover, we show how to directly construct the sawtooth polyomino that corresponds to a recurrent configuration. 
The statistics that form the $q,t$-CTI and $q,t$-ITC polynomials are also expressible as statistics on the corresponding sawtooth polyominoes. 

In Section~\ref{sectionsix} we introduce a framework for the ASM on the complete split graph and 
within that framework derive a cycle lemma for ASM configurations. 
The notion of a cycle lemma in this context comes from the work of Dvoretzky and Motzkin~\cite{dvor} in which they 
(re-)consider a vote-counting problem and show how all the possible outcomes can be partitioned into groups. 
There is precisely one `favourable outcome' in each group.
A cycle permutation acts on the elements of each group and maps to other elements within that group, and from which it follows that each group has the same size. 
Combining these facts shows that the number of favourable voting outcomes
they wish to count equals the total number of outcomes divided by the size of each group.
The content of this section is completely separate to the material presented in Sections~\ref{sectiontwo}--\ref{sectionfive}.
This builds on work from Aval et al.~\cite{aadl} in which a cycle lemma for configurations on the complete bipartite graph was proven.
That this framework enables this to be done without having to utilize some planar representations of the configurations as intermediate objects suggests a more general result may hold true.
An outcome of this is that it provides another enumeration of sorted recurrent configurations. 
Finally, in Section~\ref{section:conclusion} equality of the $q,t$-CTI and $q,t$-ITC polynomials is conjectured.

This paper also adds to the growing body of recent research that has found unexpected connections between 
recurrent configurations of the ASM on parameterized graph classes and other combinatorial objects~\cite{aadhl,aadl,imrn,mylb,ddillw,dsss,ferrers}.

\section{The sandpile model on $S_{n,d}$ and two toppling processes}\label{sectiontwo}

In this section we first recall some sandpile terminology and concepts, and then
introduce two toppling conventions that will be important throughout the paper.
The ASM may be defined on any undirected graph $G$ with a designated vertex $s$ called the \emph{sink}.  
A \emph{configuration} on $G$ is an assignment of non-negative integers to the non-sink vertices of $G$:
\begin{align*}
c:V(G)\backslash\{s\} ~\mapsto ~ \mathbb{N}=\{0,1,2,\ldots\}.
\end{align*}
The number $c(v)$ is referred to as \emph{the number of grains} at vertex~$v$ or the \emph{height} of $v$.  
Given a configuration $c$, a vertex $v$ is said to be {\emph{stable}} if the number of grains at $v$ is strictly smaller than the threshold of that vertex, which is the degree of $v$, denoted $\thresh(v)$.
Otherwise $v$ is {\emph{unstable}}.  A configuration is {\emph{stable}} if all non-sink vertices are stable.

If a vertex is unstable then it may {\emph{topple}}, which means the vertex donates one grain to each of its neighbors.  
The sink vertex has no height associated with it and only absorbs grains, thereby modelling grains exiting the system.
Starting from any configuration $c$ and successively toppling unstable vertices one will eventually reach a stable configuration $c'$.  
This final configuration $c'$ does not depend on the order in which unstable vertices were toppled. We call $c'$ the \emph{stabilization} of $c$.

Starting from the empty configuration, one may indefinitely add any number of grains to any vertices in $G$ and topple vertices should they become unstable.  
Certain stable configurations will appear again and again, that is, they \emph{recur}, while other stable configurations will never appear again.
These {\emph{recurrent configurations}} are the ones that appear in the long term limit of the system. 
Let $\Rec(G)$ be the set of recurrent configurations on $G$.

Determining the set $\Rec(G)$ for a given graph $G$ is not a straightforward task. 
In \cite[Section 6]{Dhar}, Dhar describes the so-called \emph{burning algorithm}, which establishes in linear time whether a given stable configuration is recurrent:

\begin{prop}[\cite{Dhar}, Section 6.1]\label{pro:DharBurning}
Let $G$ be a graph with sink $s$, and let $c$ be a stable configuration on $G$. Then $c$ is recurrent if and only if there exists an ordering $v_0=s,v_1,\ldots,v_{n}$ of the vertices of $G$ such that, starting from $c$, for any $i \geq 1$, toppling the vertices $v_0,\ldots,v_{i-1}$ causes the vertex $v_i$ to become unstable. Moreover, if such a sequence exists, then toppling $v_0,\ldots,v_{n}$ returns the initial configuration $c$.
\end{prop}

A variant of the sequential toppling outlined in the previous proposition is that of {\em{parallel chip-firing}} or {\em parallel toppling}.
At each time step the set of unstable vertices is recorded. 
Then, in parallel, all of these unstable vertices are toppled to give the next (possibly unstable) configuration.
We will consider two slight variants of the classical parallel chip-firing process that we call CTI toppling and ITC toppling. 
We first explain CTI toppling.
This variant is always only applied to a recurrent configuration $c \in \Rec(S_{n,d})$ whose sink has just been toppled.
Let $\ClassOneParams{n}{d}$ be the set of sorted recurrent configurations on $\Split_{n,d}$, 
i.e. those recurrent configurations
$(a_1,\ldots,a_{n};b_1,\ldots,b_d)$  for which 
\begin{equation*}
a_1\geq a_2\geq  \ldots \geq a_{n} \mbox{ and }b_1\geq b_2\geq \ldots \geq b_d.
\end{equation*}
These ordered recurrent configurations correspond to orbits of the recurrent configurations acted on by the automorphism group of the graph.

\subsection{CTI toppling}
Given an unstable configuration, we first check if there are any unstable clique vertices. 
If there are then topple these in parallel.
Next check if there are any unstable independent vertices. 
If there are then topple these in parallel.
We repeat these two steps successively until there are no remaining unstable vertices. 
We will refer to this toppling order as {\it{CTI toppling}}, an abbreviation for \underline{C}lique \underline{T}hen \underline{I}ndependent.
We will write $\ToppleRecord_{CTI}(c) = (P_1,Q_1,\ldots,P_t,Q_t)$ to indicate the vertices that were toppled in parallel at each step. 
Here $P_1$ represents the (non-sink) clique vertices that were initially toppled in parallel. 
$Q_1$ is the collection of independent vertices that were next toppled, and so on. 
Note that each of $P_i$ and $Q_i$ can be empty, but they cannot both be empty as at least one toppling 
must happen during each iteration of toppling clique-then-independent vertices, i.e. one must have 
\begin{equation*}
P_i\cup Q_i \neq \emptyset.
\end{equation*}
We will find it convenient to replace the comma that separates the clique and independent parts in a configuration with a semi-colon.

\begin{example}
Consider the sorted recurrent configuration $c=(7,6,5,2,1;5,4,4)$ on $S_{5,3}$.
Topple the sink to get $(8,7,6,3,2;6,5,5)$. 
The set of unstable clique vertices is now $P_1=\{v_1\}$. 
Topple all in $P_1$ to get the configuration $(0,8,7,4,3;7,6,6)$. 
The set of unstable independent vertices is now $Q_1=\{w_1,w_2,w_3\}$. 
Topple all in $Q_1$ to get the configuration $(3,11,10,7,6;1,0,0)$.
The set of unstable clique vertices is $P_2=\{v_2,v_3\}$.
Topple all in $P_2$ to get the configuration $(5,4,3,9,8;3,2,2)$. 
The set of unstable independent vertices is $Q_2=\emptyset$ so the configuration remains unchanged.
The set of unstable clique vertices is $P_3=\{v_4,v_5\}$.
Topple all in $P_3$ to get the original configuration $(7,6,5,2,1;5,4,4)$. 
As there are no further topplings to be done, we will set $Q_3=\emptyset$ so that $P_3$ has a pair.
To conclude, $\ToppleRecord_{CTI}(c)=(\{v_1\},\{w_1,w_2,w_3\},\{v_2,v_3\},\emptyset,\{v_4,v_5\},\emptyset)$.
\end{example}

Given $c \in \ClassOneParams{n}{d}$,
we define the {\em height of $c$}, $\height(c)$, to be the sum of the configuration entries.
We also define the {\em level of $c$} as 
\begin{align*}
\level(c)=\height(c)-\left( \binom{n+d}{2}-\binom{d}{2}\right),
\end{align*} the height less the number of non-sink incident edges.
Suppose $\ToppleRecord_{CTI}(c) = (P_1,Q_1,\ldots,P_t,Q_t)$ and set $p_i:=|P_i|$ and $q_i:=|Q_i|$.
Define 
$\bouncepath_{CTI}(c) := (p_1,q_1,\ldots,p_t,q_t),$
and 
\begin{equation*}
\topplebounce_{CTI}(c) := 1(p_1+q_1)+2(p_2+q_2)+\ldots+t(p_t+q_t).
\end{equation*}

\begin{example}
\label{examplelala}
Consider 
the ASM on $S_{2,2}$ with sink $v_3$.
In the table below for every configuration $c \in \ClassOneParams{2}{2}$ we give the height $\height(c)$, 
the sequence $\bouncepath_{CTI}(c)$, and the quantity $\topplebounce_{CTI}(c)$.
\begin{equation*}
\scriptsize
\begin{array}{cclc|cclc} \hline
c & \height(c) & \bouncepath_{CTI}(c) & \topplebounce_{CTI}(c) &
c  & \height(c) & \bouncepath_{CTI}(c) & \topplebounce_{CTI}(c) \\ \hline
(3, 3; 2, 2) & 10 &  (2, 2) & 4  	& 
  (3, 1; 1, 0) & 5 &  (1, 1, 1, 1)&  6 	\\
(3, 3; 2, 1) & 9 &  (2, 2)&  4 		&
  (3, 1; 1, 1) & 6 &  (1, 2, 1, 0) & 5	\\
(3, 3; 2, 0) & 8 &  (2, 2)&  4 		&
  (3, 0; 2, 2) & 7 &  (1, 2, 1, 0) & 5	\\
(3, 3; 1, 1) & 8 &  (2, 2) & 4		&
  (3, 0; 2, 1) & 6 &  (1, 2, 1, 0)&  5	\\
(3, 3; 1, 0) & 7 &  (2, 2)&  4		&
  (3, 0; 1, 1) & 5 &  (1, 2, 1, 0) & 5	\\
(3, 3; 0, 0) & 6 &  (2, 2) & 4 		&
  (2, 2; 2, 2) & 8 &  (0, 2, 2, 0) & 6 	\\
(3, 2; 2, 2) & 9 &  (1, 2, 1, 0)&  5 		&
  (2, 2; 2, 1) & 7 &  (0, 1, 2, 1) & 7 	\\
(3, 2; 2, 1) & 8 &  (1, 2, 1, 0) & 5		&
  (2, 2; 2, 0) & 6 &  (0, 1, 2, 1) & 7	\\
(3, 2; 2, 0) & 7 &  (1, 1, 1, 1) & 6 		&
  (2, 1; 2, 2) & 7 &  (0, 2, 2, 0) & 6 	\\
(3, 2; 1, 1) & 7 &  (1, 2, 1, 0) & 5 		&
  (2, 1; 2, 1) & 6 &  (0, 1, 1, 1, 1, 0) & 8	\\
(3, 2; 1, 0) & 6 &  (1, 1, 1, 1) & 6		&
  (2, 1; 2, 0) & 5 &  (0, 1, 1, 0, 1, 1) & 9	\\
(3, 2; 0, 0) & 5 &  (1, 0, 1, 2) & 7 		&
  (2, 0; 2, 2) & 6 &  (0, 2, 1, 0, 1, 0)&  7 	\\
(3, 1; 2, 2) & 8 &  (1, 2, 1, 0) & 5		&
  (2, 0; 2, 1) & 5 &  (0, 1, 1, 1, 1, 0) & 8	\\
(3, 1; 2, 1) & 7 &  (1, 2, 1, 0) & 5		&
  (1, 1; 2, 2) & 6 &  (0, 2, 2, 0)&  6	\\
(3, 1; 2, 0) & 6 &  (1, 1, 1, 1)&  6		&
  (1, 0; 2, 2) & 5 &  (0, 2, 1, 0, 1, 0) & 7 \\ \hline
\end{array}
\end{equation*}
\end{example}

Given the set-up above, let us define the following bivariate polynomial that we call 
the {\em $q,t$-CTI polynomial}:

\newcommand{\notrepoly}{\mathfrak{F}}
\begin{equation*}
\notrepoly_{n,d}^{CTI}(q,t) := \sum_{c \in \ClassOneParams{n}{d}} 
	q^{\level(c)} t^{\topplebounce_{CTI}(c)-(n+d)}.
\end{equation*}

\begin{example}
The polynomial $\notrepoly_{2,2}^{CTI} (q,t)$ is readily calculated from the table in Example~\ref{examplelala}.
\begin{align*}
\notrepoly_{2,2}^{CTI} (q,t) =~~ &
q^5 + t^5 + q^4 t + q t^4 + q^3 t^2 + q^2 t^3 + q^4 
+ t^4 + 2q^3 t + 2q t^3 + 2q^2 t^2 + 2q^3 + 2t^3 \\ 
& + 3q^2 t + 3qt^2 + q^2 + t^2 + 2qt + q + t. 
\end{align*}
\end{example} 

\subsection{ITC toppling}
Now that CTI toppling and the statistics associated with it have been defined, it is straightforward to define ITC toppling and its associated statistics. 
Almost everything is the same as in the CTI case, except what happens immediately after toppling the sink. 
ITC toppling stands for {\underline{I}}ndependent \underline{T}hen \underline{C}lique toppling:
Given the unstable configuration that results from toppling the sink, first identify those independent vertices that are unstable. 
Topple those unstable independent vertices in parallel and identify those clique vertices that are unstable. 
Topple those unstable clique vertices, and so on.

Given $c \in \ClassOneParams{n}{d}$,
$\height(c)$ is the sum of the configuration entries, as before.
Suppose $\ToppleRecord_{ITC}(c) = (Q'_1,P'_1,\ldots,Q'_t,P'_t)$ and set $p'_i=|P'_i|$ and $q'_i=|Q'_i|$.
Define $\bouncepath_{ITC}(c) := (q'_1,p'_1,\ldots,q'_t,p'_t),$
and 
\begin{equation*}
\topplebounce_{ITC}(c) := 1(q'_1+p'_1)+2(q'_2+p'_2)+\ldots+t(q'_t+p'_t).
\end{equation*}
Finally, we define the accompanying polynomial to these two statistics:
\begin{equation*}
\notrepoly_{n,d}^{ITC}(q,t) := \sum_{c \in \ClassOneParams{n}{d}} q^{\height(c)- \left(\binom{n+d}{2}-\binom{d}{2}\right)} t^{\topplebounce_{ITC}(c)-(n+d)},
\end{equation*}
and call it the {\em $q,t$-ITC polynomial}.

\begin{example}\label{itctopplingexample}
Consider the recurrent configuration $c=(7,6,5,2,1;5,4,4)\in \ClassOneParams{5}{3}$. 
Topple the sink to get $(8,7,6,3,2;6,5,5)$. The set of currently unstable independent vertices is $Q_1'=\{w_1\}$. 
Topple all in $Q_1'$ to get $(9,8,7,4,3;0,5,5)$. The set of currently unstable clique vertices is $P_1' = \{v_1,v_2\}$. Topple all in 
$P_1'$ to get the configuration $(2,1,9,6,5;2,7,7)$. The set of currently unstable independent vertices is $Q_2'=\{w_2,w_3\}$. 
Topple all in $Q_2'$ to get the configuration $(4,3,11,8,7;2,1,1)$. 
From this we have $P_2'=\{v_3,v_4\}$ and after toppling these we have the configuration $(6,5,4,1,9;4,3,3)$. 
The set of currently unstable independent vertices is empty, so $Q_3'=\emptyset$ and the configuration remains unchanged. 
The set of currently unstable clique vertices is $P_3'=\{v_5\}$, and toppling all in $P_3'$ results in the original configuration $c=(7,6,5,2,1;5,4,4)$.
Thus 
\begin{align*}
\ToppleRecord_{ITC}(c) =& \left( \{w_1\}, ~ \{v_1,v_2\}, ~ \{w_2,w_3\}, ~ \{v_3,v_4\}, ~ \emptyset, ~ \{v_5\} \right)\\
\bouncepath_{ITC}(c) =& (1,2,2,2,0,1),
\intertext{and so}
\topplebounce_{ITC}(c) =& 1(1+2)+2(2+2)+3(0+1) = 14.
\end{align*}
\end{example}

We will have more to say about the $\notrepoly_{n,d}^{ITC}(q,t)$ polynomials in the next section and their relationship to other polynomials that appear in the literature.
One final piece of notation that is related to ITC toppling and will be of use going forward is the following. 

\begin{definition}\label{twosix}
Suppose $c \in \ClassOneParams{n}{d}$ with $\bouncepath_{ITC}(c) := (q'_1,p'_1,\ldots,q'_t,p'_t).$
We will call the pair 
\begin{equation*}
\left[(q'_1,\ldots,q'_t),~ (p'_1,\ldots,p'_t)\right]
\end{equation*} an {\em ITC-toppling sequence of length $t$}. (Convention: $q'_0:=0$ and $p'_0:=1$.)
Let $\mathsf{ITC}_{n,d,t}$ be the set of all possible ITC-toppling sequences of length $t$ on $\Split_{n,d}$, and 
set 
\begin{equation*}
\mathsf{ITC}_{n,d} := \displaystyle\bigcup_{k\geq 1} \mathsf{ITC}_{n,d,k}.
\end{equation*}
\end{definition}

The set of all possible ITC-toppling sequences is characterised in Theorem~\ref{lem:description-ITCnm}.
Recall that a {\em composition} of an integer $n$ is a sequence of positive integers $a=(a_1,\ldots,a_k)$ 
whose sum is $n$.  We will denote this $a \iscomp n$.
We will use the notation $a\iscomp_k n$ to indicate $a$ has length $k$.
A {\em weak composition} of $n$ is a sequence of non-negative integers whose sum is $n$.
We will denote this $a \isweakcomp n$ and use the notation $a\isweakcomp_k n$  to indicate a 
weak-composition $a$ has length $k$.

\begin{theorem}\label{lem:description-ITCnm} For all $n\geq 1$, $d\geq 0$, 
\begin{equation*}
\mathsf{ITC}_{n,d} = \{\left[(d),(n)\right]\} \cup
	\bigcup_{k\geq 2} 
	\left\{ 
	\left[(b_1,\ldots,b_k), (a_1,\ldots,a_{k})\right]
		~:~ 
	\begin{array}{l}
		(b_1,\ldots,b_k) \isweakcomp d,\\
		(a_1,\ldots,a_{k}) \isweakcomp n,\\
		a_1,\ldots,a_{k-1}>0, \mbox{ and }\\
		b_k+a_k>0
	\end{array}
	\right\}.
\end{equation*}
\end{theorem}

\begin{proof}
Let $X_{n,d}$ be the set on the right-hand side of the stated equality.
From Definition~\ref{twosix} we have
\begin{align*}
\mathsf{ITC}_{n,d} = & \left\{ 
\left[ (b_1,\ldots,b_k), (a_1,\ldots,a_{k})\right] ~:~ 
\right. \\
& \qquad \left. 
(b_1,a_1,\ldots,b_k,a_k)= \bouncepath_{ITC}(c) \mbox{ for some }c \in \ClassOneParams{n}{d}
\right\}.
\end{align*}

Consider $$\ToppleRecord_{ITC}(c) = (B_1,A_1,\ldots,B_k,A_k)$$ for some $c \in \ClassOneParams{n}{d}$.
This toppling process repeatedly parallel-topples those independent vertices that are currently unstable, followed by a parallel-toppling of those clique vertices that are then unstable.
In this way it follows that $b_i+a_i>0$ for all $1\leq i\leq k$, where $a_i=|A_i|$ and $b_i=|B_i|$.
Suppose that, for some $1\leq i<k$, we were to parallel-topple all unstable independent vertices $B_i$ to then find there are no unstable clique vertices $A_i$.
It must be the case that one now has a stable configuration since the toppling of the independent vertices $B_i$ has no effect on the other independent vertices. 
For this reason $A_i$ must be non-empty for all non-final times, i.e. $a_i>0$ for all $1\leq i<k$.
Note that this does {\em not} imply that $A_k$ is necessarily empty.
Consequently, we have 
$b_i\geq 0$ for all $1\leq i \leq k$ and $a_i>0$ for all $1\leq i<k$. 
This shows the condition $a_i+b_i>0$ for all $1\leq i \leq k$ is already satisfied for all cases except $i=k$.
For the case $i=k$, since it is possible that $a_k=0$ we must include the condition $b_k+a_k>0$.

Dhar's burning algorithm tells us that for every recurrent configuration $c$, every non-sink vertex topples exactly once.
Thus $a_1+\ldots+a_k=n$ and $b_1+\ldots+b_k=d$ for some $k$. 
Therefore
\begin{align*}
\mathsf{ITC}_{n,d} \subseteq  & \left\{
    \left[(b_1,\ldots,b_k), (a_1,\ldots,a_{k})\right]
    ~:~ \begin{array}{l} (b_1,\ldots,b_k) \isweakcomp d, \\ (a_1,\ldots,a_{k-1}) \iscomp n-a_k\\ a_k+b_k>0 \mbox{ and }a_k\geq 0 \end{array}\right\} =  X_{n,d}.
\end{align*}
Next, consider $\left[(b_1,\ldots,b_k), (a_1,\ldots,a_{k})\right] \in X_{n,d}$. 
Let $c$ be the configuration on $\Split_{n,d}$ defined as follows:
\begin{align*}
c(v_i) :=& n+d - \left(\sum_{\ell =0}^{j-1} a_{\ell}+b_{\ell} \right) - b_j, 
\end{align*}
for all $n-(a_0+\ldots +a_j) \leq i < n-(a_0+\ldots +a_{j-1})$ and $1\leq j\leq k$, and
\begin{align*}
c(w_i) :=& n+1 - (a_0+a_1+\ldots+a_{j-1}),
\end{align*}
for all $d-(b_0+\ldots+b_j) \leq i < d-(b_0+\ldots + b_{j-1})$ and $1\leq j \leq k$.
It is straightforward to see that this configuration is indeed recurrent, weakly decreasing, and has $\bouncepath_{ITC}(c) = (b_1,a_1,\ldots,b_k,a_k)$.
This shows $X_{n,d} \subseteq \mathsf{ITC}_{n,d}$. 
Therefore $\mathsf{ITC}_{n,d} = X_{n,d}$ as claimed.
\end{proof}

\begin{example}\label{itc22}
The set $\mathsf{ITC}_{2,2} = \mathsf{ITC}_{2,2,1} \cup \mathsf{ITC}_{2,2,2} \cup \mathsf{ITC}_{2,2,3}$, where $\mathsf{ITC}_{2,2,1} = \left\{ \left[(2),~(2)\right] \right\}$,
\begin{align*}
\mathsf{ITC}_{2,2,2} =& \left\{ \begin{array}{l}
	\left[(0,2),~ (2,0)\right],\\
	\left[(1,1),~ (2,0)\right],\\
	\left[(2,0),~ (1,1)\right],\\
	\left[(1,1),~ (1,1)\right],\\
	\left[(0,2),~ (1,1)\right]
	\end{array}\right\}
\mbox{ and }
\mathsf{ITC}_{2,2,3} = \left\{\begin{array}{l}
\left[(1, 0, 1),~ (1, 1, 0)\right],\\
\left[(0, 1, 1),~ (1, 1, 0)\right],\\
\left[(0, 0, 2),~ (1, 1, 0)\right]
\end{array}
\right\}.
\end{align*}
\end{example}

The number of elements in the sets $\ITC_{n,d,k}$ and $\ITC_{n,d}$ are given in Lemma~\ref{itc:sequence:enumeration}.

\section{Schr\"oder paths and ITC-topplings}\label{sectionthree}
In this section we will use a modification of the bijection from Dukes~\cite{ncf} to show how an area and bounce bistatistic on Schr\"oder paths corresponds to 
a height and weighted-toppling bistatistic on sorted recurrent configurations when ITC topplings is employed.
Experimentally, both $\notrepoly_{n,d}^{CTI} (q,t)$ and $\notrepoly_{n,d}^{ITC} (q,t)$ appear to be equal and symmetric in $q$ and $t$. 
This was the case for the corresponding bivariate polynomials for both the complete and complete bipartite graphs, where connections were found between paths and symmetric functions. 
In this paper we establish such a connection for $\notrepoly_{n,d}^{ITC} (q,t)$. 

Let $\Schroder_{n,d}$ be the set of Schr\"oder words consisting of $n$ $U$'s, $n$ $D$'s, and $d$ $H$'s. 
The defining property of such words is that for every prefix, the number of $D$'s that appear is never more than the number of $U$'s that appear.
In Dukes~\cite{ncf} it was shown that configurations in $\ClassOneParams{n}{d}$ are in one-to-one correspondence with paths in $\Schroder_{n,d}$ via a bijection
\begin{equation*}
\stc: \Schroder_{n,d} \mapsto \ClassOneParams{n}{d}.
\end{equation*}
The bijection $\stc$ is defined as follows.
\begin{definition}
Let $p \in \Schroder_{n,d}$ be a Schr\"oder word. Then 
$$\stc(p):=(a_1,\ldots,a_{n};b_1,\ldots,b_d)$$ where
$b_i$ is the number of $D$'s following the $i$th $H$ of $p$, and
$a_j+1$ is the number of non-$U$'s following the $j$th $U$ of $p$.
\end{definition}
\begin{example}\label{example:three:two}
Let $p=UHUDUHHDUDUDD \in \Schroder_{5,3}$. Then
\begin{equation*}
\stc(p)=(7,6,5,2,1;5,4,4)\in \ClassOneParams{5}{3}.
\end{equation*}
\end{example}
Schr\"oder paths are pictorial representations of Schr\"oder words where we identify $U$, $H$, and $D$ with the steps $(0,1)$, $(1,1)$, and $(1,0)$, respectively.
The blue line in Figure~\ref{fig:three:three} illustrates the Schr\"oder path for $p$.

Given a Schr\"oder word/path $p$, we now define its area and bounce as in Haglund~\cite[Sec.1]{haglund-qtschroder}.
(Figures~\ref{fig:three:three} and \ref{fig:three:four} contain illustrations of these definitions.)
Let $\area(p)$ to be the number of `lower triangles' (triangles whose vertex set is $\{(i,j),(i+1,j),(i+1,j+1)\}$) between $p$ and the diagonal $y=x$.
To define $\schbounce(p)$, first remove all the $H$ steps from $p$ and collapse in the obvious way to form a Dyck path $C(p)$.
Next construct the bounce path for $C(p)$ which is the classical Dyck bounce path from $(a,a)$ to $(0,0)$: 
move west until at the boundary of the path at that level; then move down until meeting $y=x$ and change direction to west; and repeat the previous two in that order.
Label the points at which it touches the Dyck path $\PeakDyck(1),\PeakDyck(2),\ldots$, beginning at the top left and moving towards the origin. These are labelled with red dots in the associated diagrams.

The bounce, $\mathrm{bounce}(C(p))$, is the sum of the $x$-coordinates of where the bounce path meets the diagonal (not including the initial point).
The $U$ steps of this bounce path occurring just before $D$ steps of the bounce path are also $U$ steps of $C(p)$. 
The corresponding $U$ steps of $p$ are called the `peaks' of $p$. 
Let us label the points (illustrated with red dots in the associated diagrams) which correspond to the tops of the peaks with $\PeakSchroder(1),\PeakSchroder(2),\ldots$, beginning at the top right and moving towards the origin.
For each $H$ step $\alpha$ of $p$ let $b(\alpha)$ denote the number of peaks above it, and define the bounce of the Schr\"oder path $p$ to be
\begin{equation*}
\schbounce(p) := \mathrm{bounce}(C(p)) +\sum_{\alpha} b(\alpha),
\end{equation*}
where the sum is over all $H$ steps of $p$. 
Egge et al.~\cite{ehkk} noted (just before their Conjecture 1) that N. Loehr observed that $\schbounce(p)$ equals the sum, over all peaks of the Schr\"oder path, of the number of first-quadrant squares to the left of each peak in the same row.
The bounce path may also be directly defined on a Schr\"oder paths by insisting that it moves parallel to each $H$ step as it passes through the anti-diagonal of the $H$ step, and then continues in the same direction it was moving before encountering the $H$ step until next hitting the Schr\"oder path. Such anti-diagonals are illustrated as the shaded grey regions in Figure~\ref{fig:reformulate-bounce}.
\begin{example}\label{example:three:three}
Consider $w= UHUDUHHDUDUDD\in \Schroder_{5,3}$, the same word that appeared in Example~\ref{itctopplingexample}.
The configuration that corresponds to $w$ via the bijection $\stc$ is 
\begin{equation*}
c=\stc(w) = (7,6,5,2,1; 5,4,4) \in \ClassOneParams{5}{3}.
\end{equation*}
The Schr\"oder path for $w$ is illustrated in Figure~\ref{fig:three:three}.
The bounce of the collapse $C(w)$ of $w$ is 1+3 since the bounce path of $C(w)$ meets the diagonal at positions $(3,3)$ and $(1,1)$.
The peaks of the bounce path are illustrated by red dots and these are transferred over to the tops of the corresponding U steps in the diagram of $w$. 
The sum over all horizontal steps $\alpha$ of the quantity $b(\alpha)$ is 1+1+2=4, since $b(\alpha)$ is the number of red dots (that represent peaks) above step $\alpha$ in the diagram for $w$.
Thus $\schbounce(w) = 4+4=8$.
Equivalently, as per Loehr's observation, for each peak, the number of first quadrant squares to its left and in the same row (going from bottom to top) is $\schbounce(w)=0+2+6=8$.
\begin{figure}
    \begin{tikzpicture}[scale=0.75]
		\begin{scope}[xshift=0,yshift=0]
       	\draw (0,0) grid (8,8);
       	\draw[line width=1] (0,0) -- (8,8);
       	\draw[line width=3,blue] (0,0) -- (0,1) -- (1,2) -- (1,3) -- (2,3) -- (2,4) -- (3,5) -- (4,6) -- (5,6) -- (5,7) -- (6,7) -- (6,8) -- (7,8) -- (8,8);
       	\foreach \x/\y in {1/1,2/2,3/3,3/4,4/4,4/5,5/5,6/6,7/7}{
			     	\fill[white!50!black] (\x-0.05 ,\y+0.05 ) -- +(0,0.75) -- +(-0.75,0) -- cycle;
       	}
       	\foreach \x/\y in {0/1,2/4,6/8}{
         	\fill[red] (\x,\y) circle (0.15);
       	}
		\node[above] at (6,10) {$\PeakSchroder(1)$};
		\draw[->,gray,line width=1.5] (6,10) to[out=-90,in=135] (5.8,8.2);
		\node[above] at (3.5,9) {$\PeakSchroder(2)$};
		\draw[->,gray,line width=1.5] (3.5,9) to[out=-100,in=135] (1.8,4.2);
		\node[above right] at (0,8) {$\PeakSchroder(3)$};
		\draw[->,gray,line width=1.5] (1.5,8.1) to[out=-90,in=135] (-0.2,1.2);
		\node[above left] at (0.75,1.5) {\tiny$\alpha_3$};
		\node[above left] at (2.75,4.5) {\tiny$\alpha_2$};
		\node[above left] at (3.75,5.5) {\tiny$\alpha_1$};
		\node at (4,-1) {$w= UHUDUHHDUDUDD$};
		\node at (4,-1.75) {$\area(w)=9$};
		\node at (4,-2.5) {$b(\alpha_1)+b(\alpha_2)+b(\alpha_3) = 1+1+2=4$};
		\node at (4,-3.25) {$\schbounce(w)=4+4=8$};
		\end{scope}
		\begin{scope}[xshift=300,yshift=40]
       	\draw (0,0) grid (5,5);
       	\draw[line width=1] (0,0) -- (5,5);
       	\draw[line width=3,blue] (0,0) -- (0,2) -- (1,2) -- (1,3) -- (2,3) -- (2,4) -- (3,4) -- (3,5) -- (5,5);
       	\draw[line width=1.5,dashed,red] (5,5) -- (3,5) -- (3,3) -- (1,3) -- (1,1) -- (0,1) -- (0,0);
       	\draw[line width=1.5,red] (3.7,5.2) -- (3.5,5) -- (3.7,4.8);
       	\foreach \x/\y in {0/1,1/3,3/5}{
         	\fill[red] (\x,\y) circle (0.15);
       	}
		\node[above] at (4,7) {$\PeakDyck(1)$};
		\draw[->,gray,line width=1.5] (4,7) to[out=-90,in=90] (3,5.2);
		\node[above] at (2,6) {$\PeakDyck(2)$};
		\draw[->,gray,line width=1.5] (2,6) to[out=-100,in=135] (0.8,3.2);
		\node[above] at (0,5) {$\PeakDyck(3)$};
		\draw[->,gray,line width=1.5] (-0.2,5.1) to[out=-90,in=135] (-0.2,1.2);
		\node at (2.5,-1) {$C(w)$};
		\node at (2.5,-1.75) {$\mathrm{bounce}(C(w)) = 1+3=4$};
		\end{scope}
     \end{tikzpicture}
\caption{The Schr\"oder path $w= UHUDUHHDUDUDD\in \Schroder_{5,3}$ from Example~\ref{example:three:two} is illustrated by the blue line in the left diagram.
The Dyck path $C(w)$ that represents the collapse of this path, achieved by removing horizontal steps, is illustrated to the right using a blue line. 
The bounce path of $C(w)$ starts at $(5,5)$ and runs to $(0,0)$. Since it hits the diagonal at positions $(3,3)$ and $(1,1)$ we have $\mathrm{bounce}(C(w)) = 1+3=4$. 
The peaks of the bounce path are illustrated with red dots in the right diagram. 
These peaks are copied to the tops of the corresponding $U$ steps on the original Schr\"oder path. 
The sum $\sum_{\alpha} b(\alpha)$ is the sum over all $H$ steps $\alpha$ in the Schr\"oder path of the statistic $b(\alpha)$, 
which represents the number of red dots above $\alpha$ in the diagram.\label{fig:three:three}}
\end{figure}
\end{example}

We demonstrate that the statistic $\topplebounce_{ITC}$ is closely related to the 
$\schbounce$ statistic on Schr\"oder path via the bijection $\phi$ of 
Dukes~\cite{ncf} composed with the mirror map $\mu$ on words. 
The mirror map is defined as follows:
given $w=w_1w_2\ldots w_k$, let 
\begin{equation*}
\mu(w) := \mu(w_k) ~ \mu(w_{k-1}) ~ \ldots ~ \mu(w_1),
\end{equation*}
where $\mu(U):=D$, $\mu(H):=H$, and $\mu(D):=U$.

\begin{theorem}\label{prop:itc-haglund}
Given a sorted configuration $c\in \ClassOneParams{n}{d}$, 
let $w=\mu \circ \phi^{-1}(c) \in\Schroder_{n,d}$.
Then 
  \[ 
\area(w) = \level(c)
	\mbox{ and }  
	 \schbounce(w) = \topplebounce_{ITC}(c)-(n+d).\]
\end{theorem}

It is important to observe that the $\schbounce$ statistic is taken of the reverse of the Schr\"oder word in the above theorem. 
If we were to present this result without the notion of a reverse word, then it would require redefining the bounce statistic of 
Egge et al.~\cite{ehkk} {\em{or}} redefining the bijection of Dukes~\cite{ncf}.

\begin{example}\label{example:three:four}
Consider $c=(7,6,5,2,1; 5,4,4) \in \ClassOneParams{5}{3}$ that features in Examples~\ref{itctopplingexample} and \ref{example:three:three}.
We have 
\begin{align*}
w &= \mu \circ \phi^{-1}(c) \\
&= \mu (UHUDUHHDUDUDD) \\
&= UUDUDUHHDUDHD.
\end{align*}
\begin{figure}
    \begin{tikzpicture}[scale=0.75]
		\begin{scope}[xshift=0,yshift=0]
       	\draw (0,0) grid (8,8);
       	\draw[line width=1] (0,0) -- (8,8);
       	\draw[line width=3,blue] (0,0) -- (0,2) -- (1,2) -- (1,3) -- (2,3) -- (2,4) -- (3,5) -- (4,6) -- (5,6) -- (5,7) -- (6,7) -- (7,8) -- (8,8);
       	\foreach \x/\y in {1/1,2/2,3/3,3/4,4/4,4/5,5/5,6/6,7/7}{
			     	\fill[white!50!black] (\x-0.05 ,\y+0.05 ) -- +(0,0.75) -- +(-0.75,0) -- cycle;
       	}
       	\foreach \x/\y in {0/1,1/3,5/7}{
         	\fill[red] (\x,\y) circle (0.15);
       	}
		\node[above] at (6,10) {$\PeakSchroder(1)$};
		\draw[->,gray,line width=1.5] (6,10) to[out=-90,in=135] (4.8,7.2);
		\node[above] at (3.5,9) {$\PeakSchroder(2)$};
		\draw[->,gray,line width=1.5] (3.5,9) to[out=-100,in=135] (0.8,3.2);
		\node[above right] at (0,8) {$\PeakSchroder(3)$};
		\draw[->,gray,line width=1.5] (1.5,8.1) to[out=-90,in=135] (-0.2,1.2);
		\node[above left] at (2.75,4.5) {\tiny$\alpha_3$};
		\node[above left] at (3.75,5.5) {\tiny$\alpha_2$};
		\node[above left] at (6.75,7.5) {\tiny$\alpha_1$};
		\node at (4,-1) {$w= UUDUDUHHDUDHD.$};
		\node at (4,-1.75) {$\area(w)=9$};
		\node at (4,-2.5) {$b(\alpha_1)+b(\alpha_2)+b(\alpha_3) = 0+1+1=2$};
		\node at (4,-3.25) {$\schbounce(w)=4+2=6$};
		\end{scope}
		\begin{scope}[xshift=300,yshift=40]
       	\draw (0,0) grid (5,5);
       	\draw[line width=1] (0,0) -- (5,5);
       	\draw[line width=3,blue] (0,0) -- (0,2) -- (1,2) -- (1,3) -- (2,3) -- (2,4) -- (3,4) -- (3,5) -- (5,5);
       	\draw[line width=1.5,dashed,red] (5,5) -- (3,5) -- (3,3) -- (1,3) -- (1,1) -- (0,1) -- (0,0);
       	\draw[line width=1.5,red] (3.7,5.2) -- (3.5,5) -- (3.7,4.8);
       	\foreach \x/\y in {0/1,1/3,3/5}{
         	\fill[red] (\x,\y) circle (0.15);
       	}
		\node[above] at (4,7) {$\PeakDyck(1)$};
		\draw[->,gray,line width=1.5] (4,7) to[out=-90,in=90] (3,5.2);
		\node[above] at (2,6) {$\PeakDyck(2)$};
		\draw[->,gray,line width=1.5] (2,6) to[out=-100,in=135] (0.8,3.2);
		\node[above] at (0,5) {$\PeakDyck(3)$};
		\draw[->,gray,line width=1.5] (-0.2,5.1) to[out=-90,in=135] (-0.2,1.2);
		\node at (2.5,-1) {$C(w)$};
		\node at (2.5,-1.75) {$\mathrm{bounce}(C(w)) = 1+3=4$};
		\end{scope}
     \end{tikzpicture}
\caption{The Schr\"oder path $w= UUDUDUHHDUDHD\in \Schroder_{5,3}$ from Example~\ref{example:three:four}.\label{fig:three:four}} 
\end{figure}
We can now verify the statement of Theorem~\ref{prop:itc-haglund} for $c=(7,6,5,2,1; 5,4,4)$.
Notice that from Figure~\ref{fig:three:four} we have $\area(w)=9$. 
The quantity 
\begin{align*}
\level(c)=~
\height(c)-\left({n+d\choose 2} - {d\choose 2}\right) 
	&= 34 - \left( \binom{8}{2} - \binom{3}{2} \right) \\
	&= 34-(28-3) = 9 = \area(w),
\end{align*}
thereby verifying the first statement. 
Secondly, we have $\schbounce(w)=4+2=6$ from Figure~\ref{fig:three:four}. 
Since the ITC toppling is considered for this $c$ in Example~\ref{itctopplingexample}, we have 
$\topplebounce_{ITC}(c) = 14$.
These values now verify the second statement:
\begin{equation*}
\topplebounce_{ITC}(c)-(n+d) = 14-(5+3) = 6=\schbounce(w).
\end{equation*}
\end{example}

We split the proof of Theorem~\ref{prop:itc-haglund} into Theorem~\ref{lem:itc-area} and Theorem~\ref{lem:itc-bounce}, each being related to one of the two statistic equalities.
An illustration of the proof of Theorem~\ref{prop:itc-haglund} is given in Figure~\ref{fig:haglund-bounce}, 
and this is applied to the same path that appears in Haglund~\cite[Sec.1]{haglund-qtschroder}. 
An immediate corollary to Theorem~\ref{prop:itc-haglund} is the equality of the $q,t$-ITC polynomial with the $q,t$-Schr\"oder polynomial.

\begin{corollary} \label{corol:itc} For all $n\geq 1$ and $d\geq 0$, we have 
\[
\begin{array}{lcl}\displaystyle 
	\notrepoly_{n,d}^{ITC}(q,t) 
	&  	=  & \displaystyle \sum_{w \in \Schroder_{n,d}} q^{\area(w)}t^{\schbounce(w)}  
		=\langle\nabla e_{n+d},e_{n}h_d \rangle\end{array},\]
  a polynomial known to be symmetric in $q$ and $t$.
\end{corollary}

In Proposition~\ref{qt:symmetry} we present an accessible proof of this symmetry in $q$ and $t$ that relies on a combinatorial interpretation 
of $\langle\nabla e_{n+d},e_{n}h_d\rangle $ due to Haglund~\cite{haglund:monograph}.

Let us mention that there is an extension of part of this corollary in the recent paper 
D'Adderio, Dukes, Iraci, Lazar, Le~Borgne, and Vanden Wyngaerd et al.~\cite[Theorem 2.9]{ddillw} 
that is not dependent upon a lattice path interpretation.

  \begin{figure}[ht!]
    \begin{minipage}{8cm}
    \begin{tikzpicture}[scale=0.75]
      \draw (0,0) grid (8,8);
      \draw[line width=0.5] (0,0) -- (8,8);
      \draw[line width=3,blue] (0,0) -- (0,1) -- (1,2) -- (1,3) -- (2,3) -- (2,4) -- (4,6) -- (5,6) -- (5,7) -- (6,7) -- (6,8) -- (8,8);
      \draw[line width=3,blue!30!white] (8,8) -- (8,9) -- (9,9);
      \foreach \x/\ux in {7/7_C,6/7_C,5/6_C,4/5_C,3/3_I,2/3_I,1/2_C,0/1_I}{
        \draw node at (\x+0.5,8.3) {$\ux$};
      }
      \draw node at (8.5,8.3) {$s$};
      \foreach \x/\y/\r in {0/0/1_*,1/2/2_*,2/3/3_*,5/6/4_*,6/7/5_*}{
        \draw node[blue] at (\x-0.3,\y+0.5) {$\r$};
      }
      \draw[line width=1.5,red,dashed] (9,9) -- (8,9) -- (8,8) -- (6,8) -- (6,6) -- (4,6) -- (2,4) -- (2,2) -- (1,2) -- (0,1) -- (0,0);
      \draw[line width=1.5,red] (7.1,8.2) -- (6.9,8) -- (7.1,7.8);
      \foreach \x/\y/\mylsc/\myli/\iloop in {0/1/1/1/3,2/4/2/2/2,6/8/2/0/1,8/9/1/0/0}{
        \draw[line width=0.5,red] (\x,\y) -- (\x,10);
        \fill[red] (\x,\y) circle (0.15);
        \draw node[anchor=west] at (\x,11) {$q'_{\iloop}=\myli$}; 
        \draw node[anchor=west] at (\x,10.2) {$p'_{\iloop}=\mylsc$}; 
      }
      \foreach \x/\y in {3/1,4/1,4/4}{
          \fill[orange,rounded corners] (\x-0.1,\y+0.1) -- (\x-0.7,\y+0.1) -- (\x-0.1,\y+0.7) -- cycle; 
        }
      \foreach \x/\y in {2/1,2/2,5/1,5/2,5/3,5/4,5/5,6/1,6/2,6/3,6/4,6/5,6/6,7/1,7/2,7/3,7/4,7/5,7/6,7/7,8/1,8/2,8/3,8/4,8/5,8/6,8/7}{
          \fill[green!50!brown,rounded corners] (\x-0.1,\y+0.1) -- (\x-0.7,\y+0.1) -- (\x-0.1,\y+0.7) -- cycle; 
        }
      \foreach \x/\y in {1/1,3/2,3/3,3/4,4/2,4/3,4/5}{
          \fill[magenta!80!blue,rounded corners] (\x-0.1,\y+0.1) -- (\x-0.7,\y+0.1) -- (\x-0.1,\y+0.7) -- cycle; 
        }
      \end{tikzpicture}
    \end{minipage}
    \begin{minipage}{6.3cm}
	\ \\[1.5em]
      
      $c^{[0]} = (7_C,7_C,6_C,5_C,2_C;3_I,3_I,1_I)$
	\smallskip

      Topple sink:
	\smallskip

      $c^{[1]} = (8_C,8_C,7_C,6_C,3_C;4_I,4_I,2_I)$
	\smallskip

      Loop $0$: Topple no independent.
	\smallskip

      $c^{[2]} = c^{[1]}$
	\smallskip
      
      Loop $0$: Topple clique $7_C,7_C$.
	\smallskip
      
      $c^{[3]} = (1_C,1_C,9_C,8_C,5_C;6_I,6_I,4_I)$
	\smallskip

      Loop $1$: Topple independent $3_I,3_I$.
	\smallskip

      $c^{[4]} = (3_C,3_C,11_C,10_C,6_C;0_I,0_I,4_I)$
	\smallskip

      Loop $1$: Topple clique $6_C,5_C$.
	\smallskip

      $c^{[5]} = (5_C,5_C,3_C,2_C,8_C;2_I,2_I,6_I)$
	\smallskip

      Loop $2$: Topple independent $1_I$.
	\smallskip

      $c^{[6]} = (6_C,6_C,4_C,3_C,9_C;2_I,2_I,0_I)$
	\smallskip

      Loop $2$: Topple clique $2_C$.
	\smallskip

      $c^{[7]} = (7_C,7_C,6_C,5_C,2_C;3_I,3_I,1_I)$
    \end{minipage}
	\caption{Configuration $c=(7_C,7_C,6_C,5_C,2_C;3_I,3_I,1_I) \in \ClassOneParams{5}{3}$ 
	and its related Schr\"oder word $w=UHUDUHHDUDUDD$ that corresponds to the example Schr\"oder path in Haglund~\cite[Sec.1]{haglund-qtschroder}.
	The peaks are indicated with red dots and the bounce path is the red dotted line.
	To the right we follow the ITC-toppling process for $c$. 
	Note that the degree of a clique vertex is $n+d=8$ while the degree of an independent vertex is $n+1=6$.
	The green, magenta, and orange triangles are explained in the proof of Theorem~\ref{lem:itc-area}. 
	Green triangles correspond to clique vertices/columns, while the other two colours correspond to independent vertices/columns. 
	A triangle associated with an independent vertex will be orange if there are horizontal $H$ steps on the path in both its column and row.
	Otherwise those triangles are magenta.
 	\label{fig:haglund-bounce}\label{fig:two}
	}
  \end{figure}

\subsection{Mapping configuration height to Schr\"oder path area}

By considering the definition of $\mu \circ \phi^{-1}$ it is evident that, in the word $w$, 
the step that crosses each column encodes the number of grains on a vertex in the recurrent configuration.
More precisely, a column crossed by a $D$ step corresponds to a clique vertex and its number 
of grains is the number of rows below this horizontal step, less one.
A column crossed by a $H$ step corresponds to an independent vertex whose number of grains 
is the number of vertical $(U)$ steps in rows below (or, equivalently, to the left of) this diagonal $H$ step. 
(These are indicated in Figure~\ref{fig:haglund-bounce} by the blue $1_*$, $2_*$, \ldots, labels.)

\begin{theorem}\label{lem:itc-area}
Given $c\in\ClassOneParams{n}{d}$, let $w=\mu \circ \phi^{-1}(c)\in \Schroder_{n,d}$.
Then 
\[\area(w) = 
\level(c).
\]
\end{theorem}

\begin{proof}
Let $c\in\ClassOneParams{n}{d}$ with $c=(a_1,\ldots,a_n;b_1,\ldots ,b_d)$ and set $w=\mu \circ \phi^{-1}(c)$.
We will prove that $\mathsf{level}(\phi \circ \mu(w)) = \mathsf{area}(w)$ where 
$\mathsf{level}(c)=\mathsf{height}(c)-{n+d \choose 2}+{d \choose 2}$ is the number of grains in $c$ minus the number of edges non-incident to the sink in $\Split_{n,d}$. 
We prefer to do this by first proving it for words with first letter $U$, and then use this result to prove it for words with first letter $H$.

If $w$ starts with the letter $U$ (see Figure~\ref{fig:haglund-bounce})
we have two ways to count the number of lower triangles
in all the squares below the Schr\"oder path excluding those on the first row:
\begin{equation*}
\sum_{k=1}^{n} a_k   +   \sum_{k=1}^{d} b_k + {d \choose 2} = \mathsf{area}(w)+{n+d\choose 2}.
\end{equation*}
On the left hand side, the first sum corresponds to those lower triangles in columns corresponding to clique vertices 
(these are the green lower triangles in Figure~\ref{fig:haglund-bounce}).
More precisely, in the column of a clique vertex (which are those with a $D$ step), the green triangles also count the number of grains given by the definition of $\mu \circ \phi^{-1}$.

The remainder of the left-hand-side concerns columns that correspond to independent vertices (those columns crossed by a $H$ step). 
There are two cases to consider here.
\begin{itemize}
\item If a lower triangle has a $H$ step beside it, then let us call it a {\it magenta lower triangle}. 
Likewise, if a lower triangle has a $H$ step in its column and a $U$ step in its row, then also call it a magenta lower triangle. 
\item If a lower triangle has a $H$ step in its column and a different $H$ step in its row, then we call it an {\it orange lower triangle}. 
\end{itemize}
Note that there will be ${d \choose 2}$ orange triangles since there are $d$ $H$ steps and every pair of distinct $H$ steps gives rise to an orange lower triangle.
The number of magenta triangles corresponds to $ \sum_{k=1}^{d} b_k$.
On the right-hand side, the lower triangles are either counted by $\mathsf{area}(w)$ for those above the main diagonal and ${n+d \choose 2}$ for those below. 
The identity may be rewritten:
\begin{align*}
\area(w) 
	= & \displaystyle \sum_{k=1}^{n} a_k + \sum_{k=1}^{d} b_k + {d \choose 2} - {n+d\choose 2}
	=  \displaystyle \height(c)-\left ({n+d\choose 2} -{d \choose 2}\right) = \level(c).
\end{align*}
  
Alternatively, if $w$ begins with a $H$ then set $w=H^jw'$ where $j\geq 1$ and $w'$ begins with the letter $U$.
For this case $\area(w)=\area(w')$ and there are $j$ independent vertices having $0$ grains that we can simply ignore in 
$\sum_{k=1}^{d} b_k=\sum_{k=1}^{d-j} b_k$. An example of this case is illustrated in Figure~\ref{fig:bouncebis}.

  \begin{figure}[ht!]
    \begin{tikzpicture}[scale=0.75]
      \draw (-3,-3) grid (8,8);
      \draw[line width=0.5] (-3,-3) -- (8,8);
      \draw[line width=3,blue] (-3,-3) -- (0,0) -- (0,1) -- (1,2) -- (1,3) -- (2,3) -- (2,4) -- (4,6) -- (5,6) -- (5,7) -- (6,7) -- (6,8) -- (8,8);
      \foreach \x/\ux in {7/10_c,6/10_c,5/9_c,4/8_c,3/3_i,2/3_i,1/5_c,0/1_i,-1/0_i,-2/0_i,-3/0_i}{
        \draw node at (\x+0.5,8.3) {$\ux$};
      }
      \foreach \x/\y/\r in {0/0/1_*,1/2/2_*,2/3/3_*,5/6/4_*,6/7/5_*}{
        \draw node[blue] at (\x-0.3,\y+0.5) {$\r$};
      }
      \draw[line width=1.5,red,dashed] (8,8) -- (6,8) -- (6,6) -- (4,6) -- (2,4) -- (2,2) -- (1,2) -- (0,1) -- (0,0);
      \draw[line width=1.5,red] (7.1,8.2) -- (6.9,8) -- (7.1,7.8);
      \foreach \x/\y in {0/1,2/4,6/8}{
        \fill[red] (\x,\y) circle (0.15);
      }
      \foreach \x/\y in {3/1,4/1,4/4}{
          \fill[orange,rounded corners] (\x-0.1,\y+0.1) -- (\x-0.7,\y+0.1) -- (\x-0.1,\y+0.7) -- cycle; 
        }
      \foreach \x/\y in {2/1,2/2,5/1,5/2,5/3,5/4,5/5,6/1,6/2,6/3,6/4,6/5,6/6,7/1,7/2,7/3,7/4,7/5,7/6,7/7,8/1,8/2,8/3,8/4,8/5,8/6,8/7}{
          \fill[green!50!brown,rounded corners] (\x-0.1,\y+0.1) -- (\x-0.7,\y+0.1) -- (\x-0.1,\y+0.7) -- cycle; 
        }
      \foreach \x/\y in {1/1,3/2,3/3,3/4,4/2,4/3,4/5}{
          \fill[magenta!80!blue,rounded corners] (\x-0.1,\y+0.1) -- (\x-0.7,\y+0.1) -- (\x-0.1,\y+0.7) -- cycle; 
        }
      \foreach \x/\y in {-1/-2,0/-2,0/-1}{
          \fill[blue,rounded corners] (\x-0.1,\y+0.1) -- (\x-0.7,\y+0.1) -- (\x-0.1,\y+0.7) -- cycle; 
        }
      \foreach \x/\y in {1/-2,1/-1,1/0,3/-2,3/-1,3/0,4/-2,4/-1,4/0}{
          \fill[brown,rounded corners] (\x-0.1,\y+0.1) -- (\x-0.7,\y+0.1) -- (\x-0.1,\y+0.7) -- cycle; 
        }
      \foreach \x/\y in {2/-2,2/-1,2/0,5/-2,5/-1,5/0,6/-2,6/-1,6/0,7/-2,7/-1,7/0,8/-2,8/-1,8/0}{
          \fill[white!50!black,rounded corners] (\x-0.1,\y+0.1) -- (\x-0.7,\y+0.1) -- (\x-0.1,\y+0.7) -- cycle; 
        }        
      \end{tikzpicture}
  \caption{A word having prefix $H$. Here $w=HHHw'$ so that $j=3$.\label{fig:bouncebis}}
  \end{figure}

Again, we can count the number of lower triangles beneath $w$ in two different ways. 
The only difference to the previous counting argument is that there are now $j$ initial $H$ steps 
and these give rise to three new lower triangle types to which we associate the colours blue, brown, and grey.
Blue lower triangles are those that are formed as a result of the initial $j$ $H$ steps (the number of these will be ${j\choose 2}$) of $w$.
Brown lower triangles correspond to columns that are traversed by a $H$ step and have one of the initial $j$ $H$ steps in the same row. 
Grey lower triangles correspond to columns that are traversed by a $D$ step and have one of the initial $H$ steps in the same row.
This gives the equation:
\begin{align*}
      \lefteqn{{j \choose 2}+j(n-j) + \sum_{k=1}^{n} (a_k-j)  +   jn   + \sum_{k=1}^{d-j} b_k   + {n-j \choose 2}\qquad\qquad} \\
      = & \area(w')+{j \choose 2}+ j(n-j) + jd + {n+d-j\choose 2}.
\end{align*}
After simplifying we obtain:
\begin{equation*}
\sum_{k=1}^{n} (a_k-j) + \sum_{k=1}^{d-j} b_k + {n-j \choose 2} = \area(w')+ {(n-j)+d\choose 2}.
\end{equation*}
This expression corresponds to the identity for the previous case for configuration 
\begin{equation*}
c':= (a_1-j,\ldots,a_n-j; b_1,b_2,\ldots,b_{d-j}) = \mu \circ \phi^{-1}(w').
\end{equation*}
By the first case, we have $\level(c')=\area(w')$. 
Since the number of grains in $c$ and $c'$ differs by $jd$, which is also the difference of the number of edges non-incident to sink between $\Split_{n,d}$ and $\Split_{n,d-j}$, 
we deduce that $\level(c)=\level(c')$. 
We already observed that $\area(w)=\area(w')$. 
Hence $\level(c)=\level(c')=\area(w')=\area(w)$.
\end{proof}

\subsection{Mapping $\topplebounce_{ITC}$ to the statistic $\schbounce$}

\begin{theorem}\label{lem:itc-bounce} 
Given $c\in \ClassOneParams{n}{d}$,
let $w=\mu \circ \phi^{-1}(c) \in\Schroder_{n,d}$.
Then \[ \topplebounce_{ITC}(c)-(n+d) = \schbounce(w).\]
\end{theorem}

   \begin{figure}[ht!]
     \begin{tikzpicture}[scale=0.75]
       \fill[white!80!black] (0,5) -- (0,3) -- (3,0) -- (5,0) -- cycle;
       \fill[white!80!black] (0,9) -- (0,7) -- (7,0) -- (9,0) -- cycle;
       \fill[white!80!black] (4,9) -- (9,4) -- (9,6) -- (6,9) -- cycle;
       \fill[white!80!black] (7,9) -- (9,7) -- (9,9) -- (9,9) -- cycle;
       \draw[line width=6,orange] (0,2) -- (0,4) -- (2,6) -- (4,6);
       \draw[line width=6,orange] (4,7) -- (6,9) -- (9,9);
       \draw[line width=6,green!50!brown] (10,10) -- (9,10) -- (9,9) -- (7,7) -- (4,7) -- (4,4) -- (2,2) -- (0,2) -- (0,0);
       \draw (0,0) grid (9,9);
       \draw[line width=1] (0,0) -- (10,10);
       \draw[line width=3,blue!30!white] (10,10) -- (9,10) -- (9,9);
       \draw[line width=3,blue] (9,9) -- (8,8) -- (7,8) -- (6,7) -- (4,7) -- (4,6) -- (3,6) -- (2,5) -- (1,5) -- (1,4) -- (0,3) -- (0,0);
       \draw[line width=1.5,dashed,red] (10,10) -- (9,10) -- (9,9) -- (8,8) -- (7,8) -- (6,7) -- (4,7) -- (4,6) -- (4,5) -- (3,4) -- (3,3) -- (2,3) -- (1,2) -- (0,2) -- (0,0);
       \draw[line width=1.5,red] (9.6,10.2) -- (9.4,10) -- (9.6,9.8);
       \foreach \x/\y in {9/10,4/7,0/2}{
         \fill[red] (\x,\y) circle (0.15);
       }
       \draw[line width=1,<->,black!50!brown] (7.2,6.8) -- (9.2,8.8);
       \draw node[black!50!brown,anchor=west] at (8.2,7.8) {$q'_1=2$};
       \draw[line width=1,<->,black!50!brown] (4,6.7) -- (7,6.7);
       \draw node[black!50!brown,anchor=north] at (5.5,6.7) {$p'_1=3$};
       \draw[line width=1,<->,black!50!brown] (4.2,6) -- (4.2,4);
       \draw node[black!40!brown,anchor=west] at (4.2,5.5) {$p'_1-1$};
       \draw[line width=1,<->,black!50!brown] (2.2,1.8) -- (4.2,3.8);
       \draw node[black!50!brown,anchor=west] at (3.2,2.8) {$q'_2=2$};
       \draw[line width=1,<->,black!50!brown] (0,1.7) -- (2,1.7);
       \draw node[black!50!brown,anchor=north] at (1,1.7) {$p'_2=2$};
     \end{tikzpicture}
     \begin{tikzpicture}[scale=0.75]
       \draw[line width=2,white!80!black] (3.5,-0.5) -- (-0.5,3.5);
       \draw[line width=2,white!80!black] (5.5,-0.5) -- (-0.5,5.5);
       \draw[line width=2,white!80!black] (3.5,5.5) -- (5.5,3.5);
       \draw[line width=2,white!80!black] (4.5,5.5) -- (5.5,4.5);
       \draw (0,0) grid (5,5);
       \draw[line width=1] (0,0) -- (6,6);
       \draw[line width=3,blue!30!white] (6,6) -- (5,6) -- (5,5);
       \draw[line width=3,blue] (5,5) -- (2,5) -- (2,4) -- (0,4) -- (0,0);
       \draw[line width=1.5,dashed,red] (6,6) -- (5,6) -- (5,5) -- (2,5) -- (2,2) -- (0,2) -- (0,0);
       \draw[line width=1.5,red] (5.6,6.2) -- (5.4,6) -- (5.6,5.8);
       \foreach \x/\y in {5/6,2/5,0/2}{
         \fill[red] (\x,\y) circle (0.15);
       }
     \end{tikzpicture}
     \caption{Reformulation of the bounce path directly on Schr\"oder paths instead of on the Dyck path subword formed by deleting $H$ steps. 
	Configuration $c=(7,6,6,5,4;5,5,4,3)\in \ClassOneParams{5}{4}$ maps to $c'=(4,4,4,3,3)$ that is recurrent on $\Split_{5,0}=K_6$. 
	Notice 
	$\PeakDyck(2)=(0,p'_2)$ and $\PeakSchroder(2)=(0,p'_2)$;
       $\PeakDyck(1)=(p'_2,p'_2+p'_1)$ and $\PeakSchroder(1)=(p'_2+q'_2,p'_2+p'_1+q'_2)$; 
       $\PeakDyck(0)=(p'_2+p'_1,p'_2+p'_1+p'_0)$ and $\PeakSchroder(0)=(p'_2+p'_1+q'_2+q'_1,p'_2+p'_1+p'_0+q'_2+q'_1)$
	}
	\label{fig:reformulate-bounce} 
	\label{fig:four}
   \end{figure}

\begin{proof}
Recall that there are a total of $n+d+1$ vertices in the graph $\Split_{n,d}$. 
Suppose $c\in\ClassOneParams{n}{d}$ with 
\begin{equation*}
\ToppleRecord_{ITC}(c) = (Q'_1,P'_1,\ldots,Q'_k,P'_k).
\end{equation*}
We think of this toppling process as consisting of $k+1$ loops where loop $i$ is the toppling of vertices $Q'_i$ followed by $P'_i$. 
(The $0$th loop corresponds to the initial toppling of the sink and $Q'_0=\emptyset$, $P'_0=\{s\}$.)
Its ITC-toppling sequence is
$$((q'_1,\ldots,q'_k),~ (p'_1,\ldots,p'_k)),$$ 
and recall from Definition~\ref{twosix} the convention $q'_0:=0$ and $p'_0=1$.
Let us define a sequence of pairs $$(\Peakloop(0), \Peakloop(1),\ldots, \Peakloop(k))$$ that we think of as points in the plane
where $\Peakloop(i):=(x_i,y_i)$ and
\begin{itemize}
\item $x_i$ is the number of untoppled graph vertices at the end of the $k$-th loop iteration, i.e. just after the clique vertices of $P'_i$ have been parallel-toppled,
\item $y_k$ is the number of untoppled vertices in the middle of the $k$-th loop iteration, i.e. just after the independent vertices of $Q'_i$ have been parallel-toppled.
\end{itemize}
By definition of the ITC-toppling sequence, we have
\begin{equation*}
x_i = n+d+1-\sum_{j=0}^{i}(q'_j+p'_j),
\mbox{ and } 
y_i =  n+d+1-\left(\sum_{j=0}^{i-1} q'_j+p'_j \right)  - q'_i.
\end{equation*}
In particular, 
\begin{align*}
\Peakloop(1) 
	= & (n+d-q'_1-p'_1,n+d-q'_1), \mbox{ and } \\
\Peakloop(2) = & (n+d-q'_1-p'_1-q'_2-p'_2,n+d-q'_1-p'_1-q'_2).
\end{align*}
As the ITC-toppling sequence induces a partition $(P'_i)_{i=0,\ldots,k}$ of sink and clique vertices: 
\begin{equation*}
\{s\} \cup V = P'_0 \cup P'_1 \cup \cdots \cup P'_k,
\end{equation*}
where we have $P'_0 = \{s\}$ and, for $i>0$, 
\begin{equation*}
P'_{i} = \left\{ v_{\ell} ~:~   \sum_{j=0}^{i-1} p'_j \leq \ell < \sum_{j=0}^{i} p'_j\right\}.
\end{equation*}

\newcommand{\neatmap}{\pi_{\mathsf{U},\mathsf{D}}}
Let $\neatmap$ be the map that keeps only the occurrences of letters $U$ or $D$ in any word, 
i.e. it removes all occurrences of $H$ for Schr\"oder words.
Define a {\em compressed configuration $c'$} on $S_{n,0}=\{s\}\cup V = K_{n+1}$ as 
\begin{equation*}
c' := \mathsf{compress}_{K_{n+1}}(c) := \phi \circ \neatmap  \circ \phi^{-1}(c),
\end{equation*}
where $\phi^{-1}$ is applied to $\Split_{n,d}$ whereas $\phi$ is applied to $\Split_{n,0}$.
This compressed configuration provides a connection between the bouncing process for the complete split graph and for the complete graph.
The configuration $c'$ is equivalently defined by: $c'_i+1$ is the number of occurrences of letters $U$ after the $i$-th occurrence of letter $D$ 
when the path is read from {\bf{north-east to south-west}}. 
See Figure~\ref{fig:reformulate-bounce} for an example.

If $w=\mu \circ \phi^{-1}(c)$ is the Schr\"oder path of configuration $c$, then $w'$ is the Dyck word of $c'$ obtained by deletion of counter-diagonals crossed by diagonal steps.
Haglund's~\cite[Sec.1]{haglund-qtschroder} bounce path for the Schr\"oder path $w$ is obtained from the bounce path of the 
Dyck path $w'$ by inserting diagonal steps in the bounce path parallel to (and in the same counter-diagonal) those of $w$.

The bounce path of $w'$ has bounces of size $p'_0=1$, $p'_1$, $p'_2$, $\ldots$, $p'_k$. 
This is because the toppling of the sink ensures $p'_0=1$. 
Similarly, there are precisely $p'_1$ many $D$ steps from the north-east point $(n,n)$ of $w'$ going to the left before encountering a $U$ step, 
as these correspond the clique vertices that become unstable as a result of toppling the sink. 
The point at which it meets the Dyck path is
\begin{align*}
\PeakDyck(1) = (n+1-p'_0-p'_1,n+1-p'_0).
\end{align*}
This bounce path then goes to $(n+1-p'_0-p'_1,n+1-p'_0-p'_1)$, and then goes left until encountering the top of the next $U$ step. 
There will be $p'_2$ many left steps since the interplay between the number of $U$ steps preceding a $D$ step 
and the new level at which vertices are unstable as a result of the previous $p'_1$ topplings sets the toppling threshold to be $p'_1$ less than it was previously. 
The point at which it meets the Dyck path is
\begin{align*}
\PeakDyck(2) = (n+1-p'_0-p'_1-p'_2,n+1-p'_0-p'_1).
\end{align*}
Iterating the argument gives the peaks of the bounce path of $w'$ as
\begin{equation*}
\PeakDyck(i) = \left(n+1-\sum_{j=0}^i p'_j,n+1-\sum_{j=0}^{i-1}p'_j\right) = \left( \sum_{j=i+1}^k p'_j, \sum_{j=i}^k p'_j\right).
\end{equation*}
Observe that the $x$ (resp. $y$) coordinate of $\PeakDyck(i)$ counts the number of yet-to-be toppled clique vertices at the end (resp. start) of the $i$th loop.

This leads to the same partition $(P'_i)_{i\geq 0}$ of $D$ steps in the Dyck word $w'$ as in the Schr\"oder word $w$.
Any peak is an endpoint of a vertical ($U$) step of the Dyck word.
The insertion of the deleted counter-diagonals preserves this property in Haglund's definition of peaks for a Schr\"oder word.

It remains to count the number of counter-diagonals inserted below each $\PeakDyck(i)$ peak.
By the ITC-toppling process, the peak ending bounce $p_i$ admits $\sum_{j=i+1}^k q'_j$ such counter-diagonals. 
These correspond to the independent vertices toppled after the end of the $i$-th loop iteration. Hence
\begin{align*}
\PeakSchroder(i) =& \PeakDyck(i) +\left( \sum_{j=i+1}^k q'_j\right)(1,1)\\
 =& \left(\sum_{j=i+1}^k p'_j+\sum_{j=i+1}^k q'_j, \sum_{j=i}^k p'_j+\sum_{j=i+1}^k q'_j\right),
\end{align*}
and this coincides with $\Peakloop(i)$ by rewriting and making use of the facts $n+1=\sum_{j=0}^k p'_j$ and $d=\sum_{j=0}^k q'_j$.
From Haglund~\cite[Eqn.(10)]{haglund-qtschroder} we have 
$\schbounce (w) =   \mathsf{bounce}(w')+\sum_{\alpha} b(\alpha)$. 
The consideration above allows us to write 
$\mathsf{bounce}(w') = \sum_i (i-1)p'_i$ and $\sum_{\alpha} b(\alpha) = \sum_i (i-1) q'_i$.
This gives 
\begin{equation*}
\schbounce (w) = \sum_{j=1}^k (j-1)(p'_j+q'_j).
\end{equation*}
From the definition of the ITC-toppling process bounce,
\begin{equation*}
\topplebounce_{ITC}(c) = \sum_{j=1}^k j(q'_j+p'_j).
\end{equation*}
It follows that $\topplebounce_{ITC}(c) = \schbounce (w) + \sum_{j=1}^k (p'_j+q'_j) = \schbounce (w) + (n+d)$.
\end{proof}

\begin{proof}[Proof of Theorem~\ref{prop:itc-haglund}]
Combine Theorems~\ref{lem:itc-area} and \ref{lem:itc-bounce}.
\end{proof}

Egge, Haglund, Killpatrick and Kremer~\cite[Theorem 1]{ehkk} gave the following explicit sum for $q,t$-Schr\"oder polynomials:
\begin{align} 
\displaystyle \sum_{w \in \Schroder_{n,d}} q^{\area(w)}t^{\schbounce(w)} = \sum_{k=1}^{n}  
		&
		\sum_{\substack{(\alpha_1,\ldots,\alpha_k)\iscomp_k n \\ (\beta_0,\ldots,\beta_k) \isweakcomp_{k+1} d}}
			\binom{\beta_0+\alpha_1}{\beta_0}_q \binom{\beta_k+\alpha_k-1}{\beta_k}_q 
			q^{\binom{\alpha_1}{2}+\ldots+\binom{\alpha_k}{2}}\nonumber \\
		& t^{\beta_1+2\beta_2+\ldots+k\beta_k+\alpha_2+2\alpha_3+\ldots+(k-1)\alpha_k} 
			\prod_{i=1}^{k-1}
				\binom{\beta_i+\alpha_{i+1}+\alpha_i-1}{\beta_i,\alpha_{i+1},\alpha_i-1}_q. 
	\label{egge:equation}
\end{align}
The equality between the $q,t$-ITC polynomial and the $q,t$-Schr\"oder polynomial established in Corollary~\ref{corol:itc} ensures that the above explicit sum
equals $\notrepoly^{ITC}_{n,d}(q,t)$.

We present here an alternative sum (that is of course equal to the above sum) that uses the classification of ITC-toppling sequences. 
The form of the sum is slightly different to that of the above equation in that the pairs of sequences over which we sum are different. 
However, the number of pairs of sequences that contribute to each sum is the same and is
\begin{align*}
\sum_{k=1}^{n} \binom{d+k}{d} \binom{n-1}{k-1}
\end{align*}
These enumerations are proven in Lemmas~\ref{itc:sequence:enumeration} and \ref{ehkk:sequence:enumeration}.

The set $\mathsf{ITC}_{n,d}$ indexes a partition of the set $\Schroder_{n,d}$.
This partition gives rise to an explicit sum involving the Gaussian binomial coefficients for the $q,t$-ITC polynomial.
Define the partial order $\leq$ on Schr\"oder paths in $\Schroder_{n,d}$ by $w \leq w'$ if all lower triangles of $w$ are also lower triangles of $w'$.

\begin{prop}\label{prop:countingFqt}
Let $a_0:=1$.
For all $n\geq 1$ and $d\geq 1$,
\begin{equation}
\notrepoly^{ITC}_{n,d}(q,t) = 
\sum_{k=1}^{n+1}
\sum_{\substack{((b_1,\ldots,b_k),(a_1,\ldots,a_{k}))\\ \in \mathsf{ITC}_{n,d,k}}}
\prod_{i=1}^{k} q^{{a_i\choose 2}} 
			{ a_i+b_i+a_{i-1}-1\choose a_i,b_i, a_{i-1}-1}_q t^{(i-1)(a_i+b_i)}.
\label{derycke:equation}
\end{equation}
\end{prop}

A proof of this proposition can be found in the appendix.

\begin{example}\label{compelling}
To derive $\notrepoly^{ITC}_{2,2}(q,t)$, all sequences in $\mathsf{ITC}_{2,2}$ appear in Example~\ref{itc22}.
For each pair $[b,a]$ we list the term contributing to the sum for $\notrepoly^{ITC}_{2,2}(q,t)$ in 
Table~\ref{helpfultable}.
\begin{table}[!h]
$$
\begin{array}{c|c||c|c}
	\begin{array}{c}
	\left[(b_1,\ldots,b_k),~(a_1,\ldots,a_k)\right]\\ ~\in~\mathsf{ITC}_{2,2} 
	\end{array}
	& \begin{array}{c} \mbox{contributing}\\ \mbox{term} \end{array}  &
	\begin{array}{c}
	\left[(b_1,\ldots,b_k),~(a_1,\ldots,a_k)\right]\\ ~\in~\mathsf{ITC}_{2,2} 
	\end{array}
	& \begin{array}{c} \mbox{contributing}\\ \mbox{term} \end{array}  \\
	\hline\hline 
\left[(2),~(2)\right] & q \binom{4}{2}_q &
	\left[(0,2),~(1,1)\right] & t^3 \binom{3}{1}_q  \\ 
\left[(2,0),~(1,1)\right] & t\binom{3}{1}_q &
	\left[(1,0,1),~(1,1,0)\right] & t^3 \binom{2}{1}_q  \\ 
\left[(1,1),~(2,0)\right] & qt \binom{2}{1}_q \binom{3}{1}_q &
	\left[(0,1,1),~(1,1,0)\right] & t^4 \binom{2}{1}_q \\ 
\left[(1,1),~(1,1)\right] & t^2 \binom{2}{1}^2_q  &
	\left[(0,0,2),~(1,1,0)\right] & t^5  \\
\left[(0,2),~(2,0)\right] & qt^2 \binom{3}{1}_q  &\multicolumn{2}{c}{}
\end{array}
$$
\caption{Terms contributing to $\notrepoly^{ITC}_{2,2}(q,t)$ discussed in Example~\ref{compelling}\label{helpfultable}}
\end{table}
\end{example}

\section{Sorted recurrent configurations as sawtooth polyominoes}\label{sectionfive}

In this section we present a planar characterization of the sorted recurrent configurations from Section~\ref{sectiontwo}.
This was done for the complete bipartite graph in \cite{dlb} and parallelogram polyominoes were shown to uniquely encode the sorted recurrent configurations of the ASM on that graph.
Define the unit steps $\ss=(0,-1)$, $\ww=(-1,0)$, $\nn=(0,1)$, $\ee=(1,0)$ and the two diagonal steps $\nw = (-1,1)$ and $\se=(1,-1)$.

\subsection{Sawtooth polyominoes and Schr\"oder paths}

\begin{definition}\label{sawdef}
Given positive integers $n$ and $d$, define 
$\wonkypoly_{n,d}$ to be the set of pairs of paths $(Upper,Lower)$ in the plane with the following properties:
\begin{itemize}
\item The upper path $Upper$ is a path from $(0,0)$ to $(n,d)$ that takes steps in the set $\{\nn,\se\}$.
\item The lower path $Lower$ is a path from $(0,0)$ to $(n,d)$ that takes steps in the set $\{\ee, \nn\}$.
\item The lower path and upper path only touch at the end-points $(0,0)$ and $(n,d)$.
\end{itemize}
We will refer to such a pair of paths, or the planar object they define, as a {\it{\wonky}} of dimension $(n,d)$.
\end{definition}

\begin{example}\label{firstexamplewonky}
Two \wonkys\ in $\wonkypoly_{4,5}$:
\begin{center}
\begin{tabular}{c@{\qquad\qquad\qquad}c}
\newpolyone & 
\newpolytwo\\
(a) & (b)
\end{tabular}
\end{center}
\end{example}

Let $\Words_{n,d}$ be the set of words consisting of $n$ $U$'s, $n$ $D$'s, and $d$ $H$'s. 
Consider the following construction on $\ClassOneParams{n}{d}$.

\begin{definition}\label{bijdef}
Let $w=w_1 w_2 \ldots w_{2n+d} \in \Words_{n,d}$. 
Form a collection of steps in the plane as follows:
\begin{description}
\item[The upper path] Connect $(n+1,d)$ to $(n,d+1)$ with a $\nw$ step. 
Read $w$ from left to right. For every $U$ letter encountered draw a $\nw$ step and for every non-$U$ letter encountered draw a $\ss$ step. 
At the end of this process draw a final $\ss$ step so as to touch the origin. 
Let $\Upper(w)$ be this path from $(n+1,d)$ to $(0,0)$.
\item[The lower path] Read $w$ from left to right. 
Starting from $(n+1,d)$, for every $H$ letter encountered in $w$ draw a $\ss$ step and for every $D$ letter encountered in $w$ draw a $\ww$ step. 
At the end of this process connect the final point $(1,0)$ to the origin.
Let $\Lower(w)$ be this path from $(n+1,d)$ to $(0,0)$.
\end{description}
Denote by $\sts_{n,d}(w)$ the pair $(\Upper(w),\Lower(w))$.
\end{definition}

\begin{example}\label{word:to:poly:example}
Consider $w=HUHD HUHDUDUHD \in \Schroder_{4,5}$ which corresponds to the configuration $c=(7,4,2,1;4,4,3,3,1) \in  \ClassOneParams{4}{5}$.
The upper and lower paths $\Upper(w)$ and $\Lower(w)$ are illustrated in Figure~\ref{yvan:new}
\begin{figure}
  \begin{center}
  \begin{tikzpicture}[scale=1.25]
    \draw (-0.2,-0.2) grid (5.2,6.2);
    \draw[rounded corners, line width=2, blue,->] (5,5) -- (4,6) -- (4,5) -- (3,6) -- (3,3) -- (2,4) -- (2,2) -- (1,3) -- (1,2) -- (0,3) -- (0,0); 
    \draw[rounded corners, line width=2, red, ->] (5,5) -- (5,3) -- (4,3) -- (4,1) -- (2,1) -- (2,0) -- (0,0);
    \foreach \x/\y/\letter in {4/5.5/$H_1$,3.5/5.5/$U_2$,3/5.5/$H_3$,3/4.5/$D_4$,3/3.5/$H_5$,2.5/3.5/$U_6$,2/3.5/$H_7$,2/2.5/$D_8$,1.5/2.5/$U_9$,1/2.5/$D_{10}$,0.5/2.5/$U_{11}$,0/2.5/$H_{12}$,0/1.5/$D_{13}$}{
      \fill[fill=white,draw=blue,line width=1] (\x,\y) circle (0.2);
      \draw[blue] node[scale=0.6] at (\x,\y) {\letter};
      }
    \foreach \x/\y/\letter in {5/4.5/$H_1$,5.2/4/$U_2$,5/3.5/$H_3$,4.5/3/$D_4$,4/2.5/$H_5$,4.2/2/$U_6$,4/1.5/$H_7$,3.5/1/$D_8$,3/0.8/$U_9$,2.5/1/$D_{10}$,1.9/1.1/$U_{11}$,2/0.5/$H_{12}$,1.5/0/$D_{13}$}{
      \fill[fill=white,draw=red,line width=1] (\x,\y) circle (0.2);
      \draw[red] node[scale=0.6] at (\x,\y) {\letter};
      }

  \end{tikzpicture}
  \caption{The sawtooth polyomino related to $w=H_1U_2H_3D_4H_5U_6H_7D_8U_9D_{10}U_{11}H_{12}D_{13}$
	that is discussed in Example~\ref{word:to:poly:example}
	\label{yvan:new} }
  \end{center}
\end{figure}
Note that this produces the same polyomino that appears in Figure~\ref{firstexamplewonky}(a).
\end{example}

\begin{theorem}\label{bijone}
A word $w \in \Schroder_{n,d}$ iff $\sts_{n,d}(w) \in \wonkypoly_{n+1,d}$.
\end{theorem}

\begin{proof}
Let $w \in \Schroder_{n,d}$. 
By the construction in Definition~\ref{bijdef}, the paths $\Upper(w)$ and $\Lower(w)$ are both paths from $(n+1,d)$ to $(0,0)$ whose step types correspond to those in Definition~\ref{sawdef}.
We will prove:
\begin{enumerate}
\item[(i)] If $w\in \Schroder_{n,d}$ then $\sts_{n,d}(w) \in \wonkypoly_{n+1,d}$.
\item[(ii)] If $w\in \Words_{n,d} \backslash \Schroder_{n,d}$ then $\sts_{n,d}(w) \not\in \wonkypoly_{n+1,d}$.
\end{enumerate}
For (i), suppose that $w\in \Schroder_{n,d}$. 
Then $\sts_{n,d}(w)=(\Upper(w),\Lower(w))$ with 
\begin{equation*}
\Upper(w) = (a_1,\ldots,a_{2n+2+d}) \mbox{ and }\Lower(w) = (b_1,b_2,\ldots,b_{n+1+d}).
\end{equation*}
Note that $a_i \in \{\nw,\ss\}$ and $b_j \in \{\ww,\ss\}$.
Let us assume that the upper and lower paths meet at the point one reaches by starting at $(n,d)$ and following steps $ (a_1,\ldots,a_{p+2q})$ and $(b_1,\ldots,b_{p+q})$ where 
the upper path contains $q$ $\nw$ steps and $p+q$ $\ss$ steps and the lower path contains $p$ $\ss$ steps and $q$ $\ww$ steps. 
The point at which these two partial paths meet will be $(n+1-q,d-p)$.
By Definition~\ref{bijdef}, $a_1=\nw$ and for $i \in [1,p+2q-1]$, $a_{1+i} = \nw$ iff $w_i=U$ and $a_{1+i} = \ss$ iff $w_i = H$ or $D$.
Similarly, $b_j=\ss$ iff the $j$th non-$U$ letter of $w$ is $H$ and $b_j=\ww$ iff the $j$th non-$U$ letter of $w$ is $H$. 
These facts show that for $(a_2,\ldots,a_{p+2q}) = (w_1,\ldots,w_{p+2q-1})$ there are precisely $p$ $H$ steps, $q$ $D$ steps, and $(q-1)$ $U$ steps. 
The number of $D$'s in this prefix of a Schr\"oder word is one more than the number of $U$'s, thereby contradicting the definition of a Schr\"oder word. 
Hence these two paths cannot meet and $\sts_{n,d}(w) \in \wonkypoly_{n+1,d}$.

For (ii), let $w\in \Words_{n,d} \backslash \Schroder_{n,d}$. 
This means $w$ is a word consisting of $n$ $U$'s, $n$ $D$'s, and $d$ $H$'s that violates the Schr\"oder path property. 
In other words, there exists an index $k$ such that the prefix $w_1w_2\cdots w_k$ contains more $D$'s than $U$'s. 
Let $k$ be the smallest index that satisfies this property
so that there are $t$ $U$'s, $t+1$ $D$'s, and $(k-2t-1)$ $H$'s in the prefix.
Consider now $\Upper(w_1w_2\cdots w_k) = (a_1,\ldots,a_{k+1})$ and $\Lower(w_1w_2\cdots w_k)=(b_1,b_2,\ldots,b_{\ell})$.
The upper path will end at position $(n,d+1) + t(-1,1) + (k-2t-1+t+1) (0,-1) = (n-t,d+1+t-(k-t)) = (n-t,d+1+2t-k)$.
The lower path will end at position $(n+1,d)+(t+1)(-1,0) + (k-2t-1)(0,-1) = (n-t,d+1+2t-k)$, and so the upper and lower paths touch. 
This implies $\sts_{n,d}(w) \not\in \wonkypoly_{n+1,d}$.
\end{proof}

\subsection{Mapping sorted recurrent configurations to sawtooth polyominoes}
The composition of the bijections relating sorted recurrent configurations to Schr\"oder words and Schr\"oder words 
to sawtooth polyominoes allows us to give a direct mapping from sorted recurrent configurations to sawtooth polyominoes 
that we now state.

\begin{definition}\label{rec2poly}
Let $c=(a_1,\ldots,a_{n};b_1,\ldots,b_d)\in \ClassOneParams{n}{d}$.
\begin{enumerate}
\item[(i)] For each $i \in [1,d]$, draw a vertical line segment from $(1+b_{d+1-i},i-1)$ to $(1+b_{d+1-i},i)$.
\item[(ii)]  Connect the endpoints of the vertical line segments in the previous step if they have the same $x$-coordinate. 
In the case of an endpoint with $y$-coordinate 0, connect it to the origin. 
In the case of an endpoint with $y$-coordinate $d$, connect it to the point $(n+1,d)$.
\item[(iii)] For each $j \in [1,n]$, draw a diagonal line segment from $(j-1,3-j+a_{n+1-j})$ to $(j,2-j+a_{n+1-j})$.
Draw a diagonal line segment from $(n,d+1)$ to $(n+1,d)$.
\item[(iv)]  Connect the endpoints of the diagonal line segments in the previous step if they have the same $x$-coordinate. 
In the case of an endpoint with $x$-coordinate 0, connect it to the origin. 
\end{enumerate}
Let us denote by $Low$ the path that results from steps (i) and (ii), and by $Upp$ the path that results from steps (iii) and (iv).
Denote by $\fcp_{n,d}(c)$ the pair $(Upp,Low)$.
\end{definition}

\begin{example}
Consider $c=(7,4,2,1;4,4,3,3,1) \in \ClassOneParams{4}{5}$.
In this example $n=4$ and $d=5$.
We first draw the vertical line segments $(1+1,0) \to (1+1,1)$, $(1+3,1) \to (1+3,2)$, $(1+3,2) \to (1+3,3)$, $(1+4,3) \to (1+4,4)$, and $(1+4,4) \to (1+4,5)$.
Then connect those vertical line segments that have endpoints with the same $y$-coordinate. 
Connect $(1+1,0)$ to the origin $(0,0)$ and connect  $(1+4,4)$ to $(5,5)$. 
Secondly, draw the diagonal segments $(0,2+1) \to (1,1+1)$, $(1,1+2) \to (2,0+2)$, $(2,0+4) \to (3,-1+4)$, $(3,-1+7) \to (4,-2+7)$. Also connect $(n=4,d+1=6) \to (n+1=5,d=5)$.
Then connect those diagonal segments whose endpoints have the same $x$-coordinate. 
Connect $(0,2+1)$ to the origin.
The upper and lower paths $\Upper(w)$ and $\Lower(w)$ are illustrated in Example~\ref{firstexamplewonky}(a).
\end{example}

\subsection{CTI topplings and bounce paths within sawtooth polyominoes}
The polyomino representation of a sorted recurrent configuration allows us to visualize the CTI toppling process in a rather compact way. 
This is similar to the bounce path that featured in the paper on parallelogram polyominoes~\cite{dlb}.

\begin{definition}
Given $P \in  \wonkypoly_{n+1,d}$, let the {\it{CTI-bounce path of $P$}} be the path from $(n,d)$ to $(0,0)$ that takes steps in $\{\ss,\nw\}$ and is defined as follows.
Start at $(n,d)$.
\begin{enumerate}
\item[(i)] If we are currently on the upper path then go to step (ii). Otherwise move in direction $\nw$ until meeting the upper path.
\item[(ii)] If the point is already on the lower path then go to step (iii). Otherwise move in the direction $\ss$ until meeting the lower path. 
\item[(iii)] If the current position is not $(0,0)$, then go to step (i).
\end{enumerate}
The outcome will be a sequence of $p_1$ $\nw$ steps, followed by $p_1+q_1$ $\ss$ steps, 
followed by $p_2$ $\nw$ steps, followed by $p_2+q_2$ $\ss$ steps, $\ldots$,
followed by $p_k$ $\nw$ steps, followed by $p_k+q_k$ $\ss$ steps.
Define 
\begin{equation*}
\ctiBounce(P):=(p_1,q_1,\ldots,p_k,q_k) \qquad \mbox{ and }\qquad \ctibounce(P) := \sum_{i=1}^k i(p_i+q_i) .
\end{equation*}
\end{definition}

\begin{example}\label{ert3}
The bounce paths of the two \wonkys\ given in Example~\ref{firstexamplewonky} are illustrated in Figure~\ref{hjkl3}.
Note that in (a) the bounce math initially took 0 $\nw$ steps since it was already at a point on the upper path.
\begin{figure}
\begin{center}
\begin{tabular}{c@{\qquad}c}
\newpolyonebouncealt & 
\newpolytwobouncealt\\
(a) $\ctiBounce(\cdot) = (0,2,1,2, 1,0, 1,1, 1,0)$ &
(b) $\ctiBounce(\cdot) = (1,3, 1,1, 2,1)$ \\
\end{tabular}
\end{center}
\caption{The CTI bounce paths for the sawtooth polyominoes mention in Example~\ref{ert3}\label{hjkl3}}
\end{figure}
\end{example}

\begin{definition}
Given $P \in  \wonkypoly_{n+1,d}$, define the area $\area(P)$ of $P$ to be the number of unit squares whose vertices are lattice points that are contained within $P$. 
\end{definition}

The area of the \wonky\ in Example~\ref{firstexamplewonky}(a) is 12 while
the area of the \wonky\ in Example~\ref{firstexamplewonky}(b) is 15.

\begin{theorem}\label{euroone}
Let $c \in \DecRec(\Split_{n,d})$ and $P=\fcp_{n,d}(c) \in \wonkypoly_{n+1,d}$.
Then $\height(c) = \area(P)-(n+d) +\binom{n+d}{2}-\binom{d}{2}$ and 
$\bouncepath_{CTI}(c) = \ctiBounce(P)$. 
\end{theorem}

\begin{proof}
Suppose $c=(a_1,\ldots,a_n;b_1,\ldots,b_d) \in \DecRec(\Split_{n,d})$ and $P=(Upp,Low)=\fcp_{n,d}(c) \in \wonkypoly_{n+1,d}.$
From the construction for $\fcp$, we find that the total area of the unit squares contained 
beneath the lower path $Low$ of $P$ and the $x$-axis is 
\begin{equation*}
d(n+1-1) - (b_1+b_2+\ldots+b_d).
\end{equation*}
The total area beneath the upper path $Upp$ of $P$ and the $x$-axis is 
\begin{equation*}
d+n - \binom{n}{2} + (a_1+\ldots + a_{n}).
\end{equation*}
The area of $P$ is the difference of these:
\begin{equation*}
\area(P) = d+n  - \binom{n}{2} + (a_1+\ldots + a_{n}) - dn + (b_1+b_2+\ldots+b_{d}).
\end{equation*}
This may be rewritten 
\begin{align*}
\height(c) &= \area(P) - (d+n) +\binom{n}{2} +dn \\
&=  \area(P) - (d+n) +\binom{d+n}{2} - \binom{d}{2}.
\end{align*}

Let $c \in \DecRec(\Split_{n,d})$ and $P=\fcp_{n,d}(c) \in \wonkypoly_{n+1,d}$.
Consider $\bouncepath_{CTI}(c)=(p_1,q_1,\ldots)$. 
In terms of the associated sawtooth polyomino, the effect of toppling the sink corresponds to all diagonal steps 
(other than the last between $(n,d+1)$ and $(n+1,d)$ as it does not represent a height) 
on the upper path being shifted up by 1 unit.
Those that are unstable will be those that are then on (or above) the line $x+y=n+1+d$. 

Equivalently, instead of shifting each of the diagonal steps by 1, to see which will become unstable as a result of toppling the sink we can 
start at the point $(n,d)$ and see which diagonal steps to its left are on or above the line $x+y=n+1+d-1=n+d$. Suppose there are $p_1$ such steps. 
These $p_1$ steps can be seen by following a line from $(n,d)$ to $(n-p_1,d+p_1)$ which is where the line $x+y=n+d$ meets the upper path of the sawtooth polyomino. 
As $p_1$ clique vertices are toppled, this means the heights of all independent vertices will now be $1+p_1$ more than before the sink was toppled. 

Consequently, the set of those that are next toppled are those independent vertices $w_i$ whose height is now $b_i+1+p_1 \geq n+1$. This is the same as those independent vertices 
whose initial height (before toppling the sink) was $b_i \geq n-p_1$, and these independent vertices are precisely those that are strictly to the right of the line
that moved vertically down from $(n-p_1,d+p_1)$ until meeting the lower path of the sawtooth polyomino.

Repeating this argument reveals precisely why the bounce path of the polyomino models the CTI toppling process on the corresponding recurrent configuration. 
Hence $\bouncepath_{CTI}(c)=(p_1,q_1,\ldots) = \ctiBounce(P)$.
\end{proof}

\begin{example}\label{secondexamplewonky}
The \wonky\ diagrams that correspond to the 30 sorted recurrent configuration given in 
Example~\ref{examplelala} are illustrated in Figure~\ref{secfig}. 
Above each diagram we list the sorted configuration to which it corresponds in parentheses, 
and below the parentheses we list the bounce $(p_1,q_1,\ldots)$ in angle brackets.
\end{example}

\begin{figure}[!h]
\begin{center}
\setlength\tabcolsep{1.5pt} 
\begin{tabular}{|c|c|c|c|c|c|c|c|} \hline
\hreeone& 
\hreetwo &
\hreethree &
\hreefour &
\hreefive &
\hreesix &
\hreeseven & 
\hreeeight 
\\ \hline
\hreenine &
\hreeten &
\hreeeleven &
\hreetwelve &
\hreethirteen &
\hreefourteen &
\hreefifteen &
\hreesixteen \\ \hline
\hreeseventeen &
\hreeeighteen &
\hreenineteen &
\hreetwenty &
\hreetwentyone &
\hreetwentytwo &
\hreetwentythree &
\hreetwentyfour \\ \hline
\multicolumn{1}{c|}{}&
\hreetwentyfive &
\hreetwentysix &
\hreetwentyseven &
\hreetwentyeight &
\hreetwentynine &
\hreethirty  \\ \cline{2-7}
\end{tabular}
\end{center}
\caption{The sawtooth polyominoes for Example~\ref{secondexamplewonky}. 
For each, we list the recurrent configuration $c=(a_1,a_2;b_1,b_2) \in \DecRec(S_{2,2})$, 
followed by $\bouncepath_{CTI}(c)$ that is presented using angle brackets}
\label{secfig}
\end{figure}

\subsection{ITC topplings and bounce paths within sawtooth polyominoes}

Given the manner in which the bounce path of the previous section emulates the temporal changes 
to the configuration heights as a result of toppling the sink, it is perhaps unsurprising that 
ITC topplings also admit a very similar polyomino bounce path description.
For the case of ITC topplings, the only difference is that the bounce path initially moves south 
(if it can) rather than moving north-west from position $(n,d)$.

\begin{definition}
Given $P \in  \wonkypoly_{n+1,d}$, let the {\it{ITC-bounce path of $P$}} be the path from $(n,d)$ to $(0,0)$ that takes steps in $\{\ss,\nw\}$ and is defined as follows.
Start at $(n,d)$.
\begin{enumerate}
\item[(i)] If we are currently on the lower path then go to step (ii). Otherwise move in direction $\ss$ until meeting the lower path.
\item[(ii)] If the point is already on the upper path then go to step (iii). Otherwise move in the direction $\nw$ until meeting the upper path. 
\item[(iii)] If the current position is not $(0,0)$, then go to step (i).
\end{enumerate}
The outcome will be a sequence of $q'_1$ $\ss$ steps, followed by $p'_1$ $\nw$ steps, 
followed by $p'_1+q'_2$ $\ss$ steps, followed by $p'_2$ $\nw$ steps, $\ldots$,
followed by $p'_{k-1}$ $\nw$ steps, followed by $p'_{k-1}+q'_k$ $\ss$ steps.
For consistency we define $p'_k:=0$.
Define 
\begin{equation*}
\itcBounce(P):=(q'_1,p'_1,\ldots,q'_k,p'_k) \qquad \mbox{ and }\qquad \itcbounce(P) := \sum_{i=1}^k i(q'_i+p'_i) .
\end{equation*}
\end{definition}

\begin{example}\label{asdfg2}
The ITC-bounce paths of the two \wonkys\ given in Example~\ref{firstexamplewonky} are illustrated in Figure~\ref{hjk2}.
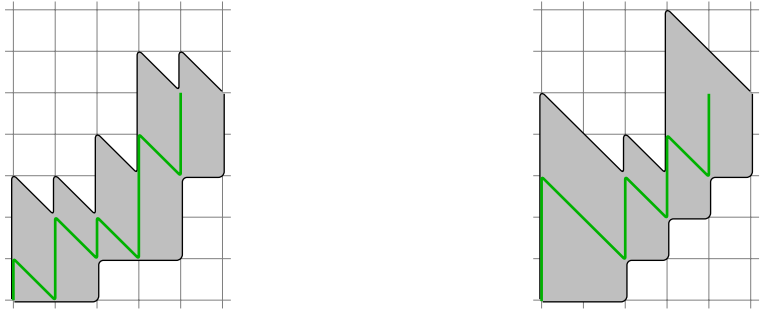
\begin{figure}
\begin{center}
\begin{tabular}{c@{\qquad}c}
\newpolyonebouncealttwo & 
\newpolytwobouncealttwo\\
(a) $\itcBounce(\cdot) = (2,1, 2,1, 0,1, 1,1)$ &
(b) $\itcBounce(\cdot) = (2,1, 1,1, 1,2, 1,0)$ \\
\end{tabular}
\end{center}
\caption{ITC bounce paths mentioned in Example~\ref{asdfg2}\label{hjk2}}
\end{figure}
\end{example}

\begin{theorem}
Let $c \in \DecRec(\Split_{n,d})$ and let $P=\fcp_{n,d}(c) \in \wonkypoly_{n+1,d}$.
It follows that 
$\bouncepath_{ITC}(c) = \itcBounce(P)$
and 
$\topplebounce_{ITC}(c) := \itcbounce(P)$.
\end{theorem}

\begin{proof}
Let $c=(a_1,\ldots,a_n;b_1,\ldots,b_d)\in \DecRec(\Split_{n,d})$ and $P=\fcp_{n,d}(c) \in \wonkypoly_{n+1,d}$.
Consider $\bouncepath_{ITC}(c)=(q'_1,p'_1,\ldots)$. 
In terms of the associated sawtooth polyomino, the effect of toppling the sink corresponds to all vertical steps 
on the lower path being shifted right by 1 unit.
Those that are unstable will be those that are then on (or to the right of) the line $x=n+1$.

Equivalently, instead of shifting each of the vertical steps to the right by 1, to see which will become unstable as a result of toppling the sink we can 
start at the point $(n,d)$ and see which vertical steps to its right are on or to the right of the line $x=n$. Suppose there are $q'_1$ such steps. 
These $q'_1$ steps can be seen by following a line from $(n,d)$ to $(n,d-q'_1)$ which is where the line $x=n$ meets the lower path of the sawtooth polyomino. 
As $q'_1$ independent vertices have just been toppled, this means the heights of all clique vertices will now be $1+q'_1$ more than before the sink was toppled. 

Consequently, the set of those that are next toppled are those clique vertices $v_i$ whose height is now $a_i+1+q'_1 \geq n+d$. 
This is the same as those clique vertices 
whose initial height (before toppling the sink) was $a_i \geq n+d-q'_1-1$, and these clique vertices are precisely those whose horizontal steps are on, or above, the line
of slope $-1$ that passes through the point $(n,d-q'_1)$ until meeting the upper path of the sawtooth polyomino.

Repeating this argument reveals precisely why the ITC bounce path of the polyomino models the ITC toppling process on the corresponding recurrent configuration. 
Hence $\bouncepath_{ITC}(c)=(q'_1,p'_1,\ldots) = \itcBounce(P)$. From this it follows that $\topplebounce_{ITC}(c) := \itcbounce(P)$.
\end{proof}

\section{A cycle lemma to count sorted recurrent configurations on split graphs}\label{sectionsix}

Our aim in this section is to derive a cycle lemma (Theorem~\ref{ourcyclelemma})  for configurations of the ASM on the 
split graph that will allow us to count the number of sorted recurrent configurations.
Recall that in the introduction we outlined the notion of a cycle lemma and the general framework to which it applies.
While one of the authors has already counted these configurations, via the previously described bijection 
with Schr\"oder paths~\cite{ncf}, our secondary aim is to provide a cycle lemma that remains within the 
framework of the sandpile model rather than for some graphical representation of the configurations. 

Since such a cycle lemma has already been shown to exist for complete bipartite graphs~\cite{aadl}, 
we suspect that such results are instances of a more general result that holds true for the sandpile model 
on similar graph classes, such as the clique-independent graphs of \cite{ddillw} which have the complete bipartite and complete split graphs as special cases. 
At the start of this paper, we defined configurations on the complete split graph to be an assignment 
of non-negative integers to the (non-sink) vertices.  
In this section, we will allow configurations, unless otherwise stated, to have a non-negative number of grains.

In this section we will decompose a configuration $u=\left(u_s;u^{[K]};u^{[I]}\right)$ on the complete 
split graph $S_{n,d}$ into three sub-configurations on three components: 

\begin{itemize}
\item the isolated sink $u_s$, 
\item the sub-configuration $u^{[K]}=\left(u^{[K]}_0,\ldots, u^{[K]}_{n-1}\right)$ on the $n$ clique ($K=V$) vertices, and
\item the sub-configuration $u^{[I]}=\left(u^{[I]}_0,\ldots, u^{[I]}_{d-1}\right)$ on the $d$ independent ($I=W$) vertices.
\end{itemize}
We adopt the convention that the number of grains on the sink is fixed so that
\[ u_s ~:=~ -\left( \sum_{k\in K} u^{[K]}_k + \sum_{i\in I} u^{[I]}_i \right).\]
This allows us to remove explicit reference to the sink part of a configuration and to write 
$u=\left(u^{[K]};u^{[I]}\right)$ instead of $u=\left(u_s;u^{[K]};u^{[I]}\right)$.

We now introduce the set of sorted non-negative quasi-stable configurations.
The reason for doing so is that this set is easier to count and we will later show it can be partitioned into sets, 
each of cardinality $n+1$, that are indexed by sorted recurrent configurations.

We will say a clique (resp. independent) sub-configuration $u^{[K]}$ (resp. $u^{[I]}$) is \emph{quasi-stable} 
if $\max_{k \in K}  u^{[K]}_k \leq n+d$ (resp. $\max_{i\in I} u^{[I]}_i \leq n$).
Notice that the quasi-stable condition corresponds to the stable condition for the sub-configuration $u^{[I]}$, 
whereas a quasi-stable sub-configuration $u^{[K]}$ may have one more grain on each of its constituent vertices than can feature in a stable sub-configuration.
We will call a configuration \emph{quasi-stable} if both its clique and independent sub-configurations are quasi-stable. 

A permutation $\sigma^{[C]}$ on vertices of a component $C\in \{K,I\}$ acts on a sub-configuration $u^{[C]}$ by permuting the distribution of its grains:
\[ \sigma^{[C]}.u^{[C]} ~:=~ \left(u^{[C]}_{\sigma^{[C]}(0)},\ldots,\sigma^{[C]}_{\sigma^{[C]}(|C|-1)}\right).\]
We will call a sub-configuration $u^{[C]}$ on the component $C$ \emph{sorted} if $\left(u^{[C]}_0, \ldots, u^{[C]}_{|C|-1}\right)$ is a weakly decreasing sequence.
Every sub-configuration $u^{[C]}$ is equivalent, by some permutation $\sigma_{sort}^{[C]}$ of its vertices, to a single \emph{sorted sub-configuration} denoted $\mathsf{sort}(u^{[C]})$.
By extension, every configuration $u$ is equivalent to a single \emph{sorted configuration} $\mathsf{sort}(u)$, under the action of permutations on $K$ and $I$ that sort the individual parts, and we denote this by $(\sigma_{sort}^{[K]},\sigma_{sort}^{[I]}).u$, 
i.e. 
\begin{equation*}
\mathsf{sort}(u) ~=~ (\sigma_{sort}^{[K]},\sigma_{sort}^{[I]}).u.
\end{equation*}

We will say that two configurations $u$ and $u'$ are {\em toppling-and-permuting equivalent}, written $u\equiv_{\sigma.\Delta} u'$, if \[u' = (\sigma^{[K]},\sigma^{[I]}).u + \sum_{v\in S_{n,d}} a_v\Delta^{(v)}\] for some permutation $\sigma^{[K]}$ permuting the vertices of the clique component $K$ and some permutation $\sigma^{[I]}$ permuting the vertices of the independent component $I$. 
Here $\Delta^{(v)}$ is the action of toppling vertex $v$ and $(a_v)_v\in \mathbb{Z}^{1+n+d}$ gives the possibly negative number of topplings (or anti-topplings) of each vertex of $S_{n,d}$ including the sink. 

\begin{theorem}[Cycle Lemma]\label{ourcyclelemma}
Every sorted recurrent configuration on $S_{n,d}$ is toppling-and-permuting equivalent to exactly $n+1$ of 
the ${2n+d \choose n}{n+d \choose n}$ sorted quasi-stable non-negative configurations (including itself but no other sorted recurrent configuration).
\end{theorem}

The proof of the above result requires several technical results to first be proven and will be presented at the end of this section.
This partition of sorted quasi-stable non-negative configurations into equivalence classes having the same size yields the following corollary.

\begin{corollary}
There are precisely $\frac{1}{n+1}{2n+d \choose n}{n+d \choose n} = \frac{1}{2n+1}{2n+1\choose n}{ 2n+d \choose d}$ sorted recurrent configurations on $S_{n,d}$.
\end{corollary}

Next we will analyse the equivalence classes of the toppling-and-permuting relation $\equiv_{\sigma.\Delta}$ through some new operators on configurations.
The \emph{topple-max-then-sort operator}, $T_C$, of a component $C$ acts on a configuration $u$ as
  \[ T_C.u = \mathsf{sort}(u+\Delta_c), \]
 where $c$ is one of the vertices for which $u_c = \max_{c'\in C} u_{c'}$, i.e. it is one of the vertices having the maximal number of grains in the component $C$.

These operators are, in general, not reversible either on configurations (due to sorting considerations) or even on sorted configurations.
However, these operators become reversible on the subset of sorted {\em compact} configurations (to be defined next) that contain sorted quasi-stable non-negative configurations (which itself includes all the recurrent configurations):\medskip
$$\left\{\mbox{sorted recurrent}\right\}  ~ \subseteq ~ \left\{ \mbox{sorted quasi-stable and non-negative} \right\} ~  \subseteq ~ \left\{\mbox{sorted compact}\right\}.$$

\begin{definition}
A configuration $u$ on $S_{n,d}$ is \emph{compact} if 
\begin{equation*}
\mathsf{spread}(u^{[K]}) \leq n+d+1 \quad \mbox{ and }\quad  \mathsf{spread}(u^{[I]}) \leq n+1,
\end{equation*}
where the \emph{spread} $\mathsf{spread}(u'^{[C]})$ in a component $C$ of a configuration $u'$ is defined by
\[\mathsf{spread}(u'^{[C]}) := \max_{v\in C} u'_v - \min_{v \in C} u'_v.\]
\end{definition}

There are two permutations of vertices of a component $C$ that will be of special interest to us.
We denote by $\iota_{[C]}$ the identity permutation and by $\tau_{[C]}$ the cycle permutation defined by $\tau_{[C]}(c):=c-1\mod |C|$.
The cycle permutation simply moves the vertex in first position $0$ to the last position $|C|-1$ and shifts the others.

We see each sub-configuration $u^{[C]}$ as a vector in $\mathbb{Z}^{|C|}$.
This allows addition of vectors such as $u^{[C]}+v^{[C]}$ for two sub-configurations.
For any constant $k$, we define the constant sub-configuration $k^{[C]}:=(k)_{c\in C}$.
Similarly, for any constant $k$ and vertex $v_C \in C$ we define the sub-configuration $k^{[v_C]}$ that contains $k$ grains on vertex $v_C$ and $0$ elsewhere. 

The following proposition details the behaviour of the three toppling operators on the set of sorted compact configurations. This will be implicitly used in the lemmas that follow it.

\begin{proposition}
For a sorted compact configuration $u=(u^{[K]};u^{[I]})$ on $S_{n,d}$ we have
  \[
	\begin{array}{lcl}
             T_s.u & = & \left(u^{[K]}+1^{[K]};u^{[I]}+1^{[I]}\right) \\[0.7em]
             T_K.u & = & \left(\tau_{[K]}.\left(u^{[K]}-(n+d+1)^{[0_K]}+1^{[K]}\right);u^{[I]}+1^{[I]}\right) \\[0.7em]
             T_I.u & = & \left(u^{[K]}+1^{[K]};\tau_{[I]}.\left(u^{[I]}-(n+1)^{[0_I]}\right)\right), 
           \end{array}
	\]
  and all three of these are themselves compact configurations.
  These operators restricted to sorted compact configurations are reversible with 
  \[
	\begin{array}{lcl}
             T_s^{-1}.u & = & \left(u^{[K]}-1^{[K]};u^{[I]}-1^{[I]}\right) \\[0.7em]
             T_K^{-1}.u & = & \left(\left(\tau_{[K]}^{-1}.u^{[K]} \right)+(n+d+1)^{[0_K]}-1^{[K]}; u^{[I]}-1^{[I]}\right) \\[0.7em]
             T_I^{-1}.u & = & \left(u^{[K]}-1^{[K]};\left(\tau_{[I]}^{-1}.u^{[I]} \right)+(n+1)^{[0_I]}\right).
           \end{array}
	 \]
\end{proposition}

\begin{proof}
Let  $u=(u^{[K]};u^{[I]})$ be a sorted compact configuration on $S_{n,d}$.
The expression for $T_s.u$ is simply a reformulation of the description of the configuration $\mathsf{sort}(u+\Delta^{(s)})=(\iota_{[K]},\iota_{[I]}).(u+\Delta^{(s)})$. 
Moreover, the spread of each component is preserved by addition of constant vectors $1^{[C]}$, so the configuration remains compact.

Since the configuration $u$ is sorted, for both sub-configurations $C\in \{K,I\}$ the first vertex contains the maximal number of grains, i.e. $u^{[C]}_0=\max_{c\in C} u_c$.
As $u$ is compact, on toppling a vertex with the maximal number of grains it then becomes one of the vertices with the minimal number of grains in the component to which it belongs.
It means that the component of the toppled vertex may be sorted by the (cycle) permutation $\tau_{[C]}$. 
Finally this shows that the expression $T_{[C]}.u$ is an alternative description for $\mathsf{sort}(u+\Delta^{(0_C)})$.

Moreover, the maximal resulting spread occurs in the case of a second vertex $1_C$ having a maximal number of grains such that $u^{[C]}_1=u^{[C]}_0$.
It coincides, by design, with the bound on spread used in the definition of compact configurations. 

In the expressions 
\begin{equation*}
T_{[K]}.u = \left(\tau_{[K]},\iota_{[I]}\right).\left(u+\Delta^{(0_K)}\right) \mbox{  and  } T_{[I]}.u = \left(\iota_{[K]},\tau_{[I]}\right).\left(u+\Delta^{(0_I)}\right),
\end{equation*}
the actions on the configurations do not depend on the particular $u$.

Since all the steps are reversible, we deduce that max-toppled-then-sort operators are all injective.
The reverse map is defined by reversing these descriptions:
\[ T_{[K]}.u = \left(\left(\tau_{[K]}^{-1},\iota_{[I]}\right).u\right)-\Delta^{(0_K)} \mbox{ and } T_{[I]}.u = \left(\left(\iota_{[K]},\tau_{[I]}^{-1}\right).u\right)-\Delta^{(0_I)}.\]
It remains to check that those reverse operators map a compact configuration to another compact configuration.
In a component, the case of maximal spread corresponds to cases for which $u^{[C]}_{|C|-2}=u^{[C]}_{|C|-1}$.
Similar reasoning to the first part of the proof ensures that the result is also compact.
\end{proof}

The following proposition shows that the three toppling operators are very well-behaved on the set of sorted compact configurations and summarizes Dhar's criterion in terms of these operators. 

\begin{proposition}\label{lem:identity-operator}
The restrictions to sorted compact configurations of the operators $T_s$, $T_K$ and $T_I$ commute and \[T_sT_K^{|K|}T_I^{|I|}=\mathsf{Id}\] where $\mathsf{Id}$ is the identity map/operator.
\end{proposition}

\begin{proof}
Let $X$ and $Y$ be two distinct components among $\{K,I,\{s\}\}$.
We show that $T_X.T_Y.u=T_Y.T_X.u$ corresponds, for any sorted compact configuration $u$, to the toppling of the 
initially maximal vertex $x(u)=0_X\in X$ and $y(u)=0_Y\in Y$ in respective components followed by sorting each component.
Indeed, when $x(u)$ is toppled, $y(u)$ remains maximal among vertices of $Y$ since each vertex of this component received the same number of grains since the sorting only act of vertices of component $X$.
 Hence when $T_Y$ is applied, $y(u)=0_Y$ is still a maximal vertex, so it may be toppled and the following sorting acts only on the component $Y$.

So $T_Y.T_X.u$ is equivalent to the parallel toppling of $x(u)$ and $y(u)$ followed by sorting.
By symmetry of this argument, $T_X.T_Y.u$ is also equivalent to this parallel toppling then sorting so $T_Y.T_X=T_X.T_Y$, and the two operators commute.
We remark that if the toppled vertex were not $0_X$ but another vertex instead, so that $u_0^{[X]}=u_{x(u)}^{[X]}$, the result is the same after the sorting of component $X$. 
  
The operator $T_s$ causes the sink $s$ to topple once.
The operator $T_K^{|K|}$ causes every vertex of $K$ to topple exactly once, 
and so we have the equivalence \begin{equation*} T_K^{|K|} = \sum_{k\in K} \Delta^{(k)},  \end{equation*}
where we make use of the fact that $\tau_{[K]}^{|K|}$ leads to the identity permutation of this component.
Similarly, the operator $T_I^{|I|}$ causes every vertex of $I$ to topple exactly once, and this results in the equivalence
\begin{equation*} T_I^{|I|} = \sum_{i\in I} \Delta^{(i)},  \end{equation*}
where we make use of the fact that $\tau_{[I]}^{|I|}$ leads to the identity permutation of this component.
Finally, every vertex topples exactly once as a result of the operator $T_sT_K^{|K|}T_I^{|I|}$ as \begin{equation*} T_sT_K^{|K|}T_I^{|I|}.u = u+\sum_{v\in S_{n,d}} \Delta^{(v)} = u, \end{equation*} 
since, in the sandpile model, the action $\sum_{v\in S_{n,d}} \Delta^{(v)}$ has the net effect of leaving the configuration to which it is applied unchanged.
\end{proof}

Next we introduce the notion of the weight of a sub-configuration and introduce a weight operator on configurations. 
Some necessary properties of these will then be proven in Proposition~\ref{lem:weight_description_of_quasistables} which will 
be essential in the proof of Proposition~\ref{lem:counting-quasistable-nonnegative}.

\begin{definition} \label{defweights}
\begin{enumerate}
\item[]
\item[(a)] The \emph{weight} of a sub-configuration $u^{[K]}$ on the clique component $K$ of $S_{n,d}$ is
\[ \mathsf{weight}(u^{[K]}) :=  \sum_{k\in K} \left\lfloor \frac{ u^{[K]}_k}{n+d+1} \right\rfloor.\]
\item[(b)] The \emph{weight operator} $T_W$ is 
\[ T_W := T_K^{|K|+1}T_I^{|I|}.\]
\item[(c)] Two configurations $u$ and $v$ are \emph{$T_W$-equivalent}, $u \equiv_{T_W} v$, if $u = T_W^t.v$ for some $t\in \mathbb{Z}$. 
\end{enumerate}
\end{definition}

Note that the equivalence in Definition~\ref{defweights}(c) implies a divisibility by $(n+d+1)$ of the difference between the number of grains at the sink in both configurations.

\begin{proposition}\label{lem:weight_description_of_quasistables}
Let $u$ and $v$ be sorted compact configurations on $S_{n,d}$.
\begin{enumerate} 
\item[(a)] $u = T_W.v$ implies that $u^{[I]}=v^{[I]}$. More generally, $u \equiv_{T_W} v$ implies that $u^{[I]}=v^{[I]}$.
\item[(b)] The (sorted) compact sub-configuration $u^{[K]}$ on $S_{n,d}$ is quasi-stable and non-negative if and only if $ \mathsf{weight}(u^{[K]}) = 0$.
\item[(c)]  Moreover $\mathsf{weight}(T_W.u^{[K]}) = \mathsf{weight}(u^{[K]})-1$ and the quasi-stable and non-negative configuration $T_W$-equivalent to $u^{[K]}$ is $T_W^{\mathsf{weight}(u^{[K]})}.u^{[K]}.$ 
\end{enumerate}
\end{proposition}

\begin{proof}
(a)
Consider the sub-configuration $u^{[I]}$.
Since the restriction to the component $I$ of operators $T_s$ and $T_K$ are equivalent, in both simply add one grain to each vertex of $I$, we have
\[(T_W.u)^{[I]}=(T_K^{|K|+1}T_I^{|I|}.u)^{[I]} = (T_sT_K^{|K|}T_I^{|I|}.u)^{[I]}=(\mathsf{Id}.u)^{[I]}=u^{[I]},\]
with the penultimate equality coming from Proposition~\ref{lem:identity-operator}.
Hence $u^{[I]}=(T_W.u)^{[I]}$. More generally, and using the same reasoning, we have that $u \equiv_{T_W} v$ implies $u^{[I]}=v^{[I]}.$

\noindent (b)
As regards the sub-configuration $u^{[K]}$, we consider the {\em Euclidean division} of each entry of the sub-configuration $u^{[K]}$ defined as follows:
  \[u_k^{[K]} =: q_k(n+d+1)+r_k  \qquad \mbox{ for all }k \in K,\]
  where $(q_k)_{k \in K}$ is called the {\em quotient vector} and $(r_k)_{k\in K}$ is called the {\em remainder vector} $(r_k)_{k\in K}$.
  By design, the operator $T_{W}$ satisfies:
  \[ T_W.u = \tau_{[K]}.\left(u-(n+d+1)^{[0_K]}\right).\]
  In terms of the Euclidean division by $(n+d+1)$, the action of $T_W$ may be described directly on the quotient and remainder vectors (in that order) as follows:
  \[ T_W.(q_k)_{k\in K} = \tau.\left((q_k)_{k\in K}-1^{[0_K]}\right) \qquad\mbox{ and }\qquad T_W.\left(r_k\right)_{k\in K} = \tau.\left(r_k\right)_{k\in K}.\]
  Since $u^{[K]}$ is compact, we have $u_{|K|-1}^{[K]} \leq u_0^{[K]} \leq u_{|K|-1}^{[K]}+(n+d+1)$ and so 
  \[ (q_k)_{k\in K} = (q_0,q_0,....,q_0,q_0-1,\ldots,q_0-1) = (q_0)^p(q_0-1)^{n-p}\]
  for some $p\in \{1,\ldots, n\}$.
  Hence
\begin{equation*}
T_W.\left((q_0)^p(q_0-1)^{n-p}\right) = (q_0)^{p-1}(q_0-1)^{n-p+1}.
\end{equation*}
By applying the operator $T_W$ a total number of $\sum_{k\in K} q_k = \mathsf{weight}(u^{[K]})$ times, 
we obtain the configuration $u':= T_W^{\mathsf{weight}\left(u^{[K]}\right)} .u^{[K]}$ 
such that $\left(q'_k\right)_{k\in K}=(0)_{k\in K}$.
Thus $u'= \tau_{[K]}^{\mathsf{weight}(u^{[K]})}.(r_k)_{k\in K}$ is also non-negative by definition of the reminders via Euclidean division. 

\noindent (c) Moreover, for any power $t\in \mathbb{Z}$ other than $t= \mathsf{weight}(u^{[K]})$, we can consider $u''^{[K]}:=T_W^{t}.u^{[K]}$.
Since $\left(q^{\prime\prime} _k\right)_{k\in K} \neq (0)_{k\in K}$, 
either we have $q^{\prime\prime}_0 > 0$ in which case $u^{\prime \prime [K]}_0$ is not quasi-stable, 
or $q^{\prime\prime}_{|K|-1} < 0$ in which case we have $u^{\prime\prime[K]}_{|K|-1}$ is negative. 
\end{proof}

Finally we are in a position to state the main result of this section, and from which the proof of Theorem~\ref{ourcyclelemma} essentially follows.

\begin{proposition}\label{lem:counting-quasistable-nonnegative}
Every sorted compact configuration is toppling-and-permuting equivalent to exactly $n+1$ sorted quasi-stable and non-negative configurations. 
\end{proposition}

\begin{proof}
Let $u$ be a sorted compact configuration and consider $v := \mathsf{rec}(u)$, the sorted recurrent configuration which is toppling-and-permuting equivalent to $u$.
Apply Dhar's criterion, Prop.~\ref{pro:DharBurning}, to $v$.
We start from $v^{[0]}=v$ which is a quasi-stable and non-negative configuration.
We topple the sink once using the operator $T_s$ and suppose $k_0$ applications of $T_I$ are needed to reach a sub-configuration on $I$ that is stable, and call the configuration $v^{[1]}$.
Following this, we topple a clique vertex by applying $T_K$ and suppose that $k_1$ applications of $T_I$ to $v^{[1]}$ are required to reach a sub-configuration on $I$ that is stable, and call the configuration $v^{[2]}$, which is again quasi-stable.
  Since Dhar's criterion for recurrent configurations may be decomposed as
  \[ T_sT_I^{k_0}\prod_{i=1}^n (T_KT_I^{k_i}) \]
  we obtain from $v$, via each prefix of operators before an occurrence of $T_s$ or $T_K$, precisely  $n+1$ configurations $(v^{[i]})_{i=0,\ldots, n}$ where the $I$ component is quasi-stable and non-negative.
By the $T_W$-equivalence described in Proposition~\ref{lem:weight_description_of_quasistables}, we deduce $n+1$ quasi-stable non-negative configurations $(w^{[i]})_{i=0,\ldots, n} := (T_W^{\mathsf{weight}(v^{[i]})}.v^{[i]})_{i=0,\ldots, n}$ toppling-and-permuting equivalent to $v$, and also to $u$.

It remains to show that each of these $n+1$ configurations are distinct. 
To do this, we will study the number of grains at the sink modulo $n+d+1$ for all the $n+1$ configurations. 
First we notice that since toppling the sink removes $n+d$ grains from the sink and any of the $n+d$ topplings of either $T_K$ or $T_I$ returns a grain to the sink, 
all $\left(v_s^{[i]}\mod (n+d+1)\right)_{i=0,\ldots, n}$ are distinct so all configurations $(v^{[i]})_{i=0,\ldots, n}$ are also distinct.
Since the operator $T_W$ contributes $n+d+1$ grains to the sink, the $T_W$-equivalence preserves the number of grains at the sink modulo $(n+d+1)$.
This shows that each $(w^{[i]}_s)_{i=0,\ldots, n}$ is distinct, and hence all $(w^{[i]})_{i=0,\ldots, n}$ are distinct (quasi-stable and non-negative configurations). 

The remainder of the proof shows that there are no other quasi-stable and non-negative configurations that are toppling-and-permuting equivalent to $u$.
If $w$ is a quasi-stable non-negative configuration toppling-and-permuting equivalent to $u$ then it can be written
\[ w = T_s^{\alpha_s}T_K^{\alpha_K}T_I^{\alpha_I}.v, \]
where $v = \mathsf{rec}(u)$ is the sorted recurrent configuration in that equivalence class.

Using Proposition~\ref{lem:identity-operator} we may write:
  \[ T_W = T_K^{|K|+1}T_I^{[I]} = T_KT_s^{-1}\left(T_sT_K^{|K|}T_I^{|I|}\right)=T_KT_s^{-1}.\] 
Thus $T^{s}=T_KT_W^{-1}$ and we can remove the $T_s$ operator in the description of
\[w = T_W^{-\alpha_s}T_K^{\alpha_s+\alpha_K}T_I^{\alpha_I}v.\]

Now let us define $q_K$ and $r_K$ by the Euclidean division of $\alpha_s+\alpha_K$ by $n+1$ : $\alpha_s+\alpha_K = q_K(n+1)+r_K$.
Similarly, define $q_I$ and $r_I$ by the Euclidean division of $\alpha_I$ by $d$ : $\alpha_I = q_Id+r_I$.
This allows the following factorization of operators for powers of $T_W$ then $T_I^{d}$: 
\[ w = T_W^{-\alpha_s}T_K^{q_K(n+1)+r_K}T_I^{q_Id+r_I}.v=T_W^{-\alpha_s}(T_K^{n+1}T_I^d)^{q_K}(T_I^d)^{(q_I-q_K)}\left(T_K^{r_K}T_I^{r_I}\right). v,\]
 where we identify the new factors $T_W=T_K^{n+1}T_I^d$ so we have 
\[ w = T_W^{q_K-\alpha_s}(T_I^d)^{(q_I-q_K)}\left(T_K^{r_K}T_I^{r_I}\right). v.\]
If we restrict the identity to sub-configurations on the component $I$, then $T_W$ keeps it invariant so we may ignore the factor $T_W^{q_K-\alpha_s}$, hence
  \[ w^{[I]} = (T_I^d)^{(q_I-q_K)}\left(T_K^{r_K}T_I^{r_I}\right). v^{[I]}.\]

Since $r_I < |I|$, the value $v_{|I|-1}^{[I]}$ is not decremented by the operator $T_I^{r_I}$ hence
\[ 0 \leq w_{|I|-1-r_I}^{[I]} = v^{[I]}_{|I|-1}+r_K-(n+1)(q_I-q_K) \leq n,\]
where the bounds come from the assumption that $w^{[I]}$ is quasi-stable and non-negative.
As $0 \leq v_{|I|-1}^{[I]} \leq n$ and $0 \leq r_K \leq n$, it follows that $0 \leq v_{|I|-1}^{[I]} + r_K \leq 2n$ and we compare this value to $n$: 
\begin{itemize}
\item if $v^{[I]}_{|I|-1}+r_K\leq n$, in particular when $r_K=0$, then  $(q_I-q_K)=0$;
\item otherwise $v^{[I]}_{|I|-1}+r_K > n$ and $(q_I-q_K)=1$.
\end{itemize}
By contradiction, we want to exclude the cases $(q_I-q_K)=1$ and $r_I > 0$.
In those cases, we have
\[ w^{[I]} =T_K^{r_K}T_I^{d+r_I}.v^{[I]}.\]
Then vertex $v_0^{[I]}$ topples twice while going from $v^{[I]}$ to $w^{[I]}$ so
\[ 0 \leq w_{|I|-r_I} = v_0^{[I]}+r_K-2(n+1) \leq n, \]
where the inequalities come from the assumption that $w^{[I]}$ is quasi-stable and non-negative.
This gives the expected contradiction since $v_0^{[I]}+r_k \leq 2n $ because both summands are lower than $n$.

Hence our analysis shows that
\[ w^{[I]} = T_K^{a}T_I^{b}.v^{[I]}\]
where $a \in \{0,\ldots, n\}$,  $b \in \{0,\ldots, d\}$, and more generally there exists $c\in \mathbb{Z}$ such that
\[ w = T_W^cT_K^aT_I^b.v.\]
\begin{itemize}
\item If $a = 0$ then $w = T_W^cT_I^b.v$ and $w^{[I]} = T_I^b.v^{[I]}$ so $b=0$. Otherwise $b\geq 1$ and the toppling of $v_0^{[I]}$ means that $w_{|I|-b}^{[I]}=v_0^{[I]}-(n+1) < 0$ so $w$ would not be non-negative.
   This gives $ w= T_W^c.v$ and, since  $\mathsf{weight}(w)=c+\mathsf{weight}(v)$ and $\mathsf{weight}(v)=0$, by Proposition~\ref{lem:weight_description_of_quasistables} it follows that $w^{K}$ is quasi-stable and non-negative if and only if $c=0$, and so $w=v$.
\item If $a \geq 1$, we apply once the relation $T_K=T_WT_s$ to obtain $w= T_W^{c+1}T_sT_{K}^{a-1}T_I^b.v$.
   We will identify the possible choices of $a$ and $b$ via the quasi-stability and non-negativity of $w^{[I]} = T_sT_{K}^{a-1}T_I^b.v^{[I]}.$
   In the sequence of operators from Dhar's criterion, \[T_sT_I^{k_0}\prod_{i=1}^n (T_KT_I^{k_i}),\] 
the $n$ prefixes of operators before an occurrence of operator $T_K$ and the appropriate value of $c$ defined by the weight, gives us the previously identified 
$n+1$ quasi-stable and non-negative configurations except the recurrent configuration $v$ that was found in the first case.  
Let $((a_k,b_k))_{k=1,\ldots, n}$ the $n$ possibles values identified via Dhar's criterion.
For other values of $(a,b)$, either $T_K$ operators are missing, in which case $w^{[I]}$ is not non-negative,  or $T_I$ operators are missing and in this case $w^{[I]}$ is not quasi-stable.
More precisely, the decomposition with respect to Dhar's criterion giving priority to toppling vertices in the independent component $I$ leads to $((a_k=k,b_k))_{k=1,\ldots, n}$.
For a possible $(a,b)$ pair, we discuss the value of $b$ with respect to $b_a$:
\begin{itemize}
\item If $b < b_a$, then this means one performs the Dhar criterion but without some expected topplings of unstable vertices in the independent component $I$. Here $w^{[I]}$ will not be (quasi-)stable and such a choice for $b$ is not possible.
\item If $b > b_a$, then this signifies the forcing of a toppling of some unstable vertices in the independent component.  Here $w^{[I]}$ will not be non-negative and such a choice for $b$ is not possible.
\item If $b=b_a$, then we recover one of the $n$ expected quasi-stable and non-negative sub-configurations $w^{[I]}$ of this case and use the weight of $T_sT_K^{a-1}T_I^b.v$ to identify $c+1=\mathsf{weight}(T_sT_K^{a-1}T_I^b.v)$.
This means we obtain the single power of $T_W$ leading to a quasi-stable and non-negative $w^{[K]}$ with the same $w^{[I]}$, and hence a quasi-stable and non-negative configuration. 
\end{itemize}
\end{itemize}
 \end{proof}

\begin{proof}[Proof of Theorem~\ref{ourcyclelemma}]
By Proposition~\ref{lem:counting-quasistable-nonnegative}, we have that every sorted recurrent configuration on $S_{n,d}$ is one of $n+1$ sorted quasi-stable and non-negative configurations that are toppling-and-permuting equivalent to it.
The set of sorted  quasi-stable and non-negative configurations admits a partition into equivalence classes modulo toppling-and-permuting equivalence.
The number of sorted quasi-stable and non-negative configurations on $S_{n,d}$ is 
easily seen to be ${2n+d \choose n}{n+d \choose n}$.
\end{proof}

\section{Conclusion}\label{section:conclusion}
There are several outstanding problems from the work presented in this paper. 
A conjecture that we have been unable to resolve is equality of the $q,t$-CTI and the $q,t$-ITC polynomials. 
While we have been able to establish this for some special cases, the general case remains elusive. 
\begin{conjecture}\label{conj:ITCCTIsymmetry}
$\displaystyle \notrepoly_{n,d}^{CTI}(q,t) = \notrepoly_{n,d}^{ITC}(q,t).$
\end{conjecture}
Since, in Corollary~\ref{corol:itc}, we have established equality of the $q,t$-Schr\"oder and $q,t$-ITC polynomials, Conjecture~\ref{conj:ITCCTIsymmetry} is equivalent to:
\begin{conjecture}\label{thm:schroeder}
$\displaystyle \notrepoly_{n,d}^{CTI}(q,t) = \sum_{w \in \Schroder_{n,d}} q^{\area(w)}t^{\schbounce(w)}$.
\end{conjecture}
We fully expect there to be an explanation for the above correspondences that is realizable in a sandpile setting.
\begin{conjecture}\label{prop:itc-cti}
There exists a bijection $\Psi: ~ \ClassOneParams{n}{d}\mapsto \ClassOneParams{n}{d} $ such that for any configuration $c\in  \ClassOneParams{n}{d}$, 
\begin{equation*}
\displaystyle (\height(c),\topplebounce_{CTI}(c))=(\height(\Psi(c)),\topplebounce_{ITC}(\Psi(c))).
\end{equation*}
\end{conjecture}

\appendix 
\section{A proof of Proposition~\ref{prop:countingFqt}}
\begin{proof}[Proof of Proposition~\ref{prop:countingFqt}]
Theorem~\ref{lem:description-ITCnm} describes the set $\mathsf{ITC}_{n,d}$ of all the possible ITC-toppling sequences 
as being in bijection with certain pairs 
$\left[b=(b_1,\ldots,b_k) , a=(a_1,\ldots,a_{k-1})\right]$ where $b\isweakcomp_k d $ 
and $a \isweakcomp_{k-1} n$.

Lemma~\ref{lem:between-lower-upper} then shows that any ITC-toppling sequence defines 
two extremal Schr\"oder words $w^{\mathsf{lower}}$ and $w^{\mathsf{upper}}$ such that 
all Schr\"oder paths admitting this ITC-toppling sequence are the Schr\"oder words $w'$ 
such that $w^{\mathsf{lower}} \leq w' \leq w^{\mathsf{upper}}$. 
The Schr\"oder words satisfying these inequalities correspond to those Schr\"oder paths 
geometrically enclosed by a sequence of (possibly degenerate) hexagons that are defined 
by the pair $(w^{\mathsf{lower}},w^{\mathsf{upper}})$, 
instead of by the sequences of rectangles as in the usual Dyck word case. 

This can be seen in Figure~\ref{fig:reformulate-bounce}. 
The two extremal Schr\"oder words are illustrated by the two upper and lower paths coloured orange and green, respectively.  They defines a sequence of two hexagons. 
The first hexagon starts at (0,2) and ends at (4,6). 
It consists of two (0,1) steps, two (1,0) steps, and two (1,1) steps. 
The second hexagon starts at (4,7) and ends at (9,9) and consists of two (1,1) steps and three (1,0) steps. 

Finally, Lemma~\ref{lem:q-binomial-in-hexagon} shows that the set of Schr\"oder words contained between two extremal Schr\"oder words
may be obtained from $w^{\mathsf{lower}}$ whose minimal (in $q$) weight is $\prod_{i=1}^{k} q^{{a_i\choose 2}}t^{(i-1)(b_i+a_i)}$ by 
commutation of steps inside the same hexagon.
Since, inside each hexagon having side lengths $a_{i-1}-1$,$b_{i}$ and $a_{i}$ (for a visual aid see see Figure~\ref{fig:reformulate-bounce}), 
every commutation is possible within each hexagon, 
it appears that the sum of all possible commutations leads to the $q$-multinomial coefficient 
${a_{i-1}-1+b_i+a_i \choose  a_{i-1}-1, b_i, a_i }_q$ for this hexagon.
Since two commutations in distinct hexagons are independent we obtain the expected formula.
\end{proof}

Every ITC-toppling sequence $((b_1,\ldots,b_k), (a_1,\ldots,a_k))$ defines two extremal Schr\"oder paths:
the \emph{lower \textup{(}Schr\"oder\textup{)} path} 
\begin{equation*}
w^{\mathsf{lower}} := \mu\left(H^{b_1}U^{a_1}\left(\prod_{i=2}^{k}D^{a_{i-1}}H^{b_i}U^{a_i}\right)\right),
\end{equation*}
and
the \emph{upper \textup{(}Schr\"oder\textup{)} path} 
\begin{equation*}
w^{\mathsf{upper}} := \mu\left(U^{a_1}H^{b_1}\left(\prod_{i=2}^k DU^{a_i}H^{b_i}D^{a_{i-1}-1}\right)\right).
\end{equation*}
Recall that $\mu$ is the word-reversal operator and for each of the extremal paths the non-commutative product of words in each is given in increasing order of the indices. 
In Figure~\ref{fig:reformulate-bounce}, the lower path $w^{\mathsf{lower}}$ is illustrated in green, sometimes hiding the upper path $w^{\mathsf{upper}}$ illustrated in orange.

\begin{lemma}\label{lem:between-lower-upper}
Let $w\in \Schroder_{n,d}$.
The Schr\"oder paths $w'$ for which $w^{\mathsf{lower}} \leq w' \leq w^{\mathsf{upper}}$ 
are precisely those Schr\"oder paths $w'$ for which $\bouncepath_{ITC}(\phi(\mu(w'))) = \bouncepath_{ITC}(\phi(\mu(w)))$.
\end{lemma}

\begin{proof}
Let $w\in \Schroder_{n,d}$ and $c=\phi(\mu(w)) \in \ClassOneParams{n}{d}$ be its corresponding recurrent configuration.
Suppose that $\bouncepath_{ITC}(c)= (b_1,a_1,\ldots,b_k,a_k)$ so that the corresponding ITC-toppling sequence is
\begin{equation*}
ITC(w)=\left[(b_1,\ldots,b_k),(a_1,\ldots,a_k)\right].
\end{equation*}
The recurrent configuration that corresponds to $w^{\mathsf{lower}}$ is the configuration stated in the proof of Theorem~\ref{lem:description-ITCnm}.
Moreover, this recurrent configuration for $w^{\mathsf{lower}}$ also has ITC-toppling sequence 
$\left[(b_1,\ldots,b_k),(a_1,\ldots,a_k)\right]$. 

For $w^{\mathsf{upper}}$ the situation is similar.
This is seen by considering the rewriting rule 
$D^{\alpha-1}H^{\beta}U^{\gamma}\longrightarrow U^{\gamma}H^{\beta}D^{\alpha-1}$ 
in every hexagon (between the paths $w^{\mathsf{lower}}$ and $w^{\mathsf{upper}}$) that corresponds to vertices toppled during one loop iteration $(b_i,a_i)$. 
Compare to the $w^{\mathsf{lower}}$ case, the rewriting rule for vertices toppled during the $k$-th loop iteration leads to 
$\alpha=a_{k-1}$,$\beta=b_k$, $\gamma=a_k$. 
This means every clique vertex toppled during this loop has $(a_k-1)+b_k$ grains more and every independent vertex has $a_k-1$ grains more.
Hence $w^{\mathsf{upper}}$ corresponds to the configuration $c^{\mathsf{upper}}$ as follows:
\begin{align*}
c^{\mathsf{upper}}(v_i) :=& n+d - \left(\sum_{j=0}^{k-2} a_j+b_j\right) - b_{k-1}-1, 
\end{align*}
for all $n-(a_0+\ldots +a_k) \leq i < n-(a_0+\ldots a_{k-1})$ and $1\leq k\leq t$, and
\begin{align*}
c^{\mathsf{upper}}(w_i) :=& n+1 - (a_0+a_1+\ldots+a_{k-2}) -1,
\end{align*}
for all $d-(b_0+\ldots+b_k) \leq i < d-(b_0+\ldots + b_{k-1})$ and $1\leq k \leq t$.
Recall that $b_0:=0$ and $a_0:=1$.
From this it is immediate that $ITC(w^{\mathsf{upper}}) = ITC(w^{\mathsf{lower}}) = ITC(w)$.

Any $w'$ such that $w^{\mathsf{lower}} \leq w' \leq w^{\mathsf{upper}}$ must have $ITC(w') = ITC(w)$ since stability of vertices 
is preserved for $ w' \leq w^{\mathsf{upper}}$ while instability is guaranteed for $w^{\mathsf{lower}} \leq w'$.
For any $w'$ that does not satisfy $w^{\mathsf{lower}} \leq w' \leq w^{\mathsf{upper}}$, 
we may consider the last difference which is also the first encounter by the ITC-toppling sequence.
If $w^{\mathsf{lower}} \nleq w$ then, by inspection of the case of $w^{\mathsf{lower}}$ with at least one less grain, there exists one vertex that is not unstable when it ought to be due to a lack of grains.
Likewise, if $w \nleq w^{\mathsf{upper}}$ then, by inspection of the case of $w^{\mathsf{upper}}$ with at least one more grain, there exists one
vertex that is unstable during the previous loop iteration due to an excess of grains.

Since the ITC toppling sequences coincide precisely when $w^{\mathsf{lower}} \leq w' \leq w^{\mathsf{upper}}$, we may conclude 
that this is also precisely when $\bouncepath_{ITC}(\phi(\mu(w'))) = \bouncepath_{ITC}(\phi(\mu(w)))$.
\end{proof}

Between their points of intersection, the lower and upper paths form hexagons having parallel sides, and these opposite sides have side-lengths $a_{i-1}-1$, $b_i$, $a_i$.
Such hexagons become rectangles when $b_i=0$, parallelograms when $a_{i-1}=1$, and horizontal lines when both $b_i=0$ and $a_{i-1}=1$.
In what follows, recall that the shuffle operator $\shuffle$ of two sequences is the set of all sequences formed from entries in both, and for which all entries appear in their original order.
E.g. 
\begin{equation*}
a_1a_2 \shuffle b_1b_2 = \{ a_1a_2b_1b_2, ~ a_1b_1a_2b_2, ~ a_1b_1b_2a_2, ~ b_1a_1a_2b_2, ~ b_1a_1b_2a_2, ~b_1b_2a_1a_2\}.
\end{equation*}

\begin{lemma}\label{lem:q-binomial-in-hexagon}
The generating function for a factor in a Schr\"oder path enclosed within a hexagon $\Hex$ defined by 
$h^{\mathsf{lower}}=D^aH^bU^c $ and $h^{\mathsf{upper}}=U^cH^bD^a$ 
according to the area enclosed in the same hexagon is given by the q-multinomial coefficient
\begin{equation*}
\sum_{w \in D^a\shuffle H^b \shuffle U^c} q^{\mathsf{area}_{\Hex}(w)} = { a+b+c \choose a,b,c }_q.
\end{equation*}
\end{lemma}

\begin{proof}
We notice by inspecting the diagrams below that each oriented 
commutation $XY\longrightarrow YX$ that generates all possible 
words in the (multi-)shuffle $D^a\shuffle H^b \shuffle U^c$ adds one enclosed lower half triangle.
  
\begin{center}
  \begin{minipage}{4cm}
    \begin{tikzpicture}[scale=0.5]
      \draw (-0.5,-0.5) grid (2.5,1.5);
      \draw[green!50!brown,line width=2] (0,0) -- (1,0) -- (2,1);
      \draw[orange,line width=2] (0,0) -- (1,1) -- (2,1);
      \fill[black,rounded corners] (0.9,0.1) -- (0.9,0.8) -- (0.2,0.1) -- cycle ;
      \draw node at (1,-1) {$\textcolor{green!50!brown}{DH} \longrightarrow \textcolor{orange}{HD}$};
    \end{tikzpicture}
  \end{minipage}
  \begin{minipage}{4cm}
    \begin{tikzpicture}[scale=0.5]
      \draw (-0.5,-0.5) grid (1.5,2.5);
      \draw[green!50!brown,line width=2] (0,0) -- (1,1) -- (1,2);
      \draw[orange,line width=2] (0,0) -- (0,1) -- (1,2);
      \fill[black,rounded corners] (0.9,1.1) -- (0.9,1.8) -- (0.2,1.1) -- cycle ;
      \draw node at (1,-1) {$\textcolor{green!50!brown}{HU} \longrightarrow \textcolor{orange}{UH}$};
    \end{tikzpicture}
  \end{minipage}
  \begin{minipage}{4cm}
    \begin{tikzpicture}[scale=0.5]
      \draw (-0.5,-0.5) grid (1.5,1.5);
      \draw[green!50!brown,line width=2] (0,0) -- (1,0) -- (1,1);
      \draw[orange,line width=2] (0,0) -- (0,1) -- (1,1);
      \fill[black,rounded corners] (0.9,0.1) -- (0.9,0.8) -- (0.2,0.1) -- cycle ;
      \draw node at (1,-1) {$\textcolor{green!50!brown}{DU} \longrightarrow \textcolor{orange}{UD}$};
    \end{tikzpicture}
  \end{minipage}
\end{center}
Since $\mathsf{area}_{\Hex}(w^{\mathsf{lower}}=D^aH^bU^c)=0$, we have a combinatorial interpretation of the $q$-multinomial coefficient.
The generating function for the area statistic over all hexagons corresponding to words consisting of $a$ $D$'s, $b$ $H$'s and $c$ $U$'s 
(in terms of the lower triangles between it and the path $D^aH^bU^c$) is seen to correspond to the number of inversions of such a word 
under the natural order $D\prec H\prec U$. 
(Recall that an inversion in the word $x_1x_2\ldots x_m$ is a pair $(i,j)$ with $i<j$ and $x_i \succ x_j$.)
Therefore, by Foata and Han~\cite[Theorem 5.1]{foata}, this is the $q$-multinomial coefficient ${ a+b+c \choose a,b,c }_q$. 
(This is also similar to Egge et al.~\cite[Lemma 1]{ehkk}.)
\end{proof}

\section{Number of sequence-pairs in exact sums for the $q,t$-Schr\"oder polynomials}
Let $c_k(m)$ be the number of compositions of the integer $m$ into $k$ strictly positive parts. 
Let $w_k(m)$ be the number of compositions of the integer $m$ into $k$ non-negative parts. 
A trivial counting argument shows that $c_k(m) = \tbinom{m-1}{k-1}$ and $w_k(m) = \tbinom{m+k-1}{m}$.

\begin{lemma}\label{itc:sequence:enumeration}
Let $\itc_{n,d,k}$ be the number of pairs of sequences $(a,b)$ in the set $\ITC_{n,d,k}$, and over which we sum in Equation~\ref{derycke:equation}.
We have $\itc_{n,d,1}=1$ and for $k\geq 2$, 
\begin{itemize}
\item $\itc_{n,d,k} =\displaystyle \binom{d+k-2}{d-1} \binom{n-1}{k-2} + \binom{d+k-1}{d} \binom{n-1}{k-1}$, and
\item $\itc_{n,d} =\displaystyle \sum_{k=1}^{n} \binom{d+k}{d} \binom{n-1}{k-1}$.
\end{itemize}
\end{lemma}

\begin{proof}
The set $\ITC_{n,d,k}$ is characterised in Theorem~\ref{lem:description-ITCnm}. We have
$\ITC_{n,d,1} = \{\left[(d),(n)\right]\}$ (and so $\itc_{n,d,1}=1$) while for $k\geq 2$:
\begin{equation*}
\ITC_{n,d,k} = 
    \left\{
    \left[(b_1,\ldots,b_k), (a_1,\ldots,a_{k})\right]
        ~:~
    \begin{array}{l}
        (b_1,\ldots,b_k) \isweakcomp d,\\
        (a_1,\ldots,a_{k}) \isweakcomp n,\\
        a_1,\ldots,a_{k-1}>0, \mbox{ and }\\
        b_k+a_k>0
    \end{array}
    \right\}.
\end{equation*}
We can condition now on whether $a_k$ is zero or positive to see that 
$\ITC_{n,d,k}$ can be partitioned into two sets and simplified as
\begin{enumerate}
\item[(i)] Those $[(a_1,\ldots,a_{k-1},0),~(b_1,\ldots,b_k)]$ wherein $(a_1,\ldots,a_{k-1}) \iscomp n$ and $(b_1,\ldots,b_k) \isweakcomp d$ with $b_k>0$. 
This is equivalent to  $(a_1,\ldots,a_{k-1}) \iscomp n$ and $(b_1,\ldots,b_k-1) \isweakcomp d-1$, 
of which there are $c_{k-1}(n) w_k(d-1) = \tbinom{n-1}{k-2} \tbinom{d-1+k-1}{d-1}$ many.
\item[(ii)] Those $[(a_1,\ldots,a_{k-1},a_k),~(b_1,\ldots,b_k)]$ wherein $(a_1,\ldots,a_{k}) \iscomp n$ and $(b_1,\ldots,b_k) \isweakcomp d$.
There are $c_k(n) w_k(d) = \tbinom{n-1}{k-1} \tbinom{d+k-1}{d}$ many of these.
\end{enumerate}
Using these expressions we have
\begin{align*}
\itc_{n,d} 
	= & 1+\sum_{k\geq 2}  \binom{d+k-2}{d-1} \binom{n-1}{k-2} + \binom{d+k-1}{d} \binom{n-1}{k-1}\\
	= & 1+\sum_{k=2}^{n+1} \binom{d+k-2}{d-1} \binom{n-1}{k-2} + \sum_{k=2}^n \binom{d+k-1}{d} \binom{n-1}{k-1}\\
	= & 1+\sum_{k=1}^{n} \binom{d+k-1}{d-1} \binom{n-1}{k-1} + \sum_{k=2}^n \binom{d+k-1}{d} \binom{n-1}{k-1}\\
	= & 1+ d+ \sum_{k=2}^{n} \binom{d+k-1}{d-1} \binom{n-1}{k-1} + \sum_{k=2}^n \binom{d+k-1}{d} \binom{n-1}{k-1}\\
	= & 1+ d+ \sum_{k=2}^{n} \binom{d+k}{d} \binom{n-1}{k-1}\\
	= & \sum_{k=1}^{n} \binom{d+k}{d} \binom{n-1}{k-1}.\qedhere
\end{align*}
\end{proof}

\begin{lemma}\label{ehkk:sequence:enumeration}
Let $\ehkk_{n,d,k}$ be the number of pairs of sequences $(\alpha,\beta)$ over which we sum in Equation~\ref{egge:equation}.
Then 
\begin{equation*}
\ehkk_{n,d,k} = \binom{d+k}{d} \binom{n-1}{k-1} \mbox{ and } 
\ehkk_{n,d} = \sum_{k=1}^n \binom{d+k}{d} \binom{n-1}{k-1}.
\end{equation*}
\end{lemma}

\begin{proof}
The number $\ehkk_{n,d,k}$ is the number of pairs of sequences $(\alpha,\beta)$ where
$(\alpha_1,\ldots,\alpha_k)\iscomp n$ and $(\beta_0,\ldots,\beta_k) \isweakcomp d$.
Thus $\ehkk_{n,d,k} = c_k(n) w_{k+1}(d) = \tbinom{n-1}{k-1} \tbinom{d+k}{d}$ 
and 
\begin{equation*}
\ehkk_{n,d} = \sum_{k=1}^n \binom{d+k}{d} \binom{n-1}{k-1}.\qedhere
\end{equation*}
\end{proof}

\section{A proof the Schr\"oder polynomials are symmetric in $q$ and $t$}
\begin{proposition} \label{qt:symmetry}
Let $S_{n,d}(q,t)$ the generating function of $q,t$-Schr\"oder paths.
\[ S_{n,d}(q,t) = S_{n,d}(t,q).\]
\end{proposition}

\begin{proof}
We recall here the combinatorial interpretation, given in Haglund's monograph~\cite[Equation 4.12]{haglund:monograph},
of 
 \[ \sum_{d=0}^n z^dS_{n,d}(q,t) = \sum_{\mu \vdash n} W(\mu;q,t) \]
 as a sum over partitions of an expression $W(\mu;q,t)$, to be defined below.
We will show below that the natural involution $\mu \rightarrow \mu'$ acts on $W(\mu;q,t)$ in such a way that $W(\mu;q,t)=W(\mu';t,q)$, and from which it follows that $S_{n,d}(q,t) = S_{n,d}(t,q)$.
As the expression for $W(\mu;q,t)$ is rather involved, we also recall some notions on partitions that are illustrated in Figure~\ref{fig:partition} and introduced in Haglund~\cite[Equations 2.8 and 2.9]{haglund:monograph}.
  \begin{figure}[ht!]
    \begin{center}
  \begin{tikzpicture}[scale=0.6]
    \foreach \x/\y in {0/10,1/10,2/10,3/7,4/6,5/5,6/2}{
      \draw node at (11,6-\x+0.5) {$\y$};
      \foreach \z in {1,...,\y}{
        \draw (\z-1,7-\x) rectangle (\z,6-\x);
        }
    }
    \draw[line width=2,opacity=0.3] (0,7) -- (7,0);
    \foreach \x/\y in {0/7,1/7,2/6,3/6,4/6,5/5,6/4,7/3,8/3,9/3}{
      \draw node at (\x+0.5,-1) {$\y$};
    }
    \draw node[scale=0.7] at (0.5,6.5) {$(0,0)$};
    \draw node[scale=0.7] at (4.5,4.5) {$(i,j)$};
    \draw node[scale=0.7] at (-0.5,4.5) {$i$};
    \draw node[scale=0.7] at (4.5,7.5) {$j$};
    \fill[blue!50!white,opacity=0.9] (0,4) rectangle (4,5);
    \draw node[scale=0.7] at (2,4.5) {$a'(x)$};
    \fill[green!50!white,opacity=0.9] (5,4) rectangle (10,5);
    \draw node[scale=0.7] at (7.5,4.5) {$a(x)$};
    \fill[red!50!white,opacity=0.9] (4,5) rectangle (5,7);
    \draw node[scale=0.7] at (4.5,6) {$\ell'(x)$};
    \fill[orange!50!white,opacity=0.9] (4,4) rectangle (5,1);
    \draw node[scale=0.7] at (4.5,2.5) {$\ell(x)$};
  \end{tikzpicture}
  \caption{Ferrers diagram of partition $\mu:=(10,10,10,7,6,5,2)$ and the conjugate symmetry leading to $\mu'=(7,7,6,6,6,5,4,3,3,3)$. The diagram illustrates the arm, coarm, leg, and coleg of the cell $x=(i,j)=(2,4)$\label{fig:partition}}
    \end{center}
  \end{figure}
  A partition $\mu$ may be represented by a Ferrers diagram where the row-lengths from top to bottom correspond to the parts of the partition.
  The main diagonal in grey defines a line of symmetry of the Ferrers diagram the reflection through which we get the conjugate partition $\mu'$.
  A cell $x=(i,j)$ in this diagram is indexed by its row $i$ and its column $j$.
  The coarm of cell $x$ is $a'(x) := i$ and the arm of $x$ is $a(x) := \mu_i-1-i$.
  The coleg of $x$ is $\ell'(x) := j$ while the leg of $x$ is $\ell(x) := \mu'_j-1-j$.
  Notice that the cell $x'=(j,i)$ maps to the cell $x=(i,j)$ with respect to the reflection (which is an involution) in the diagonal.
  Moreover the (co-)arms and (co-)legs are exchanged by the involution, so that
\begin{align*}
\begin{array}{rclcrcl}
a_\mu(x) 	&=&\ell_{\mu'}(x'), & & \ell_\mu(x) 	&=& a_{\mu'}(x') ,\\
a'_{\mu}(x) &=& \ell'_{\mu'}(x'), && \ell'_\mu(x) &=& a'_{\mu'}(x'),
\end{array}
\end{align*}
where we use subscripts to indicate the diagram on which the values are measured.
The weight associated with a partition $\mu$ is defined as 
  \[ W(\mu;q,t) :=  \frac{T_\mu\left(\displaystyle\prod_{x\in \mu} \left(z+q^{a'(x)}t^{\ell'(x)}\right)\right) M \Pi_\mu B_\mu}{w_\mu} , \]
  where 
\begin{align*}
M 		& := (1-q)(1-t),\\
B_\mu 	&:= \sum_{x\in \mu}q^{a'(x)}t^{\ell'(x)},\\
w_\mu 	&:= \prod_{x\in \mu} \left(q^{a(x)}-t^{\ell(x)+1}\right)\left(t^{\ell(x)}-q^{a(x)+1}\right),\\
\Pi_\mu	&:= \prod_{x\in \mu-\{(0,0)\}}  \left(1-q^{a'(x)}t^{\ell'(x)}\right).
\end{align*} 
In the product for the expression $\Pi_{\mu}$ we ignore the corner cell $x=(0,0)$ for which the term is $0$. 
In addition to these set 
$T_{\mu} := t^{n(\mu)}q^{n(\mu')}$,
\begin{align*}
n(\mu) := \sum_{x\in \mu} \ell'(x) = \sum_{x\in \mu} \ell(x), \mbox{ and }
n(\mu') :=\sum_{x\in \mu} a'(x) = \sum_{x\in \mu} a(x).
\end{align*}
Using the exchange properties of parameters for each term that appears in the product form for $W$, we may conclude that $W(\mu';q,t)=W(\mu;t,q)$.
It follows that 
\begin{align*}
\sum_{d=0}^n z^dS_{n,d}(q,t) &= \sum_{\mu \vdash n} W(\mu;q,t) \\
& = \frac{1}{2} \sum_{\mu \vdash n} W(\mu;q,t) + \frac{1}{2}\sum_{\mu'\vdash n} W(\mu';q,t)\\
& = \sum_{\mu \vdash n} \frac{1}{2}\left(W(\mu;q,t)+W(\mu';q,t)\right)\\
& = \sum_{\mu \vdash n} \frac{1}{2}\left(W(\mu;q,t)+W(\mu;t,q)\right)
\end{align*}
is symmetric in $q$ and $t$ as every term in the sum is symmetric in $q$ and $t$.
\end{proof}


\begin{thebibliography}{99}
\bibitem{aadhl} Jean-Christophe Aval, Michele D'Adderio, Mark Dukes, Angela Hicks, and Yvan Le Borgne. Statistics on parallelogram polyominoes and a $q,t$-analogue of the Narayana numbers. {\em{Journal of Combinatorial Theory Series A}} {\bf{123}} (2014), no. 1, 271--286. \doi{10.1016/j.jcta.2013.09.001}
\bibitem{aadl} Jean-Christophe Aval, Michele D'Adderio, Mark Dukes, and Yvan Le Borgne. Two operators on sandpile configurations, the sandpile model on the complete bipartite graph, and a Cyclic Lemma.  {\em{Advances in Applied Mathematics}} {\bf{73}} (2016) 59--98. \doi{10.1016/j.aam.2015.09.018}
\bibitem{bergeron-book} Fran\c{c}ois Bergeron. {\em Algebraic Combinatorics and Coinvariant Spaces.} CMS Treatise in Mathematics, CMS and A.K.Peters, 2009.
\bibitem{bij1} Olivier Bernardi. Tutte polynomial, subgraphs, orientations and sandpile model: new connections via embeddings. {\em Electronic Journal of Combinatorics} {\bf 15} (2008), no. 1, P109.
\bibitem{bij2} Denis Chebikin and  Pavlo Pylyavskyy. A family of bijections between $G$-parking functions and spanning trees. {\em Journal of Combinatorial Theory Series  A} {\bf 110} (2005), no 1, 31--41.
\bibitem{cori2002} Robert Cori and Dominique Poulalhon. Enumeration of $(p,q)$-parking functions. {\em Discrete Mathematics} {\bf 256} (2002), 609--623.
\bibitem{cori2000} Robert Cori and Dominique Rossin. On the sandpile group of dual graphs. {\em European Journal of Combinatorics} {\bf 21} (2000), 447--459.
\bibitem{bij3} Robert Cori and Yvan Le Borgne. The sandpile model and Tutte polynomials. {\em Advances in Applied Mathematics} {\bf 30} (2003), nos. 1--2, 44--52.
\bibitem{ddillw} Michele D'Adderio, Mark Dukes, Alessandro Iraci, Alexander Lazar, Yvan Le Borgne, and Anna Vanden Wyngaerd.  Shuffle theorems and sandpiles.  {\em Communications in Mathematical Physics} {\bf 406} (2025), Article 83. \doi{10.1007/s00220-025-05233-5}
\bibitem{imrn} Michele D'Adderio, Alessandro Iraci, Yvan Le Borgne, Marino Romero, and Anna Vanden Wyngaerd. Tiered trees and Theta operators. {\em International Mathematics Research Notices} {\bf{2023}} (2023), no. 24, 20748--20783. \doi{10.1093/imrn/rnac258}
\bibitem{mylb} Michele D'Adderio and Yvan Le Borgne. The sandpile model on $K_{m,n}$ and the rank of its configurations. {\em S\'eminaire Lotharingien de Combinatoire} {\bf{77}} (2018), Art. B77h.
\bibitem{Dhar} Deepak Dhar.  Theoretical studies of self-organized criticality.  {\em Physica A: Statistical Mechanics and its Applications} {\bf 369} (2006), no. 1, 29--70. \doi{10.1016/j.physa.2006.04.004}
\bibitem{ncf} Mark Dukes. The sandpile model on the complete split graph, Motzkin paths, and tiered parking functions. {\em Journal of Combinatorial Theory Series A} {\bf{180}} (2021) Article 105418. \doi{10.1016/j.jcta.2021.105418}
\bibitem{dlb} Mark Dukes and Yvan Le Borgne. Parallelogram polyominoes, the sandpile model on a complete bipartite graph, and a $q,t$-Narayana polynomial. {\em{Journal of Combinatorial Theory Series A}} {\bf{120}} (2013), no. 4, 816--842. \doi{10.1016/j.jcta.2013.01.004}
\bibitem{dsss} Mark Dukes, Thomas Selig, Jason P. Smith, and Einar Steingr\'imsson. Permutation graphs and the Abelian sandpile model, tiered trees and non-ambiguous binary trees.  {\em Electronic Journal of Combinatorics} {\bf 26} (2019), no. 3, P3.29.
\bibitem{ferrers} Mark Dukes, Thomas Selig, Jason P. Smith, and Einar Steingr\'imsson. The Abelian sandpile model on Ferrers graphs -- A classification of recurrent configurations. {\em European Journal of Combinatorics} {\bf 81} (2019) 221--241.
\bibitem{dvor} A. Dvoretzky and Th. Motzkin. A problem of arrangements. {\em Duke Mathematical Journal} {\bf 14} (1947), no. 2, 305--313. \doi{10.1215/S0012-7094-47-01423-3}.
\bibitem{ehkk} Erik Egge, James Haglund, Kendra Killpatrick, and Darla Kremer. A Schr\"oder generalization of Haglund's statistic on Catalan paths. {\em Electronic Journal of Combinatorics}, 10, 2002. paper \#R16
\bibitem{foata} Dominique Foata and Guo-Niu Han. {\em  $q$-series in Combinatorics; Permutation Statistics (Preliminary version)}. Preprint, March 2021. 
\bibitem{haglund:monograph} James Haglund. \textit{The $q,t$-Catalan numbers and the space of Diagonal Harmonics}.  AMS University Lecture Series, 2008.
\bibitem{haglund-qtschroder} James Haglund.  A proof of the $q,t$-Schr\"oder Conjecture. {\em International Mathematics Research Notices} {\bf{2004}} (2004), no. 11, 525--560. \doi{10.1155/S1073792804132509}
\bibitem{ps} Alexander Postnikov and Boris Shapiro. Trees, parking functions, syzygies, and deformations of monomial ideals.  {\em Transactions of the American Mathematical Society} {\bf  356} (2004), no 8, 3109--3142.
\end{thebibliography}
\end{document}